\documentclass{amsart}
\usepackage{epsfig}
\usepackage{tikz}
\usepackage{enumerate}
\usepackage{comment}
\usepackage{rotating}

\usepackage{latexsym}
\usepackage{amssymb}
\usepackage{amsmath}
\usepackage{amsthm}
\usepackage{amscd}
\usepackage{latexsym}
\usepackage{amssymb}
\usepackage{amsmath}
\usepackage{amsthm}
\usepackage{amscd}
\theoremstyle{plain}
\newtheorem{thm}{Theorem}[section]
\newtheorem{lem}[thm]{Lemma}
\newtheorem{defin}[thm]{Defintion}
\newtheorem{prop}[thm]{Proposition}
\newtheorem{cor}[thm]{Corollary}
\theoremstyle{definition}
\newtheorem{rem}[thm]{Remark}

\newcommand{\future}[1]{{}}  

\usepackage[margin=1.5in]{geometry}
\newcommand{\bb}{\mathbb}

\newcommand{\R}{\bb R}
\newcommand{\N}{\bb N}

\newcommand{\blas}{{BLAS} }
\newcommand{\pml}{\ensuremath{\mathcal{PML}} }
\newcommand{\pmls}{\ensuremath{\mathcal{PML}}s }
\newcommand{\bout}{\ensuremath{\overline{CV_n}}}

\newcommand{\F}{\mathcal F}

\newcommand{\dist}{{\rm dist} }

\def\<{{\langle}}
\def\>{{\rangle}}

\newcommand{\rank}{18} %% Case 3 needs 4 auxiliary letters, each with spacing

\begin{document}
\title[Connectivity of $\partial FF_n$]{Connectivity of the Gromov Boundary of the Free Factor Complex}
\author{Mladen Bestvina, Jon Chaika and Sebastian Hensel}
\begin{abstract}
  We show that % in rank at least $\rank$, 
  in large enough rank, the Gromov boundary of the free
  factor complex is path connected and locally path connected. 
\end{abstract}
\maketitle

\section{Introduction}

A prevailing theme in geometric group theory is to study groups using
actions on suitable Gromov hyperbolic spaces. One of the most
successful examples of this philosophy is the action of the mapping
class group of a surface on the curve graph, which is hyperbolic by
the seminal work of Masur and Minsky \cite{MM1}, and which has been
used to give a hierarchical description of the geometry of mapping
class groups.

One core tool when studying Gromov hyperbolic spaces is that they
admit natural boundaries at infinity. In the case of the curve graph,
this boundary can be explicitly described in terms of topological
objects. Namely, Klarreich \cite{Kla} proved that the boundary is the
space of ending laminations. That is, the boundary can be obtained
from the sphere of projective measured laminations by removing all
non-minimal laminations, and then identifying laminations with the
same support.  Although the Gromov boundaries of curve graphs are
fairly complicated topological spaces (in the case of punctured
spheres, they are N\"obeling spaces \cite{GabTop}), the connection to
laminations can be used to effectively study them.  Maybe most
relevant for our current work, Gabai \cite{GabAF, GabTop} used this
connection to show that the boundary is path-connected and locally
path-connected.

\smallskip In the setting of outer automorphism groups of free groups,
there are several possible analogs of the curve graph. In this
article, we focus on the \emph{free factor complex} $FF_n$, which is
hyperbolic by a result of Bestvina-Feighn \cite{BesFei}. Similar to
Klarreich's theorem, the Gromov boundary has been identified as the
space of arational trees modulo a suitable equivalence
\cite{BesRey}. The role of the sphere of projective measured
laminations is played by the boundary of Culler-Vogtmann's Outer space
(which has much more complicated local topology than a sphere).

In this article, we nevertheless begin a study of connectivity
properties of the boundary at infinity of $FF_n$.  More specifically,
we show
\begin{thm}\label{thm:main}
  The Gromov boundary $\partial FF_n$ of the free factor complex is
  path connected and locally path connected for $n \geq \rank$.
\end{thm}
As an immediate consequence we obtain the following coarse geometric
property:
\begin{cor} 
  The free factor complex $FF_n$ is one-ended for $n\geq \rank$.
\end{cor}
From this, one-endedness of various other combinatorial complexes (the
free splitting complex, the cyclic splitting complex, and the maximal
cyclic splitting complex) can also be concluded
(Corollary~\ref{cor:other-complexes}). We also emphasise that we do 
not claim that the constant $\rank$ is optimal, and in fact expect that 
the result holds for much lower $n$ (see below).

\subsection{Outline of proof}
Our strategy is motivated by the case of $\pml$ and ending lamination
space, though it requires new ideas. To link the case of $\pml$ to
$ \partial CV_n$, there is a dense connected set of copies of $\pml$
in $\partial CV_n$ coming from identifying the fundamental group of a
surface with one boundary component with $F_n$.  For simplicity, lets
start `upstairs' in $\partial CV_n$ and assuming that two arational
trees, $T,T'$ lie on two different $\pml$s.

We have a chain of copies of $\pml$ that connects the $\pml$
containing $T$ to the one containing $T'$ and where each consecutive
pair in the chain intersects in the $\pml$ of a subsurface. %
Using work from surface case, in particular \cite{CH}, which builds on
\cite{LS}, we build a path, $p_0$ across this chain of $\pml$s so that
every tree in it is arational except at the intersection of the
consecutive $\pml$s and at these it is the stable lamination of a
(partial) Pseudo-Anosov supported on the subsurface. We wish to have a
path entirely of arational trees, so we find a countable sequence
of compact sets $K_1,K_2,...$ so that the arational trees are the
complement of $\cup K_j$. 
We iteratively improve our path $p_j$ which
by inductive hypothesis
\begin{itemize}
\item avoids $K_1,...,K_j$ entirely
\item has that every foliation on it is minimal, except for a
  finite number of points
\item these points are $\lambda_\psi$, the stable laminations of
  (partial) Pseudo-Anosovs, $\psi$, supported on a subsurface.
\end{itemize} 
to a path $p_{j+1}$ which avoids $K_1,...,K_j,K_{j+1}$ entirely and
has the same structure as above. Most of the work of this paper is in
avoiding $K_{j+1}$. (The fact that $p_{j+1}$ avoids $K_1,...,K_j$ is a
consequence of the fact that we can make $p_{j+1}$ as close as we want
to $p_j$.)

We now discuss avoiding the $K_{j+1}$. First off, what are the
$K_{j+1}$? They are the unions of the trees in $\partial CV_n$ where a
fixed proper free factor $E_{j+1}$ is not both free and discrete. To rule
this out, it suffices to show that our stable laminations of the
partial Pseudo-Anosov is supported on a subsurface whose fundamental
group is not contained in $E_{j+1}$ and that $E_{j+1}$ trivially intersects
the fundamental group of the complement of the support of partial
Pseudo-Anosov. See Lemma \ref{lem:leaving the K with pAs}. (It is
automatic that the rest of our path avoids $K_j$ because these points
are already arational.) In Section \ref{sec:avoid} we describe how if
one of our stable laminations $\lambda_\psi$ is in $K_{j+1}$ we can
find a chain of $\pml$s that build a detour around it and moreover,
the image of this path under a large power of $\psi$ also avoids
$K_{j+1}$. This involves explicitly constructing relations in
$\mathrm{Out}(F_n)$, and is technically the most involved part of the
paper.

In this way we can improve our path $p_j$ on a small
segment around $\lambda_\psi$ to obtain $p_{j+1}$, which now avoids
$K_{j+1}$ and where the contraction properties of $\psi$ guarantee
that it still avoids $K_1,...,K_j$.  Our sequence of paths
$p_0,p_1,...$ is a Cauchy sequence and so we obtain in the limit
$p_\infty$ which is in $(\bigcup K_j)^c$.  This describes an argument
that the projection of the set of arational surface type trees to the
$\partial FF_n$ is path connected. To upgrade this to showing that
$\partial FF_n$ is path connected, we show (basically via Proposition
\ref{prop:chain} which uses folding paths) that we can choose
sequences of surface type arational trees converging to our fixed (not
necessarily surface type) tree so that the paths stay in a $\delta$
neighborhood of our arational tree.  In the outline we ignored some
subtleties, most notably how to construct the required relations in $\mathrm{Out}(F_n)$,
and that our ``paths" in $\partial CV_n$ are not necessarily
continuous at the end points, though the projection of these ``paths"
to $\partial FF_n$ will be. This last point exploits contraction dynamics 
on the boundary of the hyperbolic space  $\partial FF_n$.

We now briefly state the structure of the paper. Section 2 collects
known results about $\pml$ and $\mathrm{Out}(F_n)$ and modifies them for our
purposes. Section 3 relates $\partial CV_n$ and $\mathrm{Out}(F_n)$. Section 4
is the technical heart of our paper, describing how to locally avoid a
fixed $K_j$. Section 5 proves the main theorem. Section 6 uses the
main theorem to establish one-endedness of some other combinatorial
complexes. There are three appendices that treat issues related to
non-orientable surfaces, which we use to address $\partial FF_n$ when
$n$ is odd.

\subsection{Previous work}
Our work is a direct anaglogue of Gabai's work \cite{GabAF} on the
connectivity of the ending lamination space $\mathcal{EL}$, which is
the Gromov boundary of the curve complex. He obtains optimal path
connectivity results and establishes higher connectivity, where
appropriate.  He goes further and in \cite{GabTop} identifies
$\mathcal{EL}$ of punctured spheres with so called N\"obling
spaces. This had previously been done in the case of the 5-times
punctured sphere by Hensel and Przytycki \cite{HenPrz}. 

However, there is one crucial difference in approaches: 
throughout his
arguments, Gabai homotopes paths in $\pml$, which is a
sphere.  It is
unclear how to do this in our setting, because the topology of
$\partial CV_n$ is still poorly understood, and it is not even know if
it is locally connected. In particular, our successive improvements of paths 
do not proceed by homotopy.

\smallskip Perhaps a better analogy for our work is Leininger
and Schleimer's proof that $\mathcal{EL}$ is connected \cite{LS}. They
do this by using ``point pushing" to find a dense path connected set
of arational laminations upstairs in $\pml$. Building on this, the
second and third named authors use contraction properties of the
mapping class group on the curve complex to show that the subsets of
uniquely ergodic and of cobounded foliations in $\pml$ are both path
connected \cite{CH}. This motivates our approach and especially
Proposition \ref{prop:chain}. 
However, there is again an important difference between the paths 
built in \cite{LS} or \cite{CH}, and the ones we construct here: in the 
former sources, the paths are often obtained from lower complexity surfaces
by lifting along (branched) covers. Here, we do not have this option, and
instead need to construct the paths directly.

Finally, in our setting, this previous work is of little help in getting 
between adjacent $\pml$s, where the new ideas of this article are needed.

%Path connectivity of uniquely
%ergodic interval exchange transformations was proven by the second and
%third named authors, who give an algorithm for creating these paths in
%the case of 4-interval exchange transformations \cite{CHiet}. It is
%not clear to us whether any of the above mentioned work can be made
%into an algorithm. As suggested by Leininger and Schleimer's
%work, even building paths of arational, uniquely ergodic or cobounded
%objects is challenging.  }

\subsection{Questions}
We end this introduction with a short list of further questions which
this work suggests.
\begin{enumerate}
\item Is the boundary $\partial FF_N$ already path-connected for $N\geq 3$? 
	The bound $\rank$ used here certainly carries no special significance, and
	is an artifact of the proof.
	
\item Do the boundaries $\partial FF_N$ satisfy (for large enough rank
  $N$) higher connectivity properties?
	
\item Are there topological models for the boundaries $\partial FF_N$?
  Most likely, this would involve showing that they satisfy other
  universal properties (dimension, locally finite $n$--disk
  properties)?
	
\item Is the set of arational trees in $\partial CV_n$ path-connected?
  We remark that the paths constructed between boundary points of the
  free factor complex do not yield paths in $\partial CV_n$, as
  continuity at the endpoints cannot be guaranteed. The corresponding
  question for $\pml$ is an open question of Gabai.
  
 \item Is the set of points in $\partial FF_n$ which are not of surface type path-connected? The paths we construct contain subpaths each point of which is of surface type. Avoiding this seems to require new ideas.
 We were made aware of this question by Camille Horbez.
\end{enumerate}

\subsection*{Acknowledgments}
Bestvina gratefully acknowledges the support by the National Science
Foundation, grant number DMS-1905720. Chaika gratefully acknowledges the support of National Science
Foundation, grant numbers DMS-135500 and DMS- 1452762, the Sloan foundation,  and a Warnock chair.
Hensel gratefully acknowledges the support of the Special Priority Programme SPP 2026 ``Geometry at Infinity'' funded by the DFG. Finally, Chaika and Hensel would like to thank the Hausdorff Research Institute for Mathematics in Bonn, where some of the work has been carried out.

\section{$\mathrm{Out}(F_n)$ preliminaries}
This section collects the neccessary facts about (compactified)
Culler-Vogtmann Outer space, related spaces where $\mathrm{Out}(F_n)$
acts, and geodesic laminations on surfaces.

\subsection{Outer Space}
Throughout this article, any \emph{tree} is understood to be a tree
together with an isometric action of $F_n$. Recall that a tree is {\it
  minimal} if does not contain a proper invariant subtree, and it is
{\it nontrivial} if it does not have a global fixed point. Unless
stated otherwise, all trees will be minimal and nontrivial. When $T$
is an $F_n$-tree and $a$ is an element or a conjugacy class in $F_n$,
we write $\<T,a\>$ for the translation length of $a$ in $T$.

We denote by $cv_n$ \emph{(unprojectivized) Outer space} in rank $n$,
and we denote by $CV_n = cv_n / (0,\infty)$ projectivised Outer
space. Points in these spaces correspond to (projectivised) free,
simplicial, minimal $F_n$--trees. Compare~\cite{CullerVogtmann,Vog08}
for details.  One can think of the space $CV_n$ as a free group analog
of Teichm\"uller space.  The function $T\mapsto (a\mapsto \<T,a\>)$
defines an embedding of $cv_n$ into the {\it space of length
  functions} $[0,\infty)^{F_n}$.  We denote by $\overline{cv}_n$ the
closure of ${cv}_n$ in this space and by $\overline{CV}_n$ its
projectivization. A point in $\overline{cv_n}$ determines a tree,
unique up to equivariant isometry. Both $\overline{cv_n}$ and
$\overline{CV_n}$ are metrizable ($\overline{cv}_n$ is a subspace of
the metrizable space $[0,\infty)^{F_n}$ and $\overline{CV_n}$ embeds
in $\overline{cv}_n$; see below). Moreover, $\overline{CV_n}$ is
compact and $\partial CV_n=\overline{CV_n}\smallsetminus CV_n$ plays
the role of $\pml$. For our later arguments we choose a
distance \dist on $CV_n$. As usual, this distance defines a
(Hausdorff) distance on the set of compact subsets, and we keep the
same notation for this distance. An element of $\overline{CV_n}$ is
represented by a projective class of trees, but we follow the custom
and talk about trees as points in $\overline{CV_n}$.

Recall the following characterisation of this compactification.
\begin{defin}
  A nontrivial minimal tree $T$ is {\it very small}, if arc
  stabilizers are trivial or maximal cyclic subgroups, and the fixed
  set of a nontrivial element does not contain a tripod.
\end{defin}
\begin{prop}[{\cite{bf:outerlimits,horbez:boundaryouterspace}}]
  A nontrivial minimal tree $T$ is contained in $\overline{CV_n}$ if
  and only if it is {\it very small}.
\end{prop}

\subsection{Arational Trees}\label{3.1}
To describe the boundary of the free factor complex, we need the
following notion.
\begin{defin}
  A tree $T\in \partial CV_n$ is {\it arational} if for every proper
  free factor $A<F_n$ the induced action of $A$ on $T$ is free and
  discrete.
\end{defin}

Two arational trees are {\it topologically equivalent} if there is an
equivariant homeomorphism in the observers' topology between them (we
elaborate on observers' topology below).  Equivalence classes of
arational trees in $\overline{CV_n}$ can be naturally identified with
simplices, analogously to simplices of projectivized transverse
measures on geodesic laminations. Denote by
$\mathcal{AT}\subset\partial CV_n$ the space of arational trees, and
by $\mathcal{AT}/\sim$ the quotient space obtained by collapsing each
equivalence class to a point.

\begin{thm}[{\cite{BesRey,Ham}}] The Gromov boundary $\partial FF_n$
  of the free factor complex is homeomorphic to $\mathcal{AT}/\sim$.
\end{thm}

We will need a more precise version of this theorem. There is a
function $$\Phi:\overline {CV_n}\to \overline {FF_n}$$ (in
\cite{BesRey} this is the map $\pi\cup\partial\pi$) with the following
properties:
\begin{itemize}
\item The restriction of $\Phi$ to the space $\mathcal{AT}$ of
  arational trees maps it continuously onto $\partial FF_n$
  \cite[Proposition 7.5]{BesRey}, it is a closed map \cite[Lemma
  8.6]{BesRey}, and the point inverses are exactly the simplices of
  equivalence classes of arational trees \cite[Proposition
  8.4]{BesRey}.
\item The restriction of $\Phi$ to the complement of $\mathcal{AT}$
  has $FF_n$ as its range and it is defined coarsely; it maps
  $T\in\partial CV_n\smallsetminus \mathcal{AT}$ to a free factor $A$
  such that the $A$-minimal subtree of $T$ has dense orbits (or $A$ is
  elliptic), and it maps $T\in CV_n$ to a free factor $A$ realized as
  a subgraph of $T/F_n$. See \cite[Lemma 5.1 and Corollary
  5.3]{BesRey}.
\item $\Phi$ is coarsely continuous: if $T_i$ is a sequence in
  $\overline{CV_n}$ and $\Delta$ is an equivalence class of arational
  trees, then $\Phi(T_i)\to \Phi(\Delta)$ if and only if the
  accumulation set of $T_i$ is contained in $\Delta$. See
  \cite[Proposition 8.3 and Proposition 8.5]{BesRey}.
\end{itemize}

We will have to regularly construct arational trees, and this
subsection collects some tools to do so.

\smallskip We let $cv_n^+$ be the union of $cv_n$ together with a
point $0$ representing the trivial action. The space $cv_n^+$ is
naturally a cone with $[0,\infty)$ acting by rescaling.  In a similar
way we define the space $\overline{cv_n}^+$ of very small trees,
together with a point representing the trivial action.

There are many ways of realizing $\overline{CV_n}$ as a section of the
cone $\overline{cv_n}^+$ (by which we mean a section of
$\overline{cv_n}^+\smallsetminus \{0\}\to\overline{CV_n}$).  We choose
the following: by Serre's lemma \cite{SerreTrees}, an action of $F_n$
on a tree so that each element of length $\leq 2$ acts elliptically,
in fact has a global fixed point. Thus, the sum of translation lengths
of all elements of length $\leq 2$ is positive in each tree of
$\overline{cv_n}$. We identify $\overline{CV_n}$ with the subset of
$\overline{cv_n}$ where the sum of translation lengths for all
elements of length $\leq 2$ is equal to $1$.

\begin{lem}
  For every proper free factor $A<F_n$, restricting the action of
  $F_n$ on a tree to a minimal subtree for the $A$--action yields a
  continous map
  $$r_A : \overline{cv_n}^+\to \overline{cv(A)}^+.$$
\end{lem}
\begin{proof}
  First, observe that if $A$ acts elliptically on $T$, then $r_A(T)$
  is the cone point in $\overline{cv(A)}^+$.

  Otherwise, we claim that the restriction is very small. Namely,
  consider any arc $a \subset T$. If its stabiliser is trivial, the
  same is obviously true for the restricted action. If the stabiliser
  is a maximal cyclic subgroup, then the same is true for the
  restricted action: since $A$ is a free factor, if $1\neq g\in A$
  then the maximal cyclic subgroup in $F_n$ containing $g$ is
  contained in $A$.  Finally, suppose that $g$ acts nontrivially on
  the restricted tree.  If it would fix a tripod, the same would be
  true in $T$, violating that $T$ is very small.

  Continuity is clear, since translation lengths for the restriction
  are the same as translation lengths in $T$.
\end{proof}
We define
\[ \rho_A:\partial CV_n\to \overline{cv(A)}^+ \] as the restriction of
the maps $r_A$ to the subset $\partial CV_n\subset \overline{cv_n}^+$
via the above normalization.

\bigskip We now make the following definition, which will be crucial
for our construction:
\begin{defin}\label{def:the K}
  Let $\F$ be the countable set of conjugacy classes of proper
  nontrivial free factors of $F_n$. For any $A\in \F$ we define
  \[ K_A = \partial CV_n \setminus \rho_A^{-1}(cv(A)) \]
\end{defin}
Thus $K_A$ is the set of trees in $\partial CV_n$ where $A$ does not
act freely and simplicially.  The following is now clear from the
above:
\begin{prop}\label{prop:the K}
  The collection $\{ K_A, A \in \F \}$ is a countable collection of
  closed subsets whose complement is the set of arational trees:
  \[ \mathcal{AT} = \partial CV_n \setminus \left(\bigcup_{A\in \F} K_A\right) \]
\end{prop}

\subsection{Trees, currents, and the action of
  $\mathrm{Out}(F_n)$}
Any $\phi\in \mathrm{Out}(F_n)$ acts naturally (on the left) on the
set of conjugacy classes. To be consistent with our later
constructions, we also define a left action of $\mathrm{Out}(F_n)$ on
the set of trees defined by
$$\<\phi T,a\>=\<T,\phi^{-1}(a)\>$$ The length pairing can be extended from
the set of conjugacy classes to the space $\mathcal{MC}_n$ of {\it
  measured geodesic currents} that contains positive multiples of
conjugacy classes \cite{Reiner,Ilya} (see below for the
definition). It admits an action of $(0,\infty)$ by scaling and a left
action of $\mathrm{Out}(F_n)$ that commutes with scaling and extends
the action on conjugacy classes. The length pairing extends to a
continuous function $\overline{cv_n}\times \mathcal{MC}_n\to
[0,\infty)$ that commutes with scaling in each coordinate
\cite{Kapovich-Lustig-pairing}.

\subsection{Dual laminations and currents}
For more details on this section see \cite{CHL0,CHL1,CHL2,CHL3}.
Denote by $\partial F_n$ the Cantor set of ends of $F_n$. A {\it
  lamination} $L$ is a closed subset of $\partial^2F_n:=\partial
F_n\times \partial F_n\smallsetminus \Delta$ invariant under
$(x,y)\mapsto (y,x)$ and under the left action of $F_n$, where
$\Delta$ is the diagonal. To every $T\in \overline{cv_n}$ one
associates the {\it dual lamination} $L(T)$, defined as
$$L(T)=\cap_{\epsilon>0}L_\epsilon(T)$$ where $L_\epsilon(T)$ is the
closure of set of pairs $(x,y)\in \partial^2F_n$ which are endpoints
in the Cayley graph
of axes of elements with translation length $<\epsilon$ in $T$. It
turns out that $L(T)$ is always {\it diagonally closed}
i.e. $(a,b),(b,c)\in L(T)$ implies $(a,c)\in L(T)$ (if $a\neq
c$), so it determines an equivalence relation on $\partial F_n$, and
the equivalence classes form an upper semi-continuous decomposition of
$\partial F_n$.

The precise definition of a measured geodesic current is that it is an
$F_n$-invariant and $(x,y)\mapsto (y,x)$ invariant Radon measure on
$\partial^2 F_n$ (i.e. a Borel measure which is finite on compact
sets). For example, a conjugacy class in $F_n$ determines a counting
measure on $\partial^2 F_n$ and can be viewed as a current.  The
topology on the space $\mathcal{MC}_n$ of all currents is the weak$^*$
topology.  The support $Supp(\mu)$ of a current is the smallest closed
set such that $\mu$ is 0 in the complement; the support of a current
is always a lamination. An important theorem relating currents and
laminations is the following.

\begin{thm}[\cite{Kapovich-Lustig-GAFA}]\label{KL}
  Let $T\in\overline{cv_n}$ and $\mu\in \mathcal{MC}_n$. Then
  $\<T,\mu\>=0$ if and only if $Supp(\mu)\subseteq L(T)$.
\end{thm}

\subsection{Observers' topology}
Let $T\in \overline{cv_n}$. Then $T$ is a metric space, but it also
admits a weaker topology, called {\it observers' topology}. A subbasis
for this topology consists of complementary components of individual
points in $T$. One can also form the metric completion of $T$ and add
the space of ends (i.e. the Gromov boundary) to form $\hat T$. The
space $\hat T$ also has observers' topology, defined in the same
way. In this topology $\hat T$ is always compact, and in fact it is a
{\it dendrite} (uniquely arcwise connected Peano continuum).  The
following is a theorem of Coulbois-Hilion-Lustig.  There is an
alternative description of $L(T)$ in terms of the $Q$-map $Q:\partial
F_n\to \hat T$ (see Levitt-Lustig \cite{Levitt-Lustig}): $(a,b)\in
L(T)$ if and only if $Q(a)=Q(b)$.

\begin{thm}{\cite{CHL0}}\label{CHL}
  Suppose $T$ has dense orbits. Then $\hat T$ with observers' topology
  is equivariantly homeomorphic to $\partial F_n/L(T)$.
\end{thm}

In fact, the $Q$-map realizes this homeomorphism.

\subsection{Surfaces, Laminations and $\bout$}\label{sec:arationality surface}
Suppose that $\Sigma$ is a compact oriented surface with one boundary
component. Then the fundamental group $\pi_1(\Sigma)$ is the free
group $F_{2g}$. If $\Sigma$ is a nonorientable surface with a single
boundary component, the fundamental group $\pi_1(\Sigma)$ is also
free, of even or odd rank (depending on the parity of the Euler
characteristic).

A lamination $\lambda$ on $\Sigma$ will always mean a \emph{measured
  geodesic lamination}.  For any surface $\Sigma$, we denote by
$\pml(\Sigma)$ the sphere of projective measured laminations.  See
\cite{FLP} for a thorough treatment. We will also need measured
laminations of non-orientable surfaces, but the only specific result
we rely on is that for a nonorientable $\Sigma$, the lifting map
\[ \pml(\Sigma) \to \pml(\Sigma') \] is a topological embedding, where
$\Sigma'$ denotes the orientation double cover.

\smallskip In our construction we will need paths in $\pml$ consisting
only of minimal laminations.  In the orientable case these will be
provided by the following theorem.
\begin{thm}[{Chaika and Hensel \cite{CH}}]\label{thm:ch} For a surface of genus at
  least 5 (with any number of marked points) the set of uniquely ergodic ergodic foliations is
  path-connected and locally path-connected in $\pml$. Furthermore,
  given any finite set $B$ of laminations, the complement of $B$ in
  $\pml$ is still path-connected.
\end{thm}

For non-orientable surfaces, we require a similar result. As we do not
need the full strength of Theorem~\ref{thm:ch} for our argument, we
will prove the following in the appendix (which follows quickly from
methods developed in \cite{LS}):
\begin{thm}\label{thm:weak-nonorientable-connectivity}
  Suppose that $\Sigma$ is a nonorientable surface with a single
  marked point $p$.
	
  Let $\mathcal{P} \subset \pml(\Sigma)$ be the set of
  minimal foliations which either
  \begin{enumerate}
  \item do not have an angle--$\pi$ singularity at $p$, or
  \item are stable foliations of point-pushing pseudo-Anosovs.
  \end{enumerate}
  Then $\mathcal{P}$ is path-connected, and invariant under the
  mapping class group of $\Sigma$. In addition, if $F$ is any finite
  set of laminations, the set $\mathcal{P}\setminus F$ is still
  path-connected.
\end{thm}

Given a lamination on $\Sigma$, we can lift it to a lamination
$\widetilde{\lambda}$ of the universal cover
$\widetilde{\Sigma}$. Since $\partial_\infty \widetilde{\Sigma}
= \partial_\infty \pi_1(\Sigma) = \partial_\infty F_n$, this allows us
to interpret $\lambda$ as a lamination on the free group. It is dual
(in the sense above) to the dual tree of the lamination
$\widetilde{\lambda}$ (in the geometric sense).

The following theorem, due to Skora, describes exactly which trees
appear in this way.

\begin{thm}\cite{skora}
  Suppose a closed surface group acts on a minimal $\R$-tree $T$ with
  cyclic arc stabilizers. Then $T$ is dual to a measured geodesic
  lamination on the surface. The same holds for surfaces with boundary
  if the fundamental group of each boundary component acts
  elliptically in $T$.
\end{thm}

For our purposes, we will need to be a bit more careful about how we
identify surface laminations and free group laminations. Namely, for
any identification $\sigma:\pi_1(\Sigma) \to F_n$ we obtain the
corresponding map
\[ \iota_\sigma : \pml(\Sigma) \to \overline{CV_n} \] 
mapping a lamination to its dual tree. We denote by
\[ \pml_\sigma = \iota_\sigma(\pml(\Sigma)) \]
the image of $\iota_\sigma$. In other words, $\pml_\sigma$ consists of those
trees which are realisable as duals to geodesic laminations on $\Sigma$,
given the identification of $F_n = \pi_1(\Sigma)$ via $\sigma$.

If $\phi$ is any outer automorphism, then the images of $\iota_\sigma$
and $\iota_{\sigma\circ\phi^{-1}}$ differ by applying the outer
automorphism $\phi$ (acting on $\overline{CV_n}$).  Core to our
argument will be to use the intersection of ``adjacent'' such copies
of $\pml$; see Section~\ref{sec:moves}. A typical lamination contained
in the intersection is one that fills a suitable subsurface of
$\Sigma$.

\subsection{Density of surface type arational trees}
Arational trees come in two flavors: ones dual to a filling measured
lamination on a compact surface with one boundary component (we will
call them arational trees of {\it surface type}), and the ``others''
-- these are free as $F_n$-trees (see
\cite{reynolds:reducingSystems}). Recall that every arational tree $T$
belongs to a canonical simplex $\Delta_T\subset \partial CV_n$
consisting of arational trees with the same dual lamination.  We will
need the following lemma.

\begin{lem}\label{surfaces dense}
  Let $T$ be an arational tree and let $U$ be a neighborhood of the
  simplex $\Delta_T$ in $\partial CV_n$. Then $U$ contains an
  arational tree $S$ of surface type.
\end{lem}

\begin{proof}
  All arational trees in rank 2 are of surface type (and all
  associated simplices are points) so we may assume $n\geq 3$.
  Vincent Guirardel showed in \cite{guirardel:thesis} that for $n\geq
  3$ the boundary $\partial CV_n$ contains a unique minimal
  $\mathrm{Out}(F_n)$-invariant closed set $\mathcal M_n$. In
  particular, $\mathrm{Out}(F_n)$ acts on $\mathcal M_n$ with dense
  orbits. He further showed that any arational tree (or indeed any
  tree with dense orbits) with ergodic Lebesgue measure belongs to
  $\mathcal M_n$. In our situation this means that the vertices of
  $\Delta_T$ belong to $\mathcal M_n$ and they can be approximated by
  points in the orbit of a fixed surface type arational tree that also
  belongs to $\mathcal M_n$.
\end{proof}

\subsection{Dynamics of partial pseudo-Anosovs}
In this section we will assemble the dynamical properties of partial
pseudo-Anosov homeomorphism as they act on Outer space. The proofs are
standard but we couldn't find the statements we need in the
literature.

The basic theorem is that of Levitt and Lustig. An outer automorphism
of $F_n$ is {\it fully irreducible} if all of its nontrivial powers
are irreducible, i.e. don't fix any proper free factors up to
conjugation.

\begin{thm}[\cite{Levitt-Lustig}]
  Let $f\in \mathrm{Out}(F_n)$ be a fully irreducible
  automorphism. Then $f$ acts with north-south dynamics on
  $\overline{CV_n}$.
\end{thm}

We start with a pseudo-Anosov homeomorphism $f:S\to S$ of a compact surface
$S$ (with one or more boundary components) with $\pi_1(S)=F_n$. By $\lambda>1$
denote the dilatation and by $\Lambda^{\pm}$ the stable and the
unstable measured geodesic laminations, so $f(\Lambda^{\pm})=\frac
1{\lambda^{\pm 1}}\Lambda^{\pm}$. Let $T^{\pm}$ be the trees dual to
$\Lambda^{\pm}$, and let $\mu^{\pm}$ be the measured currents
corresponding to $\Lambda^{\pm}$. Thus
$$f_*T^{\pm}=\lambda^{\pm 1}T^{\pm}$$
and
$$f_*(\mu^{\pm})=\lambda\mu^{\pm}$$
This implies $\<T^\pm,\mu^\pm\>=0$ and $\<T^\pm,\mu^\mp\> >0$.

\begin{prop}\label{NS1}
	With notation as above, let $K\subset \overline{CV_n}$ be a compact set of trees such that
  $\<T,\mu^+\>\neq 0$ for every $T\in K$. Then $f_*^m K$ converges to
  $T^+$ as $m\to\infty$.
\end{prop}

When $S$ has more than one boundary component $f_*$ is not fully
irreducible. It is irreducible if it permutes the boundary components
cyclically. The Levitt-Lustig argument can be used to prove the
north-south dynamics for irreducible automorphisms, and this would
suffice for our purposes since we could arrange that pseudo-Anosov
homeomorphisms we use in our construction later have roots that
cyclically permute the boundary components. However, we will give a
direct argument.

\begin{proof}
  We will view $K\subset \overline{cv_n}$ as a compact set of
  unprojectivized trees as in Section \ref{3.1}.  Let $Y$ be the
  accumulation set of the scaled forward iterates $f_*^m K/\lambda^m$
  of $K$. Note that if $T\in K$ then $\<\frac 1{\lambda^m}f_*^m
  T,\mu^+\>=\<T,\frac 1{\lambda^m}f_*^{-m}(\mu^+)\>= \<T,\mu^+\>$, so
  the length of $\mu^+$, being bounded below by some $\epsilon>0$ over
  $K$, is also bounded below by $\epsilon$ on $Y$. Similarly, $\<\frac
  1{\lambda^m}f_*^m T,\mu^-\>=\<T,\frac 1{\lambda^m}f_*^{-m}(\mu^-)\>=
  \<T,\frac 1{\lambda^{2m}}\mu^-\>\to 0$, so the length of $\mu^-$ is
  0 in all trees in $Y$. If $a$ is any conjugacy class other than a
  power of a boundary component, then we have by surface theory $\frac
  1{\lambda^m}f_*^m(a)\to C_a\mu^+$ for some $C_a>0$, and a similar
  argument as above shows that $\<T,a\>=C_a\<T,\mu^+\> >0$ for every
  $T\in Y$. However, when $a$ represents a boundary component then
  $\<T,a\>=0$ since then $f_*^m(a)$ is a boundary component for all
  $m$, and thus $\frac 1{\lambda^m}f_*^m(a) \to 0$.  In particular,
  all iterates $f_*^m K$ are contained in a compact subset of
  $\overline{cv_n}$ and so the accumulation set in $\overline{CV_n}$
  is the projectivization of $Y$. It follows from Skora's theorem that
  $Y$ consists of trees dual to measured geodesic laminations on
  $S$. The only lamination where $\mu^-$ has length 0 in the dual tree
  is $\Lambda^+$ and hence $Y=\{T^+\}$ (projectively).
\end{proof}

The next result gives a criterion to check the condition required in
the previous proposition.
\begin{prop}
	With notation as in the paragraph before Proposition \ref{NS1},
	suppose that $R$ is a compact surface with
  marking induced by a homotopy equivalence $\phi:S\to R$ (which may
  not send boundary to boundary). Let $\Lambda$ be a measured geodesic
  lamination on $R$ and let $T$ be the dual tree. If $\<T,\mu^+\>=0$
  then $T=T^-$ projectively.
\end{prop}

\begin{proof}
  The support $Supp(\mu^+)$ of $\mu^+$ is $\Lambda^-$. The
  complementary component of $\Lambda^-$ in $S$ that contains a
  boundary component is a ``crown region'' and adding diagonals
  amounts to adding infinitely many (non-embedded) lines that start
  and end in a cusp of the crown region and wind around the boundary
  component any number of times. These accumulate on the boundary, so
  each boundary component is in the diagonal closure of
  $Supp(\mu^+)$. Any other line will intersect the leaves of
  $\Lambda^-$ transversally, so it follows that the lamination
  $L(T^-)$ dual to the tree $T^-$ coincides with the diagonal closure
  of $Supp(\mu^+)$ (cf. \cite[Proposition 4.2(ii)]{BesRey}).

  Similarly, the lamination $L(T)$ dual to $T$ consists of leaves of
  $\Lambda$ together with all lines not crossing $\Lambda$
  transversally. Since $\<T,\mu^+\>=0$ we have $Supp(\mu^+)\subseteq
  L(T)$ by the Kapovich-Lustig Theorem \ref{KL}, and hence
  $L(T^-)\subseteq L(T)$. We now argue that
  $\Lambda$ must be a filling lamination. First, $\Lambda$ cannot
  contain closed leaves, for otherwise $L(T)$ would be carried
  by an infinite index finitely generated subgroup, contradicting the
  fact that $\Lambda^-$ isn't. Now suppose that $\Lambda$ contains a
  minimal component $\Lambda_0$ carried by a proper subsurface
  $R_0\subset R$. If there is a leaf of $L(T^-)$ asymptotic to a
  leaf of $\Lambda_0$, then $L(T^-)$ contains $\Lambda_0$ as well
  as boundary components of $R_0$ and the diagonal leaves within
  $R_0$. It then follows that $L(T^-)$ doesn't contain any other
  leaves since it equals the diagonal closure of any non-closed leaf,
  and hence again it is carried by an infinite index finitely
  generated subgroup. Thus $\Lambda$ must be filling and
  $L(T^-)=L(T)$. It now follows from the
  Coulbois-Hilion-Lustig Theorem \ref{CHL} that $T$ and
  $T^-$ are equivariantly homeomorphic in observers' topology. But
  since $T^-$ is uniquely ergodic, projectively $T=T^-$.
\end{proof}

\begin{rem}
  In fact, $\<T,\mu^+\>=0$ forces $T=T^-$ even without assuming that
  $T$ is dual to a lamination on a surface. This can be proved by
  noting that $L(T^-)\subseteq L(T)$ (which is proved above for all
  $T$ with $\<T,\mu^+\>=0$) forces $L(T)=L(T^-)$ by \cite[Proposition
    3.2]{BesRey} since $T^-$ is an indecomposable tree. It follows
  that $f_*$ satisfies the north-south dynamics on all of
  $\overline{CV_n}$, not just on compact sets dual to surface
  laminations.
\end{rem}

We now generalize this to partial pseudo-Anosov homeomorphisms.
\begin{prop}\label{NS}
  Let $f$ be a homeomorphism of $\Sigma$, which restricts to a
  pseudo-Anosov on a $\pi_1$-injective subsurface $S\subset\Sigma$,
  and to the identity in the complement. Suppose that $K\subset
  \overline{CV_n}$ is compact such that every $T\in K$ satisfies
  $\<T,\mu^+\>\neq 0$, where $\mu^+$ is a current supported in
  $\pi_1(S)$ corresponding to the stable lamination $\Lambda^+$ of
  $f$.
	
  Then the sequence $f_*^i K$, $i\to\infty$, converges to the tree
  $T_f\in\partial CV_n$ dual to $\Lambda^+$ on $\Sigma$.
\end{prop}

Note that the condition implies that $\pi_1(S)$ is not elliptic in any
$T\in K$.  To begin the proof of Proposition \ref{NS}, denote by
$Y\subset \overline{CV_n}$ the set of accumulation points of $f^i K$
as $i\to\infty$. We need to show that $Y=\{T_f\}$.

\begin{lem}\label{LL}
  Let $T\in Y$. Every conjugacy class represented by a (not
  necessarily simple) curve in $\Sigma\smallsetminus S$ is elliptic in
  $T$. The minimal $\pi_1(S)$-subtree of $T$ is dual to $\Lambda^+$.
\end{lem}

\begin{proof}
  The set $K$ is defined as a subset of $\overline{CV_n}$, but we can
  lift it to unprojectivized space $\overline{cv_n}$ as in Section
  \ref{3.1}. Now let $\gamma$ be a curve in the complement of
  $S$. Since $K$ is compact, the set $\{\ell_R(\gamma)\mid R\in Y\}$
  is bounded. Since $f(\gamma)=\gamma$ we deduce that the set
  $\{\<R,f_*^{-i}(\gamma)\> \mid R\in Y, i\in\mathbb{Z}\}$ is bounded
  as well. But $\<f_*^i R,\gamma\> =\<R,f_*^{-i}(\gamma)\>$, so the
  length of $\gamma$ in $f_*^i(X)$ stays uniformly bounded for all
  $i$.
	
  On the other hand, let $\delta$ be a curve in $S$. Applying
  Proposition \ref{NS1} to the trees obtained from the trees in $K$ by
  restricting to $\pi_1(S)$ we see that these restrictions converge to
  the tree dual to $\Lambda^+$. Moreover, the length of $\delta$ along
  $f_*^i K$ goes to infinity as $i\to\infty$ (in fact, the length
  grows like $(const)\lambda^i$ where $\lambda$ is the dilatation).
  Thus projectively, the length of $\gamma$ will go to 0 as
  $i\to\infty$.
\end{proof}

\begin{proof}[Proof of Proposition \ref{NS}]
  Let $T\in Y$. By Lemma \ref{LL} conjugacy classes of boundary
  components of $\Sigma$ are all elliptic, so by Skora's theorem we
  deduce that $T$ is dual to a measured geodesic lamination on
  $\Sigma$. Again by Lemma \ref{LL} this lamination is $\Lambda^+$ on
  $S$ and possibly curves in $\partial S\smallsetminus \partial\Sigma$
  with nonzero weight. Thus the accumulation set is contained in the
  simplex in $\pml(\Sigma)$ whose vertices are $\Lambda^+$ and these
  curves, it is compact, and disjoint from the face opposite the
  vertex $\Lambda^+$ (in particular, no point in the accumulation set
  is a curve in $\partial S\smallsetminus \partial\Sigma$). The only
  $f$-invariant compact set with this property is $\{\Lambda^+\}$
  since $f$ acts by attracting towards $\Lambda^+$ all compact sets
  disjoint from the opposite face.
\end{proof}

\begin{cor}
  In the setting of Proposition \ref{NS}, let $A<F_n$ be a free factor
  such that the action of $A$ on $T^+$ is free and discrete. Then for
  large $k>0$ the action of $f^k_*(A)$ on every $T\in K$ is free and
  discrete.
\end{cor}

\begin{proof}
  The subgroup $A$ will act freely and discretely in a neighborhood of
  $T^+$, and this includes $Kf^k_*$ for large $k$. This is equivalent
  to the statement.
\end{proof}

There is a simple criterion for deciding if the action of $A$ on $T^+$
is free and discrete.

\begin{lem}\label{lem:leaving the K with pAs}
  Let $A<F_n$ be a free factor and $T^+$ the tree dual to the stable
  lamination $\Lambda_f$ of a partial pseudo-Anosov homeomorphism $f$
  supported on a subsurface $S\subset\Sigma$. The action of $A$ on
  $T^+$ is free and discrete if and only if $\pi_1(S)$ is not
  conjugate into $A$ and no nontrivial conjugacy class in $A$ is
  represented by an immersed curve in $\Sigma\smallsetminus
  \Lambda^+$.
\end{lem}

\begin{proof}
  It is clear that these conditions are necessary. Assuming they hold,
  equip $\Sigma$ with a complete hyperbolic metric of finite area and
  let $\tilde\Sigma\to\Sigma$ be the covering space with
  $\pi_1(\tilde\Sigma)=A$. Lift the hyperbolic metric to
  $\tilde\Sigma$ and let $\tilde\Sigma_C\subset\tilde\Sigma$ be the
  convex core. Then $\tilde\Sigma_C$ is compact, since any cusp would
  represent a boundary component of $\Sigma$ whose conjugacy class is
  in $A$. Lift the lamination $\Lambda^+$ to $\tilde\Lambda\subset
  \tilde\Sigma$. Each leaf of $\tilde\Lambda$ intersects
  $\tilde\Sigma_C$ in an arc (or not at all) for otherwise $A$ would
  carry $\Lambda_f$ and would have to contain a finite index subgroup
  of $\pi_1(S)$. But $A$ is root-closed, so it would contain
  $\pi_1(S)$, which we excluded. So the intersection of
  $\tilde\Lambda$ with $\tilde\Sigma_C$ consists of finitely many
  isotopy classes of arcs. These arcs are filling, for otherwise we
  would have a loop in the complement that would represent a
  nontrivial element of $A$ whose image in $\Sigma$ is disjoint from
  $\Lambda^+$. The minimal $A$-subtree of $T_f$ is dual to this
  collection of arcs, and so is free and discrete.
\end{proof}

Finally, we need the following, which describes the dynamics of
partial pseudo-Anosovs on free factors. We use the terminology of
good and bad subsurfaces, which will be motivated and introduced
in Section~\ref{sec:moves}.
\begin{prop}\label{prop:good-and-bad}
  Identify $\pi_1(\Sigma) = F_n$ and assume that the rank $n$ of the
  free group is at least $\rank$.  
  Suppose that $\psi$ is a partial pseudo-Anosov, supported on a
  ``good'' subsurface $S^g$ of $\Sigma$, and so that the ``bad subgroup'' 
  $\pi_1(\Sigma-S^g)$ has rank at most $5$.
  
  Let $E$ be any free factor, and $B'$ a subgroup which does not contain $\pi_1(S^g)$ up
  to conjugacy. 

  Then one of the following holds:
  \begin{enumerate}
  \item $\pi_1(S^g)\subset E$ , or
  \item there is some $k>0$ so that for all large enough $N$ the only
    conjugacy classes belonging to $\psi^{kN}E$ and $B'$ are contained
    in $\pi_1(\Sigma-S^g)$.
  \end{enumerate}
\end{prop}
\begin{proof}
  Assume that (1) fails, i.e. $\pi_1(S^g)$ is not conjugate into $E$.
  Choose a hyperbolic metric on $\Sigma$. From now on, we assume that all
  curves and laminations are geodesic. By Scott's theorem, there is a
  finite cover $X \to \Sigma$ so that $E = \pi_1(X_E)$ where $X_E \subset
  X$ is a subsurface. We can choose a power $k$ so that 
  $\psi^k$ lifts to a partial pseudo-Anosov map $\hat{\psi}$ of $X$.
	
  Let $\hat{\lambda}^\pm$ be the lifts of the stable/unstable
  lamination of $\psi$ to $X$; in other words, these are the stable/unstable
  laminations of $\hat{\psi}$. Since both of these laminations fill
  $\pi_1(S^g)$, and we assume that $\pi_1(S^g)$ is not conjugate into
  $E$, no leaf of these laminations is completely contained in $X_E$.
	
  The intersection $\hat{\lambda}^\pm \cap X_E$ is supported in a
  subsurface $Y \subset X_E$, and we let $Y' = X_E - Y$ be its
  complement. We emphasise that the complement could be empty. Also
  observe that any curve in $Y'$ maps (under the covering map) into
  $\Sigma-S^g$.
	
  By compactness and the fact that no leaf of $\hat{\lambda}^-$ is
  supported in $X_E$ there is a number $\alpha$ with the following
  property: any geodesic starting in a point $p \in \hat{\lambda}^-$
  and making angle $<\alpha$ with $\hat{\lambda}^-$ leaves $X_E$.  In
  particular we conclude that any closed geodesic in $X_E$ which
  intersects $Y$ (and hence $\hat{\lambda}^-$) makes angle $\geq
  \alpha$ with $\hat{\lambda}^-$.
	
  As a consequence, we have the following: for any $L$ and $\epsilon$
  there is a number $N(L,\epsilon)$ with the following property. If
  $\gamma\subset X_E$ is any geodesic intersecting $Y$, then
  $\hat{\psi}^{n}\gamma$ contains a segment of length $\geq L$ which
  $\epsilon$--fellow-travels a leaf of $\hat{\lambda}^+$, for all
  $n>N(L,\epsilon)$.
	
  We now claim that there are $L, \epsilon>0$ with the following property: no geodesic
  $\gamma$ in $X$, which represents a conjugacy class of $B'$, contains
  a geodesic segment of length $\geq L$ which $\epsilon$--fellow-travels a leaf of $\hat{\lambda}^+$.
  
  To see this, we argue by contradiction. Namely, if not, then we could find a sequence of geodesics
  $\gamma_n \subset X$ which represent conjugacy classes in $B'$, and which limit to a leaf of $\hat{\lambda}^+$.
  In particular, the endpoints at infinity of a leaf of $\lambda^+$ would be contained in the boundary
  at infinity $\partial_\infty B'$ of the subgroup $B'$. Since $\lambda^+$ is the stable lamination
  of a partial pseudo-Anosov supported on $S^g$, this would imply that $B'$ contains $\pi_1(S^g)$ up to 
  conjugacy -- which contradicts our assumption. 
   
%  If $B'$ is a rank at most $5$ subgroup, no leaf of $\hat{\lambda}^+$
%  has both endpoints at infinity contained in $\partial_\infty B'$.  Namely, in a suitable cover of
%  $\Sigma$, the subgroup $B'$ is realised as the fundamental group of
%  a subsurface by Scott's theorem. If $\hat{\lambda}^+$ had both
%  endpoints at infinity in $\partial_\infty B'$, then this leaf (and thus all of
%  $\lambda^+$) would lift into that subsurface. However, since
%  $S^g$ has rank at least $\rank-5$, this is impossible, as
%  $S^g$ has the minimal genus of a subsurface supporting
%  $\lambda^+$.  The same is true if $B'$ does not contain
%  $\pi_1(S^g)$, as any subgroup whose limit set contains the endpoints
%  at infinity of a leaf of $\lambda$ contains the group $\pi_1(S^g)$ (up to
%  conjugacy).
%    
  
  We let $N=N(L,\epsilon)$ be the corresponding constant.
  Now, let $r>0$ be a number so that $\rho^r$ lifts to $X$ for any
  loop $\rho$ in $\Sigma$ (this exists, since $X$ is a finite cover). Let
  $\rho$ be any element of $E$, which is not conjugate into
  $\pi_1(\Sigma-S^g)$.  Since fundamental groups of subsurfaces are
  root-closed, the element $\rho^r$ is then also not conjugate into
  $\pi_1(\Sigma-S^g)$.  Then, $\rho^r$ lifts to a geodesic $\gamma$ in
  $X_E$, which intersects $Y$. Thus, $\hat{\psi}^n\gamma$ contains a segment
  with the property of the previous paragraph, showing that
  $\psi^{kn}\rho^r$ is not contained in $B'$ for all $n>0$.
	
  In other words, if $\psi^{kn}\rho$ (and thus $\psi^{kn}\rho^r$) is
  contained in $B'$, then $\rho$ is contained in $\pi_1(\Sigma-S^g)$,
  showing (2).
\end{proof}

\section{Basic Moves, Good and Bad Subsurfaces}
\label{sec:moves}

In this section, we will study how to relate different identifications
of a free group with the fundamental group of a surface. The basic
situation will be to relate two identifications which differ by
applying (certain) generators of $\mathrm{Out}(F_n)$.

\subsection{Standard Geometric Bases}\label{sgb}
Let $\Sigma = S_{g,1}$ be a compact oriented surface of genus $g$ with
one boundary component.  We pick a basepoint $p$, contained in the
interior of the surface $\Sigma$.  A collection of simple closed
curves $a_i, \hat{a}_i, 1\leq i\leq g$ is called a \emph{standard
  geometric basis for $S_{g,1}$} if the following hold:
\begin{enumerate}
\item The $a_i, \hat{a}_j$ generate $F_{2g} = \pi_1(\Sigma, p)$.
\item The $a_i, \hat{a}_j$ intersect only in $p$.
\item The cyclic order of incoming and outgoing arcs at $p$ is
  \[ \hat{a}_1^+, a_1^+, \hat{a}_1^-, a_1^-, \hat{a}_2^+, \ldots,
  a_g^-. \]
\end{enumerate}
\begin{figure}[h]
  \includegraphics[width=0.9\textwidth]{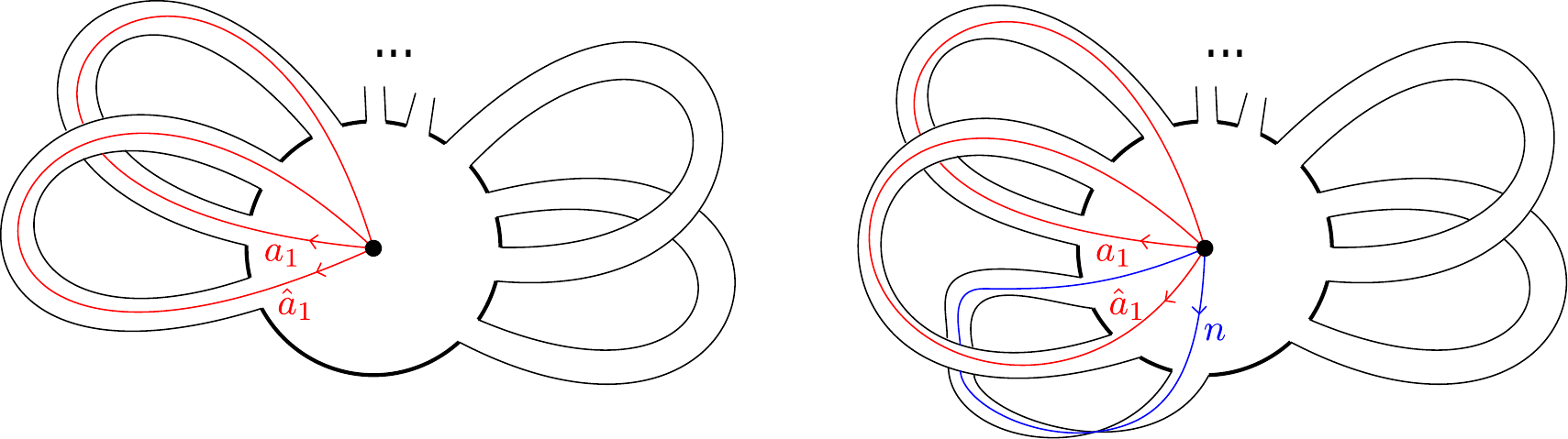}
  \caption{A basis for $\Sigma$ of the type we use in this section.}
  \label{fig:basis}
\end{figure}
Compare Figure~\ref{fig:basis} for an example of such a basis.  Given
a standard geometric basis, we say that $a_i, \hat{a}_i$ are an
\emph{intersecting pair} (of that basis). For ease of notation we
define $\hat{\hat{a}}_i = a_i$. Observe that up to the action of the
mapping class group of $\Sigma$ there is a unique standard geometric
basis.

\bigskip To deal with free groups of odd rank, we need to also
consider certain nonorientable surfaces.  Namely, let $\Sigma =
N_{2g+1}$ be the surface obtained from $S_g$ by attaching a single
twisted band, at the two sides of an initial segment of $a_1^+$.

As before, we pick a basepoint $p$ in the interior of $\Sigma$. A
collection of simple closed curves $n,a_i,\hat{a}_i 1\leq i\leq g$ is
called a \emph{standard geometric basis for $N_{2g+1}$} if the
following hold:
\begin{enumerate}
\item The curve $n$ is one-sided,
\item all $a_i, \hat{a}_i$ are two-sided,
\item the elements $n,a_i,\hat{a}_i$ generate $F_{2g+1} =
  \pi_1(\Sigma, p)$,
\item The $n,a_i, \hat{a}_i$ intersect only in $p$.
\item The cyclic order of incoming and outgoing arcs at $p$ is
  \[ n^+, \hat{a}_1^+, n^-, a_1^+, \hat{a}_1^-, a_1^-, \hat{a}_2^+,
  a_2^+, \hat{a}_2^-, a_2^-\ldots \]
\end{enumerate}
As above, we say that $a_i, \hat{a}_i$ are an \emph{intersecting
  pair}. In addition, we also say that $n, \hat{a}_1$ and $n,a_1$ are
\emph{intersecting pairs}.  We call $a_1, \hat{a}_1$ the
\emph{nonorientable-linked letters}.
\begin{rem}
  The reason for the somewhat asymmetric setup in the nonorientable
  setting is as follows. For later arguments, we will need to find
  two-sided curves which intersect the one-sided curve given by the
  basis letter in a single point. This forces at least one of the
  two-sided bands to be linked with the one-sided band. However, since
  we also need to be able to have an odd total number of bands, the
  described setup emerges.
\end{rem}
\begin{defin}
  Suppose that $x_1, \ldots x_n$ is a free basis for $F_n$.  For any
  $x=x_i$, $y=x_j$ with $i\neq j$, we call an outer automorphism
  defined by the automorphism
 
  \[\rho_{x,y}(z) = \begin{cases}
    xy & z = x \\
    z & z = x_k, k \neq i
  \end{cases}\] or
  \[\lambda_{x,y}(z) = \begin{cases}
    yx & z = x \\
    z & z = x_k, k \neq i
  \end{cases}\] a \emph{basic (Nielsen) move}. For $x=x_i$, we call an
  outer automorphism defined by the automorphism
  \[\iota_x(z) = \begin{cases}
    x^{-1} & z = x \\
    z & z = x_k, k \neq i
  \end{cases}\] a \emph{basic (invert) move}.
\end{defin}

\begin{lem}
  Given any standard geometric basis, $\mathrm{Out}(F_n)$ is generated
  by the basic invert moves, and Nielsen moves $\phi_{x,y}$ for $x,y$
  not an intersecting pair.
\end{lem}
\begin{proof}
  It is well-known that for any basis (in particular, standard
  geometric bases) all basic moves of the form above generate
  $\mathrm{Out}(F_n)$ \cite{Nielsen24}.  So, to prove the lemma, we
  just need to show that a Nielsen move for an intersecting pair can
  be written in terms of nonintersecting pairs. This is clear,
  e.g. for an unrelated letter $z$ we have:
  \[ \phi_{a_i, \hat{a}_i} = \phi_{z, \hat{a}_i}^{-1}\phi^{-1}_{a_i,
    z}\phi_{z, \hat{a}_i}\phi_{a_i, z}.\]
\end{proof}

\begin{defin}
  Suppose we have chosen an identification $\sigma:\pi_1(\Sigma) \to
  F_n$, and basic move $\phi$. We say a subsurface $S \subset \Sigma$
  \emph{good for $\phi$} if there is a pseudo-Anosov $\psi$ supported
  on $S$ which commutes with $\phi$ on the level of fundamental
  groups, under the identification $\sigma$, i.e.
  \[ \phi\sigma\psi_\ast\sigma^{-1} = \sigma\psi_\ast\sigma^{-1}\phi \]
  We call such a $\psi$ an \emph{associated partial
    pA}.  We call the complement of a chosen good subsurface a
  \emph{bad subsurface}.
\end{defin}

We need a bit more flexibility than basic moves, given by the
following definition.
\begin{defin}
  Suppose $\sigma:\pi_1(\Sigma) \to F_n$ is an identification, and
  $\mathcal{B}$ is a standard geometric basis for $\Sigma$.  We then
  call a conjugate of a basic move (of $\mathcal{B}$) by a mapping
  class of $\Sigma$ an \emph{adjusted move}.
\end{defin}
We observe that the good and bad subsurfaces of an adjusted move are
obtained from the corresponding subsurfaces of the basic move by
applying the mapping class.

Observe that, strictly speaking, neither good nor bad subsurfaces are
unique, but we will explain below which ones we choose.

\smallskip The key reason why we are interested in good and bad
subsurfaces is that we will try to apply Lemma~\ref{lem:leaving the K
  with pAs} to the partial pseudo-Anosovs guaranteed to exist on the
good subsurface.  In order to do this, we will need to find relations
in $\mathrm{Out}(F_n)$ avoiding the following two ``problems''
(corresponding to the two conditions in Lemma~\ref{lem:leaving the K
  with pAs}):

\begin{defin}
  Let $\phi$ be a basic or adjusted move, and $E < F_n$ a free factor.
  \begin{enumerate}
  \item We say that $E$ \emph{is an overlap problem for $\phi$ (and a
      choice of good and bad subsurface)} if some nontrivial conjugacy
    class $w \in E$ is contained (up to conjugacy) in the fundamental
    group of the bad subsurface.
  \item We say that $E$ \emph{is a containment problem for $\phi$ (and
      a choice of good and bad subsurface)} if the fundamental group
    of the good subsurface is contained (up to conjugacy) in $E$.
  \end{enumerate}
\end{defin}

Finally, recall that an identification $\sigma:\pi_1(\Sigma) \to F_n$
defines a copy $\pml_\sigma$ of $\pml(\Sigma)$ inside $\bout$. The
following notion is central for our construction.
\begin{defin}
  Given any identification $\sigma:\pi_1(\Sigma) \to F_n$, and
  adjusted move $\phi$ with respect to a standard geometric basis of
  $\Sigma$, we say that the copies
  \[ \pml_\sigma \quad\mbox{and}\quad \phi\pml_\sigma =
  \pml_{\phi\sigma} \] are \emph{adjacent}.
\end{defin}

\subsection{Good Subsurfaces}
To find good subsurfaces, we use the following two lemmas. The first
constructs an ``obvious'' good subsurface (which is not large enough
for our purposes). The second one will construct curves that yield
additional commuting Dehn twists, which extend the good subsurface.

Throughout this section, we fix a standard geometric basis
$\mathcal{B}$, based at a point $p$.
\begin{lem}[''Obvious'' good
  subsurfaces]\label{lem:obvious-good-subsurface}
  Let $x,y$ be two basis elements of $\mathcal{B}$ which are not an
  intersecting pair. Then there is a subsurface $S_0$ with the
  following properties.
  \begin{enumerate}
  \item If $x$ is two-sided and not linked with the one-sided loop,
    then $x, \hat{x}, y, \hat{y}$ are disjoint from $S_0$. If $x=n$ is
    one-sided and linked with $a$, or $x=a,\hat{a}$ is linked with the
    one-sided letter $n$, then $n, a, \hat{a}, y, \hat{y}$ are
    disjoint from $S_0$,
  \item Any other basis loop in $\mathcal{B}$ is freely homotopic into
    $S_0$, and intersects $\partial S_0$ in two points.
  \item Any mapping class supported in $S_0$ commutes with $\iota_x,
    \lambda_{x,y}$ and $\rho_{x,y}$.
  \end{enumerate}
  We call the letters as in (1) the \emph{active} letters of the
  basis, and all others the \emph{inactive}.
\end{lem}
\begin{proof}
  \begin{figure}
    \includegraphics[width=0.7\textwidth]{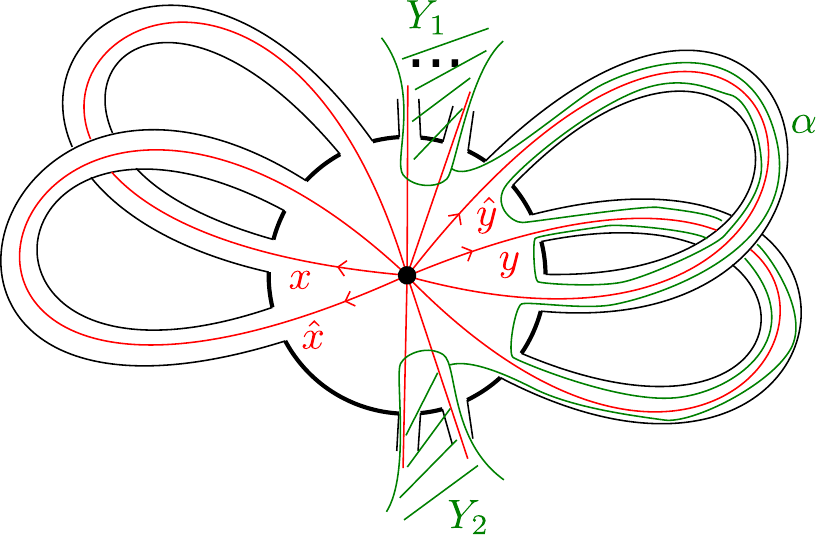}
			
    \caption{Standard geometric bases, and ``obvious'' good subsurfaces}
    \label{fig:finding-good-subsurface-ab}
  \end{figure}

  Let $Y_1, Y_2$ be the subsurfaces filled by the elements of the
  given standard geometric basis $\mathcal{B}^0$ which are between
  $x,\hat{x}$ and $y, \hat{y}$ in the cyclic ordering induced by the
  orientation of the surface. We homotope $Y_1, Y_2$ slightly off the
  basepoint so that they are disjoint from $x,\hat{x},y,\hat{y}$.
  Depending on the configuration, one of the $Y_i$ may be empty. If
  both $Y_i$ are nonempty, choose an arc $\alpha$ connecting $Y_1$ to
  $Y_2$ disjoint from all $a_i, \hat{a}_i$, homotopic to the product
  $\hat{y}y\hat{y}^{-1}y^{-1}$.  Compare
  Figure~\ref{fig:finding-good-subsurface-ab} for this setup.  We let
  $S_0$ be the subsurface obtained as a band sum of $Y_1, Y_2$ along
  $\alpha$, homotoped slightly so the basepoint is outside $S_0$. 

  Now observe that if $F$ is any mapping class supported in $S_0$,
  then $F$ acts trivially on all of $x, \hat{x},y,
  \hat{y}$. Furthermore, by construction, any loop in $S_0$ can be
  written in the basis $\mathcal{B}$ without $x$.  Together, these
  imply that $F_*$ commutes with $\rho_{x,y}, \lambda_{x,y}$ and
  $\iota_x$.
	
  \smallskip The argument in the nonorientable case is very similar,
  with the three letters $n, \hat{a}_1, a_1$ playing the role of
  $x,\hat{x}$.
\end{proof}
	
\begin{lem}\label{lem:additional-twist}
  In the setting of Lemma~\ref{lem:obvious-good-subsurface}, denote by 
  $x,\hat{x},y,\hat{y}$ the active letters, and let $S_0$ be the subsurface 
  guaranteed by that lemma.
		
  Then there are two-sided curves $\delta^+, \delta^-$ with the
  following properties:
  \begin{enumerate}
  \item $\delta^+$ (or $\delta^-$) intersect $x$ in a single point on
    $x^+$ (or $x^-$),
  \item the curves $\delta^+, \delta^-$ do not cross the band
    corresponding to $x$.
  \item $\delta^+, \delta^-$ are disjoint from $y$.
  \item $\delta^+, \delta^-$ intersect $S_0$ essentially.
  \end{enumerate}
  
  If one of $y, \hat{y}$ is either one-sided or linked
  with the one-sided letter, then there is additionally a curve
  $\delta^0$ intersecting $\partial S_0$ essentially, and which
  satisfies (2) and (3), but is disjoint from $x$.
\end{lem}
\begin{proof}
  We construct the curves case-by-case, beginning with the orientable
  case.
  \begin{figure}
    \includegraphics[width=0.95\textwidth]{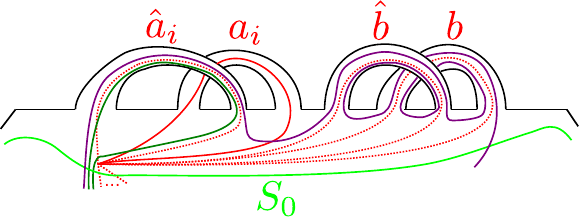}
			
    \caption{Constructing ``extra twists'' in Lemma~\ref{lem:additional-twist}.}
    \label{fig:finding-extratwists-case1}
  \end{figure}
		
  Here, we simply take $\gamma$ to be an embedded arc in $\Sigma-S_0$
  which intersects $x^+$ (or $x^-$) in a single point and is disjoint
  from $y, \hat{y}$, and does not intersect the interior of the band
  corresponding to $x$ (compare
  Figure~\ref{fig:finding-good-subsurface-ab}
  and~\ref{fig:finding-extratwists-case1}). The desired curve is then
  obtained by concatenating $\gamma$ with any nonseparating arc in
  $S_0$.
		
  In the nonorientable case, we do exactly the same, making sure that
  the arc $\gamma$ (and the arc in $S_0$) are two-sided.
\end{proof}

For some of the arguments in the sequel, we will need explicit
descriptions of the curves produced by
Lemma~\ref{lem:additional-twist} and the fundamental groups of the
resulting bad subsurfaces. As this is a somewhat tedious exercise in
constructing and analyzing explicit curves (and the proof follows the
exact same strategy in all cases), we only discuss the orientable case
here, and defer all further cases to
Lemma~\ref{lem:explicit-good-subsurfaces} in
Appendix~\ref{sec:pictures-cases}.
\begin{lem}\label{lem:explicit-good-subsurfaces-case1}
  Fix a standard geometric basis $\mathcal{B}$ of an orientable
  surface $\Sigma=\Sigma_{g,1}$, and use it to identify
  $\pi_1(\Sigma)$ with $F_{2g}$.  We denote by $\partial$ the word
  representing the boundary of the surface, i.e.
  \[ \partial = \prod_{i=1}^g [\hat{a}_i, a^{-1}_i], \]
  and by $\partial_w$ the cyclic permutation of $\partial$ starting with the element $w$. 
  
  Let $x,y$ be two elements of $\mathcal{B}$ which are not linked.
  \begin{itemize}
  \item If $x=a_i$, then the bad subsurface for the right
    multiplication move $\rho_{x,y}$ has fundamental group
    \[ \pi_1(\Sigma-S^g) = \langle y, \hat{y},
    x^{-1}\hat{x}x,\partial_{\hat{a}_{i+1}}\rangle. \] The bad
    subsurface for the left multiplication move $\lambda_{x,y}$ has
    fundamental group
    \[ \pi_1(\Sigma-S^g) = \langle y, \hat{y},
    \hat{x},\partial_{a_i}\rangle \]
  
  \item If $x=\hat{a}_i$, then the bad subsurface for the right
    multiplication move $\rho_{x,y}$ has fundamental group
    \[ \pi_1(\Sigma-S^g) = \langle y, \hat{y},
    \hat{x},\partial_{a^{-1}_i}\rangle \] The bad subsurface for the
    left multiplication move $\lambda_{x,y}$ has fundamental group
    \[ \pi_1(\Sigma-S^g) = \langle y, \hat{y},
    x\hat{x}x^{-1},\partial_{\hat{a}_i}\rangle \]
  \end{itemize}
  In both cases, every loop corresponding to a basis letter except $x,
  \hat{x}, y, \hat{y}$ is freely homotopic into the good subsurface.
\end{lem}
\begin{proof}
  We suppose that $x = a_i$, and begin with the subsurface $S_0$ from
  Lemma~\ref{lem:obvious-good-subsurface}. Part~(3) of that lemma
  shows that it is indeed good, and part~(2) shows that it has the
  property claimed in the last sentence of the lemma.
  
  We now use curves from Lemma~\ref{lem:additional-twist} to find
  additional Dehn twists which commute with the basic moves. We begin
  with the case of $\phi = \rho_{x,y}$. Here, we use the curve
  $\delta^+$ guaranteed by that lemma (shown in dark green in
  Figure~\ref{fig:finding-extratwists-case1}). The action of the twist
  $T_{\delta^+}$ on $\mathcal{B}$ depends on the type of
  letter. However, the relevant properties for us are the following:
  \begin{enumerate}
  \item $T_{\delta^+}(x) = wx$, where $w$ is a word not involving
    $x$. Namely, by property (1) of Lemma~\ref{lem:additional-twist},
    the twisted curve $T_{\delta^+}(x)$ is obtained by following $x^+$
    until the intersection point, following around $\delta^+$, and
    then continuing along $x$. By property (2), the curve $\delta^+$
    does not cross the band corresponding to $x$, yielding the desired
    property of $w$.
  \item $T_{\delta^+}(y) = y$. This follows since by (3) of
    Lemma~\ref{lem:additional-twist}, $\delta^+$ and $y$ are disjoint.
  \item For any other basis element $z$, the image $T_{\delta^+}(z)$
    is a word in $\mathcal{B}$ not involving $x$. Again, this follows
    from Property~(2), since the curve $\delta^+$ does not cross the
    band corresponding to $x$.
  \end{enumerate}
  These imply that $T_{\delta^+}$ commute with $\rho_{x,y}$:
  \begin{enumerate}
  \item Since $w$ does not involve $x$, we have
    \[ \rho_{x,y}T_{\delta^+}(x) = \rho_{x,y}(wx) = wxy. \]
    Since, by (2) above, the twist fixes $y$, we also have:
    \[ T_{\delta^+}(\rho_{x,y}(x)) = T_{\delta^+}(xy) = wxy \]
  \item Since both $T_{\delta^+}$ and $\rho_{x,y}$ fix $y$, we have
    \[ \rho_{x,y}T_{\delta^+}(y) = y = T_{\delta^+}(\rho_{x,y}(y)) \]
  \item Finally, since for any other basis element $z$, the image
    $T_{\delta^+}(z)$ is a word in $\mathcal{B}$ not involving $x$ by
    (3) above, we have
    \[ \rho_{x,y}T_{\delta^+}(z) = T_{\delta^+}(z) =
    T_{\delta^+}(\rho_{x,y}z). \]
  \end{enumerate}
  Now, let $S^g$ be a regular neighbourhood of $S_0 \cup
  \delta^+$. By property (4) of Lemma~\ref{lem:additional-twist}, this is 
  strictly bigger than $S_0$. 
  Observe that it is filled by $\delta^+$ and curves
  contained in $S_0$. Since we have shown that such twists commute
  with $\rho_{x,y}$, and a suitable product of such twists is a
  pseudo-Anosov map of $S^g$, it is indeed a good subsurface for
  $\rho_{x,y}$. It remains to compute the fundamental group, which can
  be read off from Figure~\ref{fig:finding-extratwists-case1}.

  \smallskip For the basic move $\lambda_{x,y}$ the proof is
  analogous, using that $\delta^-$ intersects $x$ only in $x^-$,
  proving that $T_{\delta^-}(x) = xw'$.
  
  \smallskip The strategy for the case $x=\hat{a}_i$ is analogous.
\end{proof}				

We also need the analog for invert moves.
\begin{lem}\label{lem:other-moves}
  With notation as in the previous lemma, consider a basic invert move
  $\phi=\iota_x$.  The bad subsurface has fundamental group
  \[ \pi_1(\Sigma-S^g) = \langle x, \hat{x},\partial_{\hat{a}_{i+1}}
  \rangle \] if $x=a_i$ or $x=\hat{a}_i$.
\end{lem}
\begin{proof}
  Here, only the first part of the proof of
  Lemma~\ref{lem:explicit-good-subsurfaces-case1} is necessary; the
  desired good subsurface is the subsurface $S_0$ constructed in
  Lemma~\ref{lem:obvious-good-subsurface}.
\end{proof}

\subsection{\blas paths}\label{sec:blas}

In this section we begin to discuss the paths we use as the basis for
all of our constructions.  The picture to have in mind is that we
build paths by concatenating arational paths in \emph{different}
copies of $\pml$, joining them at points which fail to be arational in
a very controlled way.

More formally, we say a path $p:[0,1] \to \partial CV_n$ is a \blas
path if there is a finite set $B_p \subset p([0,1])$ so that
\begin{itemize}
\item Every point on $p([0,1])\setminus B_p$ is an arational tree. 

\item If $T \in B_p$, then there is an identification $F_n =
  \pi_1(\Sigma)$ of the free group with a surface with one boundary
  component (where $\Sigma = \Sigma_g$ if $n=2g$ is even, and $\Sigma
  = N_{2g+1}$ otherwise), and the following holds: $T$ is the dual
  tree to the stable foliation of a pseudo-Anosov mapping class
  $\psi_T$ supported on a subsurface $S$.

\item For any $T\in B_p$ there is a neighborhood $\mathcal{N}(T)$ of
  $T$ in $\partial CV_n$, so that $\mathcal{N}(T) \cap
  p([0,1])\setminus T$ has two connected components, $\gamma_1,\,
  \gamma_2$.

  For each $\gamma_i$ there is a path $\xi_i$ so that $\xi_i\cup
  \psi_T \xi_i$ is a path and $\gamma_i=\cup_{k=0}^\infty \psi_T ^k
  \xi_i$.
\end{itemize}

Using Theorem~\ref{thm:ch} and
Theorem~\ref{thm:weak-nonorientable-connectivity}, we show:
\begin{prop}\label{prop:BLAS}
  Let $T_s,T_e$ in $\bout$ be two surface type arational trees, dual
  to uniquely ergodic laminations (or, in the nonorientable case, in
  the set $\mathcal{P}$ from
  Theorem~\ref{thm:weak-nonorientable-connectivity}).

  Then $T_s, T_e$ can be connected by a \blas path. Moreover, if one
  prescribes a chain of adjacent $\pml$s connecting the $\pml$ on
  which $T_s$ lies with the one on which $T_e$ lies, we may assume the
  \blas path travels exactly through that chain of $\pml$s in exactly
  the same order.
	
  In addition, for any finite set $F$ of arationial trees not
  containing $T_s, T_e$, the path may be chosen to be disjoint from
  $F$.
\end{prop}
\begin{proof}
  It suffices to show the proposition in the case where $T_s, T_e$ lie
  in adjacent copies of $\pml$, say $\pml(\Sigma), \phi\pml(\Sigma)$.
  Furthermore, choose a partial pseudo-Anosov $\psi$ on the good
  subsurface of $\phi$, and observe that its stable lamination
  $\lambda^+$ is not dual to any tree contained in the finite set $F$
  (since $\psi$ is a partial pseudo-Anosov).
	
  Observe that there is a neighbourhood $U$ of $\lambda^+$ in
  $\pml(\Sigma)$ so that $\psi(U) \subset U$.  We may assume further
  that no lamination in $U$ is dual to any tree in $F$.  Pick any
  uniquely ergodic lamination $\lambda_0$, and a path $\gamma$ of
  uniquely ergodic laminations joining $\lambda_0$ to
  $\psi(\lambda_0)$. In particular, for $k\to \infty$, the paths
  $\psi^k(\gamma)$ converge to $\lambda^+$ (as no point on it is a
  lamination disjoint from the boundary of the active subsurface of
  $\psi$). Thus, we can choose an number $K>0$ so that $\psi^k(\gamma)
  \subset U$ for all $k \geq K$. In particular,
  \[ \psi^K(\gamma)\ast\psi^{K+1}(\gamma)\ast\cdots \] can be
  completed to a path $p$ joining $\psi^K(\lambda_0)$ to $\lambda_+$
  contained in $U$. In particular, no point on $p$ is dual to a tree
  in $F$.
	 
  Now, by Theorem~\ref{thm:ch} or
  Theorem~\ref{thm:weak-nonorientable-connectivity}, there is a path
  joining $\lambda_0$ to $\psi^K(\lambda_0)$, and so that no point on
  it is dual to a tree in $F$. The concatenation of these two paths is
  the desired path.
\end{proof}

\section{Avoiding problems}\label{sec:avoid}
Recall from Section~\ref{3.1} that $K_E$ is the set of trees in
$\partial CV_n$ where a free factor $E$ does not act freely and
simplicially.  The main point of this section is to prove the
following local result:
\begin{thm}\label{thm:avoid} 
  Suppose that the rank of the free group $F_n$ is at least
  $\rank$. Let $E \in \F$ be any proper free factor. Assume that
\begin{itemize}
\item $\pml_{\sigma_1}$, $\pml_{\sigma_2}$ are adjacent copies of
  $\pml$ (i.e. differ by applying an adjusted move $\phi$),
\item $\lambda_1\in \pml_{\sigma_1}$, $\lambda_2\in \pml_{\sigma_2}$
  are uniquely ergodic (respectively, in the set $\mathcal{P}$ from
  Theorem~\ref{thm:weak-nonorientable-connectivity}).
\item $\psi$ is a partial pseudo-Anosov mapping class on $\Sigma$
  supported on the good subsurface $S\subset\Sigma_1$ of the adjusted
  move $\phi$ as in Section~\ref{sec:moves}.
\end{itemize}
Then there exists a \blas path $q:[0,1] \to \partial CV_N$ joining the
dual trees $T_1, T_2$ of $\lambda_1, \lambda_2$, so that 
\begin{enumerate}
\item \label{C:avoid} There is a number $m>0$ so that for all $k=mr \geq 0$ we have $\psi^k(q[0,1])\cap K_E=\emptyset$.

\item\label{C:contract} $\psi^k(q[0,1])$ converges to the stable
  lamination of $\psi$ as $k$ goes to infinity.
\end{enumerate}
Also observe that $\psi^k(q[0,1])$ still connect $\pml_{\sigma_1}$,
$\pml_{\sigma_2}$, since $\psi$ commutes with the adjusted move.
\end{thm}

We begin by describing how to construct a relation in
$\mathrm{Out}(F_n)$ (which will serve as a ``combinatorial skeleton''
for a \blas path), so that overlap and containment problems of $A$ can
be avoided at each step. It is important for the strategy that
containment problems can be solved first.

Before we begin in earnest, we want to briefly discuss the setup for
the rest of this section.  We begin by fixing once and for all an
identification $\pi_1(\Sigma) \simeq F_n$, a corresponding ``basepoint
copy'' $\pml(\Sigma)$, and a standard geometric basis $\mathcal{B}$
(as in Section \ref{sgb}). In building \blas paths, we will always use
$\mathrm{Out}(F_n)$ to move the current copy of $\pml$ to this
basepoint copy, and work there.

Suppose we have a relation
\[ \phi = \phi_1\circ\cdots\circ\phi_l \] in
$\mathrm{Out}(F_n)$. Associated to this we have a sequence of
consecutively adjacent copies of $\pml$
\[ \pml(\Sigma), \phi_1\pml(\Sigma), \phi_1\phi_2\pml(\Sigma), \ldots,
\phi_1\circ\cdots\circ\phi_l\pml(\Sigma) = \phi\pml(\Sigma). \] The
$i$--th adjacency of this sequence, i.e. between
$\phi_1\circ\cdots\circ\phi_{i-1}\pml(\Sigma)$ and
$\phi_1\circ\cdots\circ\phi_l\pml(\Sigma)$ is the image under
$\phi_1\circ\cdots\circ\phi_{i-1}$ of the adjacency given by the
adjusted move $\phi_i$. This motivates the following
\begin{defin}
  Let \[ \phi = \phi_1\circ\cdots\circ\phi_l \] be a relation in
  $\mathrm{Out}(F_n)$. We say that \emph{$E$ is a overlap or
    containment problem at the $i$--th step of the relation}, if
  $(\phi_1\circ\cdots\circ\phi_{i-1})^{-1}(E)$ is an overlap or
  containment problem for the adjusted move $\phi_i$.
\end{defin}
Our strategy will be to replace adjusted moves $\phi$ by such
relations, so that a given factor $E$ is not an overlap or containment
problem at any stage of the relation. Recall that later, such
relations will guide the construction of \blas paths in which $E$ will
no longer be an obstruction to arationality at any point.

The details of this approach are involved and so we now state the two
main ingredients, eliminating containment problems (Proposition
\ref{prop:solving-containment}) and eliminating overlap problems
(Proposition \ref{prop:solving-overlap}), and prove Theorem
\ref{thm:avoid} conditional on these results. (We will prove
Proposition \ref{prop:solving-containment} in the next subsection,
prove Proposition \ref{prop:solving-overlap} in the orientable case in
the following subsection and prove Proposition
\ref{prop:solving-overlap} in the non-orientable case in Appendix
\ref{sec:check}.)

\begin{prop}\label{prop:solving-containment}
  Suppose that $E < F_{n}$ is a proper free factor, and $\phi$ is an
  adjusted move. If $E$ is a containment problem for $\phi$, then
  there is a relation
  \[ \phi = \phi_1\cdots\phi_l\] with the property that $E$ is not a
  containment problem at any stage of the relation. 
\end{prop}

\begin{prop}\label{prop:solving-overlap}
  Suppose that $\phi$ is an adjusted move, and $E$ is a free factor
  which is not a containment problem for $\phi$. Then there is a
  relation
  \[ \phi = \phi_1\cdots\phi_l, \] where each $\phi_i$ is an adjusted
  move, and so that $E$ is neither a containment nor an overlap
  problem at any stage of the relation.
		
  If $\psi$ is a partial pseudo-Anosov supported on the good
  subsurface of the basic move $\phi$ then there is an number $k>0$ so that for
  any $n\geq 0$ the conjugated relation
  \[ \phi = \psi^{-kn}\phi_1\cdots\phi_l\psi^{kn} \] has the same
  property.
\end{prop}
    
\begin{proof}[Proof of Theorem \ref{thm:avoid} assuming Propositions
  \ref{prop:solving-containment} and \ref{prop:solving-overlap}]

  Let $\phi$ be the adjusted move by which the adjacent $\pml$s
  differ.  Using first Proposition~\ref{prop:solving-containment}, and
  then Proposition~\ref{prop:solving-overlap} (to each adjusted move
  appearing in that first relation), we can replace $\phi$ by a
  relation
  \[ \phi = \phi_1\cdots\phi_l \] so that in each stage $E$ is neither
  an overlap nor containment problem.

  Now we will use Proposition~\ref{prop:BLAS} to find a \blas path $q$
  going through the chain of {\pmls} defined by the relation, that
  is $$\pml(\Sigma_1), \phi_1\pml(\Sigma_1), \ldots,
  \phi_2\cdots\phi_l\pml(\Sigma_1),
  \phi\pml(\Sigma_1)=\pml(\Sigma_2)$$

  We claim that this path $q$ itself
  is disjoint from $K_E$.  Observe that it suffices to check this at
  all of the points of the
  \blas path which are not minimal, and therefore dual to the
  stable lamination of a partial pseudo-Anosov. For these finitely
  many points, Lemma~\ref{lem:leaving the K with pAs} applies (exactly
  because we have
  guaranteed that $E$ is not an overlap or containment problem at any
  stage of the relation by Proposition~\ref{prop:solving-overlap}),
  and shows that they are outside $K_E$ as well.
          
  To prove Property~(1), i.e. that $\psi^kq$ is disjoint from $K_E$,
  observe that $\psi^kq$ can be thought of as a \blas path guided by
  the conjugated relation
  \[ \phi = \psi^{-k}\phi_1\cdots\phi_l\psi^k, \] to which (by the
  last sentence of Proposition~\ref{prop:solving-overlap}) the same
  argument applies.
         
  Finally, Property~(2) is implied by Proposition~\ref{NS}, assuming
  that the \blas path is constructed to never intersect the unstable
  lamination of $\psi$ -- which can be done by avoiding a single
  lamination in the construction of the \blas path, and is therefore
  clearly possible.
\end{proof}

\subsection{Containment problems: Proof of Proposition
  \ref{prop:solving-containment}}
       
The proof of this proposition relies on the construction of a curve
with certain properties.
\begin{lem}\label{lem:good-curve-containment}
  Suppose we are given elements $z,w,a$ of our chosen standard
  geometric basis $\mathcal{B}$, all of which are two-sided, and no
  two of which are linked. Then there is a two-sided curve $\delta$
  with the following properties:
  \begin{enumerate}[i)]
  \item $\delta$ intersects $a$ in a single point,
  \item $\delta$ does not cross the band corresponding to $w$ (i.e. in
    the fundamental group, $\delta$ can be written without the letter
    $w$),
  \item There is another, unrelated two-sided letter $e$, so that $\delta$ 
  	does not cross the band corresponding to $e$ (i.e. $\delta$ can be written without $e$).
  \item $\delta$ intersects each basis loop of $\mathcal{B}$ at most
    two points, one on an initial and one on a terminal segment,
  \item in homology we have $[\delta] = \pm[\hat{a}] \pm [z]$.
  \end{enumerate}
\end{lem}
Indeed, the desired curve can be found as the concatenation of
$\hat{a}$ and $z$ as in Figure~\ref{fig:lemma-4-5}.
\begin{figure}
	\includegraphics[width=0.7\textwidth]{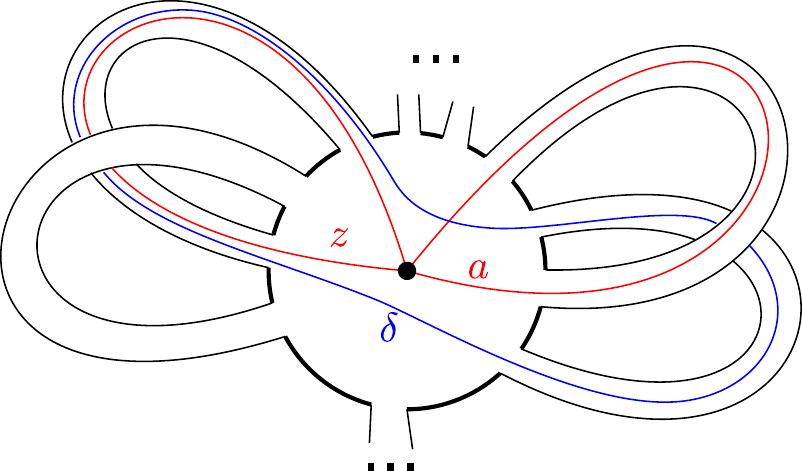}
	\caption{The curve from Lemma~\ref{lem:good-curve-containment}}
	\label{fig:lemma-4-5}
\end{figure}

\begin{proof}[{Proof of Proposition~\ref{prop:solving-containment}}]
  We begin by considering the basic move $\phi = \lambda_{x,y}$ or
  $\phi=\rho_{x,y}$, and assume that $E$ is a containment problem for
  the basic move. Let $z$ be a basis letter so that $[z] \notin
  H_1(E)$. Such a letter exists, since $E$ is a proper free
  factor. Also observe that by Lemma
  \ref{lem:explicit-good-subsurfaces} (since we assume that $E$ is a
  containment problem), $z$ is one of $x, \hat{x}, y, \hat{y}$ or the
  one-sided letter $n$. Next, choose $w$ an unrelated, two-sided
  letter, i.e. different from all of $x, \hat{x}, y, \hat{y}$, and not
  linked with $z$; in particular it is good for $\phi$ (i.e. contained
  in the good subsurface) by
  Lemma~\ref{lem:explicit-good-subsurfaces}.  Let $a$ be a two-sided
  basis element so that $a,\hat{a}$ are good for $\phi$ and
  $\lambda_{z,w}$, and are distinct from any of the previously chosen
  letters.
		
  Now, let $\delta$ be the curve guaranteed by
  Lemma~\ref{lem:good-curve-containment}. Define an auxiliary adjusted
  move
  \[ \theta = T_\delta \lambda_{w,z} T^{-1}_\delta. \]
                
  We observe that $\theta$ has the following properties:
  \begin{enumerate}
  \item $\theta$ fixes every basis element, except possibly $w$.
			
    Namely, by property~iv) of the curve $\delta$, the Dehn twist
    $T_\delta$ acts on each basis element by conjugation, left, or
    right multiplication by a word obtained by tracing $\delta$
    starting at a suitable point. By property~ii), none of these words
    involve the letter $w$. Thus, the letter $w$ appears only in the
    image of $w$ in $T_\delta(\mathcal{B})$. Since $\lambda_{w,z}$
    fixes all letters except $w$, the claim follows,
				
  \item The word $\theta(w)$ does not involve the letter $e$ from
    Lemma~\ref{lem:good-curve-containment}.
			
    This follows from the description of the action of $T_\delta$
    above, together with property~iii) of $\delta$.
			
  \item $[\theta(w)] = [w]+[z] \notin H_1(E)$.
			
    Namely, by property~v) of $\delta$, we have $[T_\delta(w)]=[w]$
    (as the algebraic intersection between $\delta$ and $w$ is
    zero). Thus, $[\lambda_{w,z}T_\delta(w)]=[w]+[z]$. Since the
    algebraic intersection number between $\delta$ and $z$ is also
    zero, we therefore have $[\theta(w)] = [w]+[z]$. Since $w$ is good
    for $\lambda_{x,y}$, and we assume that $E$ is a containment
    problem for $\phi$, we have $[w] \in H_1(E)$.  Thus, since
    $[z]\notin H_1(E)$, the claim follows.
				
  \item $E$ is not a containment problem for $\theta$.

    Since $a$ and $\hat{a}$ are good for $\phi$, and we assume that
    $E$ is a containment problem, it follows that $[a]+[\hat{a}]\in
    H_1(E)$.  On the other hand, the good subsurface of $\theta$ can
    be obtained from the good subsurface of $\lambda_{w,z}$ by
    applying $T_\delta$.  Recall that $a$ is good for $\lambda_{w,z}$,
    and thus $T_\delta(a)$ is good for $\theta$. In homology we have
    $[T_\delta(a)] = [a]+[\hat{a}]+[z]$. Since $[a]+[\hat{a}]\in
    H_1(E)$ but $[z] \notin H_1(E)$ this implies $[T_\delta(a)]\notin
    H_1(E)$.
				
  \item The automorphism $\phi^{-1}\theta^{-1}\phi\theta$ fixes all
    basis elements except possibly $w$.
			
    This is an immediate consequence of the fact that $\theta$ fixes
    all letters except $w$, which is distinct from $x,y$ (which are
    the only letters involved in $\phi$).
  \end{enumerate}
  This shows that there is a relation of the form
  \begin{equation}
    \label{eq:containment-relation1}
    \phi = \theta \phi D_w \theta^{-1}, 
  \end{equation}
  where $D_w$ is a product of basic moves of the form $\phi^\pm_{w,q}$
  acting on the letter $w$, and no $q$ is equal to $e$. Hence, $D_w$
  commutes with $\lambda_{e,z}$, and we obtain a relation
  \begin{equation}
    \label{eq:containment-relation2}
    \phi = \theta \phi \lambda_{e,z} D_w \lambda^{-1}_{e,z} \theta^{-1}.
  \end{equation}
  Observe that $D_w$ may be identity; in which case we also remove the
  $\lambda_{e,z}$--terms from this relation.
			
  We now check that this relation has no containment problems, as
  claimed. Recall that we have to check this left-to-right. 
  \begin{itemize}
  \item \emph{The initial $\theta$ move:} This is (4) above.
			
  \item \emph{The move $\phi$:} Here, we observe that $w$ is good for
    the two possibilities $\lambda_{x,y}$ and $\rho_{x,y}$ that $\phi$
    can be. On the other hand, we have that $\theta(w) \notin E$ by
    (3) above. Thus, $w \notin \theta^{-1}E$, and the claim follows.
			
  \item \emph{The auxiliary move $\lambda_{e,z}$:} This move has $w$ as
    a good letter, and $\theta\lambda_{x,y}(w) = \theta(w) \notin E$
    just like above.
			
  \item \emph{The moves in $D_w$:} These basic moves all have $e$ as a
    good letter. Now, we have $\theta\phi\lambda_{e,z}(e) = \theta(ez)
    = ez$, and $[e]+[z]\notin H_1(E)$.
				
  \item \emph{Undoing the auxiliary move $\lambda_{e,z}$:} This move
    has $w$ as a good letter, and $\theta \phi \lambda_{e,z} D_w (w) =
    \theta \phi D_w (w) = \phi\theta(w)$ by
    Equation~(\ref{eq:containment-relation1}). But, $[ \phi\theta(w) ]
    = [z] + [w] \notin H_1(E)$. Thus, $ \theta \phi \lambda_{e,z} D_w
    (w) \notin E$.
			
  \item \emph{The final $\theta^{-1}$ move:} Here, we use again (as in
    (4) above) that $T_\delta(a)$ is good for $\theta$. We have
    $[T_\delta(a)] = [a]+[\hat{a}]+[z]$, and thus
    \[ [\theta \phi \lambda_{e,z} D_w \lambda^{-1}_{e,z} T_\delta(a)]
    = [a]+[\hat{a}]+[z] \] which does not lie in $H_1(E)$. Thus, we do
    not have a containment problem.
  \end{itemize}
	
  \medskip Finally, we need to deal with a basic inversion move $\phi
  = \iota_x$. Again, assume that $E$ is a containment problem. Thus,
  either $[x]\notin H_1(E)$ or $[\hat{x}]\notin H_1(E)$. In the former
  case, we start with the relation
  \[ \iota_x = \lambda_{x, z}\iota_x\rho_{x, z}, \] and in the latter
  case, with a relation of the type
  \[ \iota_x = \rho_{\hat{x}, z} \iota_x \rho_{\hat{x}, z}^{-1} \] for
  an unrelated letter $z$. Observe that in either case,
  $\lambda_{x,z}^{-1}E$ or $\rho_{\hat{x}, z}^{-1} E$ is not a
  containment problem for $\iota_x$.
                
  Now, suppose we are in the former case (the other one is completely
  analogous). Then, by the first part of the proof, we can find a
  relation
  \[ \lambda_{x, z} = \phi_1\circ\cdots\circ\phi_i, \] so that $E$ is
  not containment problem at any step of this relation. Hence, in the
  relation
  \[ \iota_x = \phi_1\circ\cdots\circ\phi_i\iota_x\rho_{x, z}, \] the
  factor $E$ is now not a containment problem at the first $i+1$
  steps. Now, appealing to the first part of the proof again, we can
  find a relation
  \[ \rho_{x, z} = \phi'_1\circ\cdots\circ\phi'_j, \] so that
  $\lambda_{x, z}\iota_x(E)$ is not a containment problem at any
  step. Then, the relation
  \[ \iota_x =
  \phi_1\circ\cdots\circ\phi_i\iota_x\phi'_1\circ\cdots\circ\phi'_j, \]
  has the desired properties.
                
  \smallskip Finally, we discuss adjusted
  moves. Suppose $\phi = \varphi^{-1}\phi_0\varphi$ is the conjugate
  of a basic move by a mapping class group element $\varphi$ of
  $\Sigma$. We then apply the Proposition to the basic move $\phi_0$
  and the factor $\varphi^{-1}E$, and conjugate the resulting by
  $\varphi$. This resulting relation (of adjusted moves) then has the
  desired property.
\end{proof}

\subsection{Overlap Problems: Proof of Proposition
  \ref{prop:solving-overlap}}
The proof of Proposition \ref{prop:solving-overlap} is technically
very involved, and the details vary depending on the nature of the
move $\phi$. First, observe that exactly as in the last paragraph of
the proof of Proposition~\ref{prop:solving-containment}, the case of
adjusted moves can be reduced to the case of basic moves.  The rest of
this section is therefore only concerned with basic moves.
	
For basic moves, we will (again, similar to the proof of
Proposition~\ref{prop:solving-containment}), reduce the case of invert
moves to the case of Nielsen moves. For Nielsen moves the relation
claimed in the proposition will be constructed using the following two
lemmas, which construct a ``preliminary relation'', and ``short
relations'':

\begin{lem}[Preliminary Relation]\label{lem:preliminary-relation}
  Under the assumptions of Proposition~\ref{prop:solving-overlap}, if
  $\phi$ is not an invert move, there is a relation
  \[ \phi = \phi_1\cdots\phi_r, \] so that if $B_i$ is the bad
  subgroup for $\phi_i$ and $B$ is the bad subgroup of $\phi$, then
  (up to conjugation) the intersection
  \[ \phi_1\cdots\phi_{i-1}B_i \cap B \] is trivial or (up to
  conjugacy) contained in a ``problematic'' group of the form
  $\<\partial\>$,$\langle x_i\rangle$ or $\langle
  x_i, \partial\rangle$ for some word $x_i$, where $\partial$ is the
  boundary component of the surface (compare
  Section~\ref{lem:obvious-good-subsurface})\footnote{We remark that, in general, the intersection of two subgroups up to conjugacy could be a collection of conjugacy classes of subgroups. Here, the intersection always consists of at most one such conjugacy class.}.
\end{lem}
\begin{lem}[Short Relations]\label{lem:short-relation}
  For all indices $i$ in Lemma~\ref{lem:preliminary-relation} where
  the collection of problematic groups is nonempty, there is a
  relation
  \[ \phi_i = \rho_i \phi_i \rho_i^{-1}, \] with the properties
  \begin{enumerate}
  \item No conjugacy class of the problematic group
    $\<\partial\>$,$\langle x_i\rangle$ or $\langle
    x_i, \partial\rangle$ is contained in the bad subsurface of
    $\rho_i$ and in $E$,
  \item and also $E \cap \phi_1\cdots\phi_{i-1}\rho_i B_i$ is trivial.
  \end{enumerate}
\end{lem}
The proofs of these lemmas construct the desired relations fairly
explicitly, and involve lengthy checks. Before we begin with these
proofs, we explain how to use the lemmas in the proof of
Proposition~\ref{prop:solving-overlap}. We need two more tools: first,
the following immediate consequence of
Proposition~\ref{prop:good-and-bad} (this corollary is the reason why
in our strategy, containment problems need to be solved before overlap
problems). Observe that the Proposition~\ref{prop:good-and-bad} may be used, since
the fundamental group of good subsurfaces are free factors of rank $>\rank-5$, while
the fundamental groups of the complements of good subsurfaces have rank at most $5$, and
so the former can never be contained in the latter up to conjugacy.
\begin{cor}\label{cor:making-overlap-cyclic}
  Suppose $\phi$ is a basic move with bad subgroup $B=\pi_1(\Sigma-S^g)$,
  $\psi$ the commuting partial pseudo-Anosov supported on the good
  subsurface on $\phi$, and $E$ any free factor. Suppose that $E$ is
  not a containment problem for $\phi$. Then there is a number $k$
  with the following property.
    	
  If \[\phi = \phi_1\cdots \phi_l\] is a relation, and $M=km$ is large
  enough, then the conjugated relation
  \[ \phi =
  (\psi^M\phi_1\psi^{-M})(\psi^M\cdots\psi^{-M})(\psi^M\phi_l\psi^{-M}), \]
  has the following property: at every step of the relation, an
  overlap problems with $E$ occurs exactly if $E$ contains elements
  conjugate into $B \cap B_i$, the intersection of the bad subgroups
  of the original relation.
\end{cor}
Second, we need the following lemma, guaranteeing that within a
relation which replaces a move without containment problem, no new
containment problems are created.
\begin{lem}\label{lem:removing-inner-containment}
  Assume that $E$ is not a containment problem for $\phi$, and that
  \[ \phi = \phi'_1\circ \cdots \circ \phi'_R \] is a relation. Then,
  there is a number $k>0$ so that conjugating the relation by any
  large power $kN$ of a partial pseudo-Anosov $\psi$ supported on the
  good subsurface of $\phi$, we can guarantee that $E$ is not a
  containment problem at any stage of the resulting relation
  \[ \phi = \psi^{kN}\circ \phi'_1\circ \cdots \circ \phi'_R
  \circ\psi^{-kN}. \]
\end{lem}
\begin{proof}
  Let $G_i$ be the fundamental group of the good subsurface at the
  $i$--th step of the relation, and note that it is a free
  factor. After conjugating the relation by $\psi^{-n}$, this good
  free factor becomes $\psi^{-n}G_i$.

  Now, recall that the intersection of the free factor $G_i$ with the
  bad subgroup $B$ of $\phi$ is a free factor of $B$.  In particular,
  since the rank of the bad subgroups is at most $5$, but the good
  free factor $G_i$ has rank strictly larger than $5$, it cannot be
  completely contained in $B$.  In other words, there is some element
  $g_i \in G_i$ which intersects the subsurface in which $\psi$ is
  supported.

  Now, apply Proposition~\ref{prop:good-and-bad} for the factor $E$,
  and $B'=G_i$. Since Conclusion~(1) of that proposition is impossible
  here (as $E$ is not a containment problem for $\phi$), we see that
  for large $n=kN$, the only classes contained in $\psi^nE$ and $G_i$
  are contained in the bad subsurface fundamental group $B$. Since
  $g_i$ is not contained in $B$, this shows that $E$ is not a
  containment problem.
\end{proof}
We are now ready for the proof of the central result of the section.
\begin{proof}[{Proof of Proposition~\ref{prop:solving-overlap}}]
  First, we prove the proposition for basic Nielsen moves.  We first
  apply Lemma~\ref{lem:preliminary-relation} to obtain the preliminary
  relation, and then Lemma~\ref{lem:short-relation} to each index it
  applies to. We then have a relation
  \[ \phi = \phi_1\cdots\phi_{i-1}(\rho_i \phi_i
  \rho_i^{-1})\phi_{i+1}\cdots\phi_r, \] which still may have overlap
  problems (in particular, since $\rho_i$ may have other, ``new''
  overlap problems, but at least these will be guaranteed to be
  outside the intersection $E \cap \phi_1\cdots\phi_{i-1}B_i$).
		
  Now, for all $N>0$, Lemma~\ref{lem:removing-inner-containment} shows
  that for conjugated relation
  \[ \phi = \psi^{kN}\circ\phi_1\cdots\phi_{i-1}(\rho_i \phi_i
  \rho_i^{-1})\phi_{i+1}\cdots\phi_r\circ\psi^{-kN} \] the factor $E$
  is not a containment problem at any stage.
		
  Hence, we can apply Corollary~\ref{cor:making-overlap-cyclic} to the
  inserted ``small relations'', further replacing them by conjugates
  of suitable powers of the associated pseudo-Anosov of $\phi_i$,
  yielding a relation of the form
  \[ \phi =
  \psi^{kN}\circ\phi_1\cdots\phi_{i-1}(\psi_i^{-M_i}\rho_i\psi_i^{M_i}
  \phi_i
  \psi_i^{-N}\rho_i^{-1}\psi_i^{N})\phi_{i+1}\cdots\phi_r\circ\psi^{-kN}. \]
  Since the preliminary relation had no containment problems at any
  stage, Lemma~\ref{lem:removing-inner-containment} can again be
  applied to guarantee that there replacements also do not have
  containment problems.
		 
  Furthermore, Corollary~\ref{cor:making-overlap-cyclic} implies that
  for this relation any overlap problems can only occur within the
  intersection of the bad factor $B_i$ of $\phi_i$ and the bad factor
  of the move $\rho_i$.  Now, by construction, there are no conjugacy
  classes that both of those factors have in common with
  $(\phi_1\cdots\phi_{i-1})^{-1}E$. Hence, this final relation indeed
  solves all containment and overlap problems.
		
  If we conjugate this relation by a further power of $\psi^k$, then
  Lemma~\ref{lem:removing-inner-containment} shows that in the
  resulting relation $E$ is still no containment problem at any stage,
  and Corollary~\ref{cor:making-overlap-cyclic} shows the same for
  overlap problems. This shows the proposition for
  basic Nielsen moves.
		
  \smallskip Now, let $\phi=\iota_x$ be a basic invert move. Since the
  subgroup generated by basic Nielsen moves is normal, for any product
  $\alpha$ of basic Nielsen moves there is a product of basic Nielsen
  moves $\beta$, so that
  \[ \iota_x = \alpha \iota_x \beta. \] By choosing $\alpha$ to be a
  large power of a pseudo-Anosov mapping class, we may assume that $\alpha^{-1}E$ is not a
  containment or overlap problem for $\iota_x$.
		 
  Now (similar to the proof of
  Proposition~\ref{prop:solving-containment}), by applying the current
  proposition for basic Nielsen moves, we can write
  \[ \alpha = \alpha_1\circ\cdots\circ \alpha_r, \beta =
  \beta_1\circ\cdots\circ\beta_s \] so that $E$ is not an overlap or
  containment problem at any stage of the first relation, and so that
  $(\alpha\iota_x)^{-1}E$ is not an overlap or containment problem at
  any stage of the second. The resulting relation
  \[ \iota_x = \alpha_1\circ\cdots\circ \alpha_r \iota_x
  \beta_1\circ\cdots\circ\beta_s \] then has the desired property.
\end{proof}
   
To prove Lemmas \ref{lem:preliminary-relation} and
\ref{lem:short-relation} which construct relations, we need to collect
some results on controlling the intersections between finitely
generated subgroups of free groups. These results are basically
standard (see \cite{Stallings}), but we present them in a form useful
for the checks below. Throughout, we denote by $R_n$ the rose labelled
by the elements of out chosen standard geometric basis $\mathcal{B}$.
We identify edge-paths in $R_n$ with words in $\mathcal{B}$.
    	
Suppose we are given a subgroup
\[ A = \langle \alpha_1, \ldots, \alpha_r \rangle \] where each
$\alpha_i$ is a reduced word in our fixed basis $\mathcal{B}$. We
denote by $R_A$ the subdivided rose labelled by the $\alpha_i$, and by
$f:R_A \to R_n$ the graph morphism inducing the inclusion of $A$ as a
subgroup of $F_n$ (recall that graph morphisms map vertices to
vertices, and edges to edges).
    	
Let $\Gamma_A$ be a graph obtained by folding from $R_A$, so that $f$
factors as
\[ R_A \stackrel{p_A}{\to} \Gamma_A \stackrel{g_A}{\to} R_n \] where
$g_A$ is an immersion.

\begin{defin}
  \begin{enumerate}
  \item A subword $w$ of one of the $\alpha_i$ is called a
    \emph{certificate in $\alpha_i$}, if there is an embedded path
    $\gamma_w \subset \Gamma_A$, which lifts to a path
    $\widetilde{\gamma}_i$ in $R_A$ contained in the petal
    corresponding to $\alpha_i$, and representing $w$.
  \item We say that a certificate is \emph{uncancellable} if
    $\gamma_w$ is disjoint from the images of all other petals
    $\alpha_j, j \neq i$ of $R_A$ under $p_A$.
  \item A reduced word $w$ in $\mathcal{B}$ is \emph{impossible in
      $A$}, if the corresponding path in $R_n$ does not lift to
    $\Gamma_A$ (equivalently, there is no path in $\Gamma_A$ labelled
    by $w$)
  \end{enumerate}
\end{defin}
    	
\begin{lem}[Dropping Generators -- Impossible
  Certificates]\label{lem:drop-turn-easy}
  Suppose that
  \[ A = \langle \alpha_1, \ldots, \alpha_r \rangle, \]
  \[ B = \langle \beta_1, \ldots, \beta_s \rangle \] are two subgroups
  (where the $\alpha_i, \beta_j$ are words in a common basis
  $\mathcal{B}$).
    		
  Suppose that $\tau$ is an uncancellable certificate in $\alpha_1$,
  which is impossible in $B$.  Then any conjugacy class contained in
  $A$ and $B$ is also contained in $\langle \alpha_2, \ldots, \alpha_r
  \rangle$
\end{lem}
        
\begin{proof}
  Let $x\in A$ be an element which is not conjugate into $\langle
  \alpha_2, \ldots, \alpha_r \rangle$.  Then, let $\gamma \subset
  \Gamma_A$ be a geodesic representing $x$. Since any loop
  representing $x$ in $R_A$ has to involve $\alpha_1$, and by
  definition of uncancellable certificate, $\gamma$ contains a subpath
  labelled by $\tau$. Thus, the geodesic $g_A(\gamma)$ contains a
  subpath $g_A(\tau)$ which, as $\tau$ is impossible for $B$, is in
  the image of no loop $\gamma' \subset \Gamma_B$ under $g_B$. This
  shows the claim.
\end{proof}
    	
We need a version of the dropping letters lemma which applies when $A$
and $B$ share a generator.
\begin{lem}[Dropping Generators -- Impossible Unique
  Followup]\label{lem:drop-turn-hard}
  We are given two subgroups
  \[ A = \langle \alpha_1, \ldots, \alpha_r, \delta_A \rangle \]
  \[ B = \langle \beta_1, \ldots, \beta_s, \delta_B \rangle, \] where
  the $\alpha_i, \delta_A, \beta_j, \delta_B$ are words in a fixed
  basis $\mathcal{B}$.
    		
  Suppose that
  \begin{enumerate}
  \item $\beta_1$ contains an uncancelable certificate $\tau$,
  \item the only path $\tau_A$ in $\Gamma_A$ which lifts to $\tau$ is
    contained within the geodesic representative $\overline{\delta_A}$
    of the image of $\delta_A$ in the immersed graph $\Gamma_A$,
  \item there is an uncancellable certificate $\tau'$ in $\delta_A$
    whose image immediately follows $\tau_A$,
  \item no path corresponding to a reduced word $\beta_1 b$ (for $b\in B$).
    lifts to a path starting with $\tau\tau'$.
  \end{enumerate}
  Then any conjugacy class in $A$ and $B$ is also contained in
  \[ B = \langle \beta_2, \ldots, \beta_s, \delta_B \rangle. \]
\end{lem}
\begin{proof}
  The proof is very similar to the previous one.  Let $x\in B$ be an
  element which is not conjugate into $\langle \beta_2, \ldots,
  \beta_s, \delta_B \rangle$.  Then, let $\gamma \subset \Gamma_B$ be
  a geodesic representing $x$. Since any loop representing $x$ in
  $R_B$ has to involve $\beta_1$, and by definition of uncancellable
  certificate, $\gamma$ contains a subpath labelled by $\tau$.  Now,
  suppose $\gamma' \subset \Gamma_A$ is a geodesic representing the
  same conjugacy class $x$. Then, $\gamma'$ contains a subpath
  labelled by $\tau$, and by (2) this occurs in $\overline{\delta_A}$
  and is followed by $\tau'$. By uncancelability, $\tau'$ also follows
  $\tau$ in the loop $\gamma$ -- which contradicts (4).
\end{proof}
    	
Finally, we need the following well-known fact.
\begin{lem}[Intersecting with factors]\label{lem:whitehead}
  Let
  \[ \partial = \prod_{i=1}^g [\hat{a}_i, a_i^{-1}] \] or
  \[ \partial = (n\hat{b}b\hat{b}^{-1}nb)\prod_{i=1}^g [\hat{a}_i,
  a_i^{-1}] \] be the boundary of the surface. Then $\partial$ is
  contained in no proper free factor of the free group.
\end{lem}
\begin{proof}
  The claim follows immediately from Whitehead's algorithm
  \cite{Whitehead}, since the Whitehead graph for $\partial$ is a
  single loop in both cases, and therefore has no cut point.
\end{proof}
    
We are now ready to prove the lemmas.
\begin{proof}[{Proof of Lemma~\ref{lem:preliminary-relation}}]
  The construction depends on the nature of the involved letters (one-
  or two-sided, linked with the one-sided or not; as in
  Section~\ref{sec:moves}) of the basic move $\phi$.
          
  Here, we discuss the case of $\phi=\rho_{x,y}$ on an orientable
  surface in detail. The computations for the other cases follow the
  same general approach; we have collected the details in
  Appendix~\ref{sec:check}.
    		
  For ease of notation in this construction, we assume that the order
  of loops in the basis is
  \[ \hat{x}, x, \hat{y}, y, \hat{a}_3, a_3, \ldots \] Thus, the bad
  subgroup is $B=\langle
  y,\hat{y},x^{-1}\hat{x}x,\partial_{\hat{y}}\rangle$, where
  $\partial_{\hat{y}}$ denotes the cyclic permutation of the boundary
  word
  \[ \partial = [\hat{x}, x^{-1}][\hat{y}, y^{-1}]\prod^g_{i>2}
  [\hat{a}_i, a_i^{-1}] \] starting at $\hat{y}$.
    	
  We use the relation
  \[
  \rho_{x,y}=\rho_{\hat{y},u}^{-1}\rho_{y,z}^{-1}\rho_{x,y}\rho_{x,z}\rho_{y,z}\rho_{\hat{y},u}, \]
  where $u = a_4, z = a_6$ and $g\geq 7$.
    	
  \begin{enumerate}[a)]
  \item $\rho_{\hat{y},u}^{-1}$ has bad subgroup
    \[ A= \langle u,\hat{u}, y, \partial_{y^{-1}} \rangle,\] where
    $\partial_{y^{-1}}$ is the cyclic permutation of $\partial$
    beginning with $y^{-1}$. We need to intersect this subgroup with
    \[ B=\langle y,\hat{y},x^{-1}\hat{x}x, \partial_{\hat{y}}
    \rangle \] We begin by finding graphs which immerse into the rose
    with petals corresponding to the basis $\mathcal{B}$, and which
    represent $A$ and $B$.

    We begin with $A$. Here, the starting point is a rose with four
    petals corresponding to the four generators $u,\hat{u},
    y, \partial_{y^{-1}}$. This is not yet immersed, as the petal
    corresponding to $\partial_{y^{-1}}$ begins and ends with segments
    $y^{-1}\hat{y}^{-1}y$ and $\hat{y}$ which can be folded over the
    other petals. The resulting folded petal $\partial_A$ starts with
    $\hat{a}_3$ and ends with $a_g$.  Hence, this resulting graph
    immerses (compare the left side of
    Figure~\ref{fig:immersed-graphs}).

    The immersed graph for $B$ is similarly obtained by first folding
    the first and last segment of the petal labeled by
    $x^{-1}\hat{x}x$ together, and then folding the initial commutator
    $[\hat{y},y^{-1}]$ and last segment $x^{-1}\hat{x}x$ of
    $\partial_{\hat{y}}$ over the rest. We denote by $\partial_B$ the
    image of this folded petal; note that it is still based at the
    same point (compare the right side of
    Figure~\ref{fig:immersed-graphs}).

    \begin{figure}
      \includegraphics[width=0.7\textwidth]{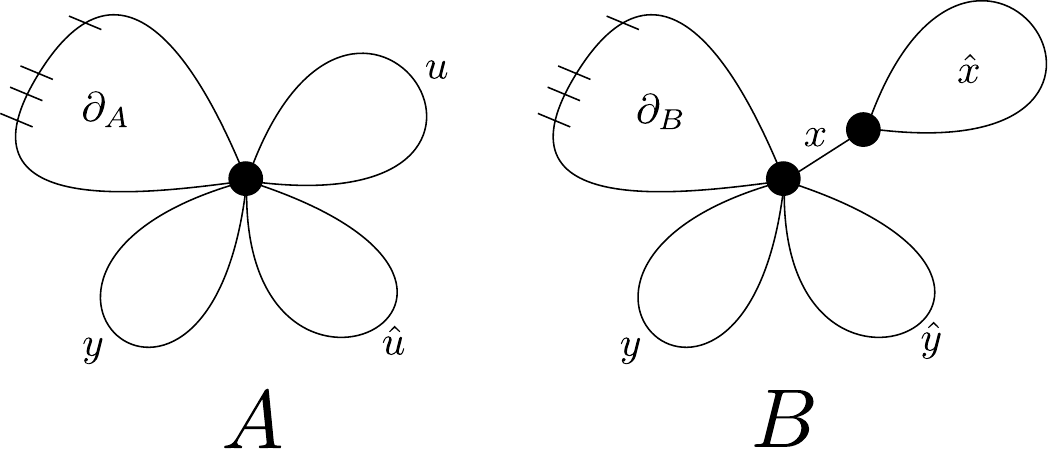}
      \caption{The immersed graphs for the intersection in step a)}
      \label{fig:immersed-graphs}
    \end{figure}
    To compute the conjugacy classes in the intersection of these
    groups, we begin by using Lemma~\ref{lem:drop-turn-hard} with
    $\tau = u$ as the input path. Observe that it is indeed
    uncancellable in $A$, and appears in $B$ only in the petal
    $\partial_B$.

    Since $u = a_4$ and the rank is at least $6$, the petal
    $\partial_B$ will contain a subpath labelled $[\hat{a}_4,
    a_4^{-1}][\hat{a}_5, a_5^{-1}]$. We let $\tau'$ be the
    (uncancellable) path $a_4[\hat{a}_5, a_5^{-1}]$ following $u$ in
    this subpath.

    Observe that this it is impossible to achieve such a path in $A$
    starting with $u$, since $a_5$ appears only in the interior
    $\partial_A$. Hence, Lemma~\ref{lem:drop-turn-hard} applies, and
    any conjugacy class contained in $A$ and $B$ is in fact also
    contained in
    \[ A' = \langle \hat{u}, y, \partial_{y^{-1}} \rangle. \] Hence,
    we now aim to compute the intersection of $A'$ and $B$ using the
    same method. The immersed graph for $A'$ is obtained by simply
    deleting the petal labeled $u$ from the graph for $A$.  We can
    then argue exactly as above (with the input path $\tau = \hat{u}$)
    to also drop the generator $\hat{u}$, and find that any conjugacy
    class common to $A$ and $B$ is also contained in
    \[ A'' = \langle y, \partial_{y^{-1}} \rangle. \] Observe that
    this rank-$2$ group is indeed contained in both $A$ and $B$, and
    so it is the full intersection. Since it has the desired form, we
    are done with this step.
   			
  \item $\rho_{y,z}^{-1}$ has bad subgroup $A_2 = \langle
    z,\hat{z},y^{-1}\hat{y}y, \partial_{\hat{a}_3}\rangle$, which we
    need to intersect with
    \[ \rho_{\hat{y},u}B=\langle y,\hat{y}u,x^{-1}\hat{x}x,
    \rho_{\hat{y},u}\partial_{\hat{y}}\rangle. \]

    For this intersection, we need to take some care of the order of
    simplifications. We begin by observing that the path $\hat{y}u$,
    which corresponds to a petal of the immersed graph of
    $\rho_{\hat{y},u}B$ is impossible in $A_2$ -- the only generator
    which contains $u=a_4$ at all is $\partial_{\hat{a}_3}$, and there
    it is never directy adjacent to $y$. Hence, by
    Lemma~\ref{lem:drop-turn-easy} we may replace $\rho_{\hat{y},u}B$
    by
    \[ \langle y,x^{-1}\hat{x}x,
    \rho_{\hat{y},u}\partial_{\hat{y}}\rangle. \] Now, we can further
    remove $\rho_{\hat{y},u}\partial_{\hat{y}}$, as it also contains
    $\hat{y}u$ as a subword (observe that this would have been
    impossible as the first step, since this subword was folded over
    the petal $\hat{y}u$ in the original immersed graph). Now, we need
    to compare
    \[ \langle y,x^{-1}\hat{x}x\rangle \quad\mbox{ and } A_2 = \langle
    z,\hat{z},y^{-1}\hat{y}y, \partial_{\hat{a}_3}\rangle.\] From the
    latter, we can drop $\partial_{\hat{a}_3}$ since it clearly
    contains uncancellable subwords which are impossible in the former
    (again, using Lemma~\ref{lem:drop-turn-easy}). Then, it is easy to
    see that the remaining groups have no conjugacy classes in common
    (by drawing immersed graphs representing them, or further applying
    Lemma~\ref{lem:drop-turn-easy}).
    		
  \item $\rho_{x,y}$ has bad subgroup $B=\langle
    y,\hat{y},x^{-1}\hat{x}x,\partial_{\hat{y}}\rangle$ and we need to
    intersect with
    \[ \rho_{y,z}\rho_{\hat{y},u}B=\langle yz,
    \hat{y}u,x^{-1}\hat{x}x,
    \rho_{y,z}\rho_{\hat{y},u}\partial_{\hat{y}} \rangle. \] The
    argument is similar to b). We first focus on the generators $yz,
    \hat{y}u$ of $\rho_{y,z}\rho_{\hat{y},u}B$. Using
    Lemma~\ref{lem:drop-turn-easy} we can drop these in order to
    compute the intersection (as these certificates are impossible
    $B$). After that is done, we can then also further drop
    $\rho_{y,z}\rho_{\hat{y},u}\partial_{\hat{y}}$ from
    $\rho_{y,z}\rho_{\hat{y},u}B$ using Lemma~\ref{lem:drop-turn-easy}
    again, as $yz$ or $\hat{y}u$ are now certificates (after the
    previous step, these survive in the immersed graph), which are
    impossible in $B$. Hence, the intersection is $\< \hat{x} \>$.

  \item $\rho_{x,z}$ has bad subgroup $A_3 = \langle
    z,\hat{z},x^{-1}\hat{x}x,\partial_{\hat{y}}\rangle$, and we need
    to intersect with
    \[
    \rho_{x,y}^{-1}\rho_{y,z}\rho_{\hat{y},u}B=\langle
    yz,\hat{y}u,yx^{-1}\hat{x}xy^{-1},
    \rho_{x,y}^{-1}\rho_{y,z}\rho_{\hat{y},u}\partial_{\hat{y}}\rangle.\]
    We begin by dropping $\hat{y}u$ from the latter, since it is
    impossible in $A_3$. Afterwards, we can also drop
    $\rho_{x,y}^{-1}\rho_{y,z}\rho_{\hat{y},u}\partial_{\hat{y}}$
    (since it also contains $\hat{y}u$, and this subpath is now
    certainly not folded over anymore, as above). After this, we can
    remove $\partial_{\hat{y}}$ from $A_3$ since it contains (many)
    subpaths which are impossible in the other group. At this stage,
    we need to compare
    \[ \langle z,\hat{z},x^{-1}\hat{x}x\rangle \quad\mbox{ and }\quad
    \langle yz,yx^{-1}\hat{x}xy^{-1}\rangle,\] whose intersection is
    clearly $\<\hat{x}\>$.

  \item $\rho_{y,z}$ has bad subgroup $A_2 = \langle
    z,\hat{z},y^{-1}\hat{y}y, \partial_{\hat{a}_3}\rangle$ and we need
    to intersect with \[
    \rho_{x,z}^{-1}\rho_{x,y}^{-1}\rho_{y,z}\rho_{\hat{y},u}B =
    \langle yz,\hat{y}u,x^{-1}\hat{x}x,
    \rho_{x,z}^{-1}\rho_{x,y}^{-1}\rho_{y,z}\rho_{\hat{y},u}\partial_{\hat{y}}
    \rangle.\] As before, we start by removing $\hat{y}u$ from the
    latter, then the boundary word from both. The remaining
    intersection between
    \[ \langle z,\hat{z},y^{-1}\hat{y}y\rangle \quad\mbox{ and
    }\quad\langle yz,x^{-1}\hat{x}x\rangle. \] is trivial.

  \item Finally, $\rho_{\hat{y},u}$ has bad subgroup $A = \langle
    u,\hat{u}, y, \partial_{y^{-1}}\rangle$, which we intersect with
    \[\rho_{y,z}^{-1}\rho_{x,z}^{-1}\rho_{x,y}^{-1}\rho_{y,z}\rho_{\hat{y},u}B=\langle
    y,\hat{y}u,x^{-1}\hat{x}x ,
    \rho_{y,z}^{-1}\rho_{x,z}^{-1}\rho_{x,y}^{-1}\rho_{y,z}\rho_{\hat{y},u}\partial_{\hat{y}}\rangle\]
    to find in $\langle y\rangle$ (arguing as before).
  \end{enumerate}
  The relation for $\lambda_{x,y}$ is similar, with $\rho$ changed to
  $\lambda$. The case where either $x$, $y$ or both are ``hatted
  letters'' is also analogous.
\end{proof}

\begin{proof}[{Proof of Lemma~\ref{lem:short-relation}}]
  As in the previous lemma, the details vary depending on the nature
  of $\phi$, and the construction is explicit. In contrast to the
  previous lemma, the arguments are straightforward here, and we only
  give the details for the case discussed in the proof of
  Lemma~\ref{lem:preliminary-relation}. The letters below indicate the
  terms in the relation constructed in that proof.
  \begin{enumerate}[a)]
  \item We perform $\lambda_{y,w}$ before this move and
    $\lambda_{y,w}^{-1}$ after.  Note that these moves indeed commute
    with $\rho_{\hat{y},u}^{-1}$.

    The bad subgroup of $\lambda_{y,w}$ is $\langle \hat{y}, w,
    \hat{w}, \partial_y \rangle$. We want to compute the intersection
    with the rank $2$ intersection group from step a) of the previous
    lemma, i.e. with $\langle y, \partial_{y^{-1}}\rangle$. Using
    e.g. Lemma~\ref{lem:drop-turn-hard} we can see that the
    intersection of these two is in fact
    $\langle\partial_{y^{-1}}\rangle$. By Lemma~\ref{lem:whitehead}
    $\langle\partial_{y^{-1}}\rangle$ intersects $E$ trivially, and so
    (1) holds as claimed.

    Finally, as $yw$ is not bad for $\rho_{\hat{y},u}^{-1}$, claim (2)
    holds.
  \item No need
  \item We perform $\rho_{\hat{x},w}$ before this move and
    $\rho_{\hat{x},w}^{-1}$ after.  Note they commute with
    $\rho_{x,y}$, and that $\hat{x}$ is not bad for
    $\rho_{\hat{x},w}$.  As $\hat{x}w$ is not bad for $\rho_{x,y}$,
    the conclusion holds.
  \item No need
  \item No need
  \item This is analogous to a).
  \end{enumerate}
\end{proof}
    
\section{Proof of Theorem \ref{thm:main}} \label{sec:contract} Before
proving the main theorem, we establish the main ingredient, that the
set of arational surface type elements of $\bout$ (even in different
copies of \pml) is path connected. Note that the last sentence of
Theorem \ref{thm:surface case} is used with Proposition
\ref{prop:chain} to prove Theorem \ref{thm:main}.
\begin{thm}\label{thm:surface case}If $x_s$ and $x_e$ are dual to uniquely ergodic (or, in the nonorientable case, elements of $\mathcal{P}$) surface type elements of \bout then there exists $p:[0,1] \to \bout$ continuous so that $p(t)$ is arational for all $t\in[0,1]$. Moreover, for any $\epsilon>0$ and combinatorial chain of $\pml$s from $x_s$ to $x_e$ we may assume that this path is in an $\epsilon$ neighborhood of that chain. 
\end{thm}

\begin{prop}\label{prop:improve BLAS}Let $p:[0,1]\to \bout$ be a \blas
  path, $K_E$ as in Proposition~\ref{prop:the K}, and $\epsilon>0$ be
  given. There exists $p':[0,1] \to \bout$ so that
\begin{enumerate}
\item the distance from $p(x)$ to $p'(x)$ is at most $\epsilon$ for all
  $x \in [0,1]$,
\item $p'([0,1]) \cap K_E=\emptyset$.
\end{enumerate}
\end{prop}

\begin{proof} 
  It suffices to prove the proposition in the case where there is
  exactly one point in $p([0,1])$ which is not arational, call that
  point $\sigma$.  Recall, from the definition of \blas paths, that in
  that case $\sigma$ is the dual tree to a stable lamination of a
  partial pseudo-Anosov $\psi_\sigma$ (for some identification with a
  surface). Also recall that in a neighbourhood of $\sigma$ the path
  $p$ has the form $\cup_k\psi_\sigma^k\xi_i$ for $i=1,2$.  Let $x_1$
  be the starting point of $\xi_1$ (which means $\psi_\sigma x_1$ is
  the ending point) and similarly for $x_2$ and $\xi_2$.  Let $q$ be
  the path as in Theorem \ref{thm:avoid} with $x_1,x_2$, $K_i$ and
  $\psi_\sigma$.  By Theorem \ref{thm:avoid} \eqref{C:contract} there
  exists $k_0$ so that for all $k\geq k_0$, the distance from
  $\psi_\sigma^k (q([0,1])$ to $\sigma$ is at most $\frac \epsilon 2$.
  Let $k_1\geq k_0$ so that $\psi_{\sigma}^{k_1}(q([0,1]))\cap
  K_i=\emptyset$ and the Hausdorff distance from
  $\psi_\sigma^{r}\xi_1$ and $\psi_\sigma^r\xi_2$ to $\sigma$ is at
  most $\frac \epsilon 2$. This exists by Theorem \ref{thm:avoid}
  \eqref{C:avoid}.  Let $p'=p$ outside of $ \cup_{i=k_1}^\infty
  \gamma_i$ and let $p'(t)=\psi_\sigma^{k_1}q$ on $p \setminus
  \cup_{i=k_1}^\infty \gamma_i$. Condition (1) is clear for the $x$ so
  that $p(x)=p'(x)$. All other $x$ have that the distance from both
  $p(x)$ and $p'(x)$ to $\sigma$ is at most
  $\frac{\epsilon}2$. Condition (2) is obvious for the points in $p'$
  that are arational. The other points are contained in
  $\psi_\sigma^{k_1}q$, which was constructed to avoid $K_i$.
\end{proof}

\begin{proof}[Proof of Theorem \ref{thm:surface case}] 
  Enumerate the set of proper free factors in some way as $\mathcal{F}
  = \{ E_i, i \in \mathbb{N} \}$, and denote by $K_i = K_{E_i}$. By

  Proposition \ref{prop:BLAS} there exists a \blas path from $x_s$ to
  $x_e$.%, $p$ and $K_i$ as in Proposition \ref{prop:the K}.
  Let $\epsilon'>0$ be given. By Proposition \ref{prop:improve BLAS},
  with $\epsilon_0:=\epsilon=\frac {\epsilon'} 4$ we may assume
  $p([0,1]) \cap K_1=\emptyset$. Since $K_1$ is closed and $p([0,1])$
  is compact, there exists $\epsilon_1>0$ so that
  $\dist(p([0,1)],K_1)>\epsilon_1$.  Inductively we assume that we are
  given a $\blas$ path $p_i$ and a $\epsilon_1,...\epsilon_i>0$ so
  that
  \begin{equation}\label{eq:distance}
    \dist(p([0,1],K_j)>\left(1-\sum_{\ell=j+1}^i3^{j-\ell} \right)\epsilon_j>\frac 1 2 \epsilon_j
  \end{equation}
  for all $j\leq i$.  By Proposition \ref{prop:improve BLAS} with
  $\epsilon=\epsilon_{i+1}=\frac 1 3^{i+1}\min\{\epsilon_j\}_{j=1}^i$
  and $p=p_i$ and $K=K_{i+1}$ there exists a \blas path, $p_{i+1}$
  from $x_s$ to $x_e$ satisfying equation \eqref{eq:distance} for all
  $j\leq i+1$.  Let $p_{\infty}$ be the limit of the $p_i$. By our
  inductive procedure our sequence of function $p_i$ converges.  By
  \eqref{eq:distance} we have $\dist(p_{\infty}([0,1]),K_i)\geq \frac 1
  2 \epsilon_i>0$ for all $i$. Thus by Proposition \ref{prop:the K} we
  have a path from $x_s$ to $x_e$ so that every $p_\infty(t)$ is
  arational for all $t\in[0,1]$, establishing Theorem \ref{thm:surface
    case}.
\end{proof}

%Let $T\in\partial CV_n$ be an arational tree which is not dual to a
%surface lamination and let $\Delta$ be the simplex of arational trees
%with the same dual lamination as $T$.

To complete the proof of Theorem \ref{thm:main} we need the following result:
\begin{prop} \label{prop:chain} Let $T\in\partial CV_n$ be an
  arational tree, and $\Delta$ be the simplex of arational trees with
  the same dual lamination as $T$.

  For every neighborhood $U$ of $\Delta$ in $\overline {CV_n}$ there
  is a smaller neighborhood $V$ with the following property. Suppose
  $x,y\in V$ are arational and dual to surface laminations (possibly
  on different surfaces). If the dual lamination to $T$ is supported
  on a surface $\Sigma$, we additionally assume that neither $x,y$ are
  dual to laminations on $\Sigma$.  Then $x$ and $y$ can be joined by
  a chain of consecutively adjacent $\pml$'s, each of which is
  contained in $U$.
\end{prop}

The proof of this proposition requires a variant of \cite[Theorem 4.4]{BesRey}. In its
statement we denote by $L(T)$ the dual lamination to a tree $T$. Given a lamination $L$
we denote by $L'$ the sublamination formed by all non-isolated leaves of $L$.
\begin{prop}\label{prop:intersection-0-prop}
  Let $T\in\partial CV_n$ be an arational tree. If $\mu$ is a current so that
  \[ \langle T, \mu \rangle = 0, \]
  and $U\in\partial CV_n$ is another tree with
  \[ \langle U, \mu \rangle = 0, \]
  then
  \begin{enumerate}
   \item either, the dual laminations of $T,U$ agree: $L(U) = L(T)$, or
   \item $T$ is dual to a lamination on a surface $S$, and the support of $\mu$ is
     a multiple of the boundary current
       $\mu_{\partial S}$ of that surface.
  \end{enumerate}
\end{prop}
\begin{proof}
  By the assumption on $T, \mu$, \cite[Theorem~1.1]{Kapovich-Lustig-GAFA} yields
  \[ \mathrm{Supp}(\mu) \subset L(T). \]
  We begin with the case where $T$ is not dual to a surface lamination. In this case,
  \cite[Proposition~4.2~(i)]{BesRey} applies, and shows that $L(T)$ is obtained from
  the minimal lamination $L'(T)$ by adding isolated leaves, each of which is diagonal and
  not periodic. On the other hand, the support of a current cannot contain non-periodic
  isolated leaves. Thus, we then have $\mathrm{Supp}(\mu) \subset L'(T)$, hence
  $\mathrm{Supp}(\mu) = L'(T)$ by minimality.

  Applying \cite[Theorem~1.1]{Kapovich-Lustig-GAFA} to $U, \mu$ yields
  \[ L'''(T) \subset L'(T) = \mathrm{Supp}(\mu) \subset L(U). \]
  In this case, \cite[Corollary~4.3]{BesRey} shows that $L(T) = L(U)$, and we are in case (1).

  \medskip Now suppose that $T$ is dual to a surface lamination. In this case we need
  to describe the dual lamination of $T$ more precisely (see also the proof of
  \cite[Proposition~4.2~(ii)]{BesRey}). Let
  $S$ be a hyperbolic surface with one boundary component which is
  totally geodesic and let $\Lambda$ be a minimal filling measured
  geodesic lamination on $S$, so that $T$ is the $\R$-tree dual to
  $\Lambda$.

  Consider the universal cover $\tilde S$ and the preimage
  $\tilde\Lambda$ of $\Lambda$. The complementary components of
  $\tilde \Lambda$ are ideal polygons and regions containing the lifts
  of the boundary (these are universal covers of hyperbolic crowns and
  are bounded by a lift of $\partial S$ and a chain of leaves with
  consecutive leaves cobounding a cusp) and these, along with
  non-boundary leaves of $\tilde\Lambda$, are in 1-1 correspondence
  with the points of $T$. The lamination $L(T)$ dual to $T$
  consists of pairs of distinct ends of $\tilde S$ that are joined by
  geodesics with $0$ measure. Thus the leaves of $L(T)$ are as
  follows:
  \begin{enumerate}[(i)]
  \item leaves of $\tilde\Lambda$,
    \item diagonal leaves in the complementary components that are
      ideal polygons,
    \item leaves in the crown regions connecting distinct cusps,
      \item leaves in the crown regions connecting a cusp with an end
        corresponding to a lift of $\partial S$,
      \item lifts of $\partial S$.
  \end{enumerate}
  Recall that $\mathrm{supp}(\mu) \subset L(T)$. Since
  the leaves of type (ii) and (iii) are isolated and accumulate on
  leaves of type (i), the measure $\mu$ must assign zero measure to
  them.  Thus the support of $\mu$ is contained in the sublamination
  of $L(T)$ consisting of leaves of type (i), (iv) and (v). In
  this sublamination, the leaves of type (iv) are isolated and
  accumulate on the leaves of both types (i) and (v), so $\mu$ is
  supported on the disjoint union of $\tilde\Lambda$ and the
  lamination $\Delta$ consisting of the lifts of $\partial S$. Thus
  \[ \mu = \nu_1+\nu_2, \] where $\nu_1$ is supported on
  $\tilde\Lambda$ and $\nu_2$ supported on $\Delta$. If $\nu_2$
  assigns $\alpha\geq 0$ to a lift of $\partial S$ then $\nu_2=\alpha
  \mu_{\partial S}$.

  Now, if $\nu_1 \neq 0$, then since 
  \[ \langle U, \nu_1 \rangle = 0, \]
  we can apply \cite[Theorem~1.1]{Kapovich-Lustig-GAFA} to $U, \nu_1$ to obtain
  \[ L'''(T) = \tilde\Lambda = \mathrm{Supp}(\nu_1) \subset L(U), \]
  and \cite[Corollary~4.3]{BesRey} again shows that $L(T) = L(U)$, hence we are in case (1).

  Otherwise, $\mu = \nu_2$ and we are in case (2).
\end{proof}

\begin{lem}\label{lem:complete containment}
  For every neighborhood $U$ of $\Delta$ in $\overline {CV_n}$ there
  is a smaller neighborhood $V$ with the following property. If
  $\pml(\Sigma)$ intersects $V$, and in the case that $T$ is dual to a surface lamination
  on $\Sigma'$, is distinct from $\pml(\Sigma')$, then it is contained in $U$.
\end{lem}

\begin{proof}
  Suppose such $V$ does not exist. Then we have a sequence of
  pairwise
  distinct surfaces
  $\Sigma_i$ and points $x_i,y_i\in \pml(\Sigma_i)$ such that $x_i\to
  x\in\Delta$ and $y_i\to y\notin U$. The boundary curve $\gamma_i$ of
  $\Sigma_i$ is elliptic in both $x_i$ and $y_i$. After a subsequence,
  $\gamma_i$ projectively converges to a current $\mu$, and by the
  continuity of the length pairing we have
  \[ \langle x,\mu\rangle=\langle y,\mu\rangle=0.\]

  Now, apply Proposition~\ref{prop:intersection-0-prop} to
  $x,y,\mu$. If we are in case~(1) of that proposition, then $y$ has
  the same dual lamination as $x$ (equivalently, $T$),
  i.e. $y\in\Delta$. This is a contradiction.

  In case~(2), we instead conclude that the boundary curves $\gamma_i$
  of the $\Sigma_i$ converge (as a current) to the boundary $\gamma$ of the surface $\Sigma$
  supporting the dual lamination of $T$.

  We choose an $i$ large enough (see below), and consider $\delta = \gamma_i$,
  written as a cyclically reduced word 
  \[ \delta = \prod \gamma^{n_j}b_j. \]
  in a basis where $\gamma$ is written as a shortest possible word.
  Since the $\gamma_i$ converge to $\gamma$ as currents, for any given $\epsilon$ we
  may ensure (by choosing $i$ large enough) that
  \[ \sum l(b_j) \leq \epsilon \sum n_jl(\gamma). \]
  Further, since the $\Sigma_i$ are all distinct, the length of $\gamma_i$ diverges,
  and we may thus assume that $l(\delta)$ is much larger than $l(\gamma)$.

  On the other hand, since the $\gamma_i$ are boundary curves of
  surfaces, and therefore have uniformly small length $L_0=l(\gamma)$
  in a suitable basis, by Whitehead's theorem, there is a sequence of
  Whitehead moves $\phi_k$ so that
  \[ l(\delta) > l(\phi_1\delta) > l(\phi_2\phi_1\delta) > \ldots \]
  Let $B$ be a bounded-cancellation constant that works for all Whitehead moves.

  Let $N$ be the largest number so that $l(\phi_i\cdots\phi_1\gamma) = l(\gamma)$ for
  all $i \leq N$.

  We first claim that $N\geq 1$. Namely, suppose for contradiction
  that $l(\phi_1\gamma) \geq l(\gamma)+1$. Intuitively, the increase
  in length in $\gamma^{n_j}$ is much larger than the decrease in
  length in $b_j$. More formally,
  we then have
  \begin{eqnarray*}
  l(\phi_1\delta) &\geq& \sum (n_j(L_0+1) - B) \\
    &\geq& \frac{L_0+1}{L_0} (\sum n_jl(\gamma)) - B \sum l(b_j) \\
    &\geq& \left(\frac{L_0+1}{L_0} - B\epsilon\right) \sum n_jl(\gamma) \\
    &\geq& \left(\frac{L_0+1}{L_0} - B\epsilon\right) \sum ( n_jl(\gamma) + l(b_j) ) - \sum l(b_j) \\
    &\geq& \left(\frac{L_0+1}{L_0} - (B+1)\epsilon\right) l(\delta),  
  \end{eqnarray*}
  which, if $\epsilon$ is chosen small enough, would imply $l(\phi_1\delta) > l(\delta)$ which is
  impossible by the above.

    \smallskip There are a finite
  number of identifications of $F_n$ with a surface with one boundary
  component so that the length of its boundary word is
  $l(\gamma)$. Denote those surfaces by
  $\Sigma=\Sigma^1, \ldots, \Sigma^K$.  By the above, for all
  $i \leq N$, the maps $\phi_i$ can be represented by homeomorphisms
  between suitable $\Sigma^r, \Sigma^s$, and thus the same is true for
  the map $\Psi= \phi_N\cdots \phi_1$. Now, consider (as unreduced words)
  \[ \Psi\delta = \prod \Psi(\gamma^{n_j}) \Psi(b_j). \]
  We can write $b_j$ as a product of at most $l(b_j)$ simple loops on $\Sigma_i=\Sigma^{r_0}$.
  Since $\Psi$ is a surface map, $\Psi(b_j)$ is also a product of at most $l(b_j)$ simple nonseparting loops
  on $\Sigma^s$. Now, a reduced word representing a simple nonseparating loop cannot contain the square of
  the boundary as a subword (e.g. by considering Whitehead graphs, and observing that any nonseparating simple loop is primitive). Thus, reducing the description
  of $\Psi\delta$ above removes at most one copy of the boundary word $\Psi(\gamma)$ for each
  simple component. Hence, there is a \emph{reduced} description
  \[ \Psi\delta = \prod \gamma_0^{m_j} c_j \]
  which has
  \[ \sum m_j l(\gamma) \geq \sum n_j l(\gamma) - \sum l(b_j)
    \geq (1-\epsilon) \sum n_j l(\gamma), \]
  and 
  \[
     (1+\epsilon) \sum n_jl(\gamma) \geq l(\gamma) > l(\Psi\delta) = \sum n_j l(\gamma) + \sum l(c_j), 
   \]
   hence
   \[ \sum l(c_j) \leq \epsilon \sum n_jl(\gamma) \leq \frac{\epsilon}{1-\epsilon}\sum m_j l(\gamma). \]
   But now, if $\epsilon$ was chosen small enough, so that the argument showing
   $l(\phi_1\gamma) = L_0$ also applies to $\epsilon' = \frac{\epsilon}{1-\epsilon}$, then
   that same argument shows $l(\phi_{N+1}\Psi(\gamma)) = L_0$, contradicting maximality of $N$.
\end{proof}

\begin{lem}\label{lem:folding close} Let $U$ be a neighborhood of
  $\Delta$ in $\overline{CV_n}$. There exists a neighborhood $V$ of
  $\Delta$ in $\overline{CV_n}$ so that if $x',y' \in CV_n \cap V$
  then any folding path from $x'$ to $y'$ is contained in $U$.
\end{lem}

\begin{proof}
  Recall from Section \ref{3.1} that there is a coarsely continuous
  function $\Phi:\overline{CV_n}\to \overline{FF_n}$ that restricted
  to arational trees gives a quotient map to $\partial CV_n$. This map
  takes folding paths in $CV_n$ to reparametrized quasigeodesics with
  uniform constants in $FF_n$ \cite{BesFei}
  and it takes $\Delta$ to a point $[\Delta]\in\partial FF_n$. By the
  coarse continuity, there
  is a neighborhood $U'$ of $[\Delta]\in \overline{FF_n}$ such that
  $\Phi^{-1}(U')\subset U$. By hyperbolic geometry there is a
  neighborhood $V'\subset U'$ of $[\Delta]$ such that any
  quasigeodesic with above constants with endpoints in $V'$ is
  contained in $U'$. Finally, let $V$ be a neighborhood of $\Delta$
  such that $\Phi(V)\subset V'$ ($V$ exists by the coarse continuity).
\end{proof}

Before we can prove Proposition~\ref{prop:chain}, we need one more
definition. Namely, given an identification $\sigma$ of the free group
with $\pi_1(\Sigma)$, we define the {\it extended projective measured
  lamination sphere} $\widetilde{\pml}_\sigma$ to be the union of
$\pml_\sigma$ and the subset of $CV_n$ consisting of graphs where the
boundary curve of $\Sigma$ crosses every edge exactly twice
(alternatively, the graph can be embedded in the surface with the
correct marking).

\begin{proof}[Proof of Proposition \ref{prop:chain}]
  Let $U=U_0$ be a given neighborhood of $\Delta$. For a large (for
  now unspecified) integer $N$ find neighborhoods $$U_0\supset
  U_1\supset U_2\supset\cdots\supset U_N$$ of $\Delta$ so that each
  pair $(U_i,U_{i+1})$ satisfies Lemmas \ref{lem:complete containment}
  and \ref{lem:folding close}. We then set $V=U_N$. To see that this
  works, let $x,y\in V$ be arational and dual to surface
  laminations. Let $P_x$ and $P_y$ be the extended PML's containing
  $x,y$ respectively. Thus $P_x,P_y\subset U_{N-1}$. Choose roses
  $x'\in P_x\cap CV_n$ and $y'\in P_y\cap CV_n$. After adjusting the
  lengths of edges of $x'$ there will be a folding path from $x'$ to
  $y'$ which is then contained in $U_{N-2}$. We can assume that the
  folding process folds one edge at a time. We can choose a finite
  sequence of graphs along the path, starting with $x'$ and ending
  with $y'$, so that the change in topology in consecutive graphs is a
  simple fold. It follows that the extended PML's can be chosen so
  that the surfaces share a subsurface of small cogenus. Further, in
  each graph we can collapse a maximal tree so we get a
  rose. Consecutive roses will differ by the composition of boundedly
  many Whitehead automorphsims and each Whitehead automorphism is a
  composition of boundedly many basic moves. We can then insert a
  bounded chain of extended PML's between any two in our sequence so
  that in this expanded chain any two consecutive PML's differ by a
  basic move. If $N$ is sufficiently large this new chain will be
  contained in $U=U_0$.
\end{proof}

\begin{proof}[Proof of Theorem \ref{thm:main}]
  To prove that $\partial FF_n$ is path-connected, it suffices to join
  by a path points $\Phi(T),\Phi(S)\in \partial FF_n$ where both
  $T,S\in\partial CV_n$ are arational trees and $S$ is 
    dual to a surface lamination $\lambda$ which is
    uniquely ergodic (or, in the nonorientable case, an element of
    $\mathcal{P}$). Additionally, if $T$ is itself dual
  to a lamination on a surface $\Sigma$, we assume that $\lambda$ is not
  a lamination of that same surface. For brevity, we call such trees \emph{good surface trees} in this proof.
    
Let $\Delta\subset \partial CV_n$ be the simplex of arational trees
equivalent to $T$, and let $U_1\supset U_2\supset\cdots$ be a nested
sequence of smaller and smaller neighborhoods of $\Delta$ so that each
pair $(U_i,U_{i+1})$ satisfies Proposition \ref{prop:chain}. Choose
$S_i\in U_i$ to be a good surface tree, for $i\geq 1$ (see Lemma
\ref{surfaces dense}).  By Theorem \ref{thm:surface case} there is a
path $p$ from $S$ to $S_1$ in $\partial CV_n$ consisting of arational
trees, and likewise there is such a path $p_i$ from $S_i$ to
$S_{i+1}$. By our choice of the $U_i$ and the last sentence in Theorem
\ref{thm:surface case} we can arrange that each $p_i$ is contained in
$U_{i-1}$ for $i\geq 2$. The concatenation $q=p*p_1*p_2*\cdots$ is a
path parametrized by a half-open interval that accumulates on $\Delta$
since it is eventually contained in $U_i$ for every $i$. It may not
converge in $\partial CV_n$ unless $\Delta$ is a point (that is, $T$
is uniquely ergometric, see \cite{CHL0}) but $\Phi(q)$ converges to
$\Phi(T)$, proving path connectivity.

Local path connectivity is similar.  The key observation
  is that, in the construction of paths above, if we choose
  $S\in U_{i+1}$ then the path joining $S$ to the simplex of $T$ can
  be chosen to lie in $U_i$.
  
  Now, recall that a space $X$
  is locally path connected at $x\in X$ if for every neighborhood $U$
  of $x$ there is a smaller neighborhood $V$ of $x$ so that any two
  points in $V$ are connected by a path in $U$; this implies the
  ostensibly stronger property that $x$ has a path connected
  neighborhood contained in $U$ (namely, by taking the path component of $x$
  in $U$).

    Hence, we want to show that for every $i$ there is
    $j>i$ so that if $T'\in U_j$ is arational, there is a path of
    arationals defined on an open interval accumulating to the
    associated simplices $\Delta,\Delta'$ on the two ends which is contained
    in $U_i$.

    We begin by choosing $j>i$ so that when $T'\in U_j$ is
    arational then its simplex $\Delta'$ is contained in $U_{i+1}$ (this is
    possible because $\Phi(U_{i+1})$ contains neighbourhoods of $\Phi(T)$).

    Now pick a 
    $S\in U_{i+1}$ which is a good surface tree (for both $T, T'$) close to $\Delta'$. The key observation above implies that
    for any such $S$ we can find a path
    from $S$ accumulating on $\Delta$ which lies in $U_i$. 
    If we choose $S$ close enough to
    $\Delta'$, then again by the key observation, we can also find a path from
    $S$ accumulating on the simplex of $\Delta'$, which is contained
    in $U_i$ (by choosing it to lie in the corresponding sequence
    $U'_k$ for $\Delta'$ for a large enough $k$). 
  Putting them together gives the desired path.
\end{proof}

\section{One-endedness of other combinatorial complexes}

In this section, we discuss one-endedness of various combinatorial
complexes. To this end, we use the following criterion.
\begin{prop}\label{1-ended}
  Let $X,Y$ be $\delta$-hyperbolic spaces, $G$ a group acting
  coboundedly by isometries on $X$ and $Y$, and let $\pi:X\to Y$ be an
  equivariant Lipschitz map which is alignment preserving. Suppose
  there is some $g\in G$ which is loxodromic in $Y$ (and therefore
  also in $X$).

  If $Y$ is 1-ended, so is $X$.
\end{prop}

Recall \cite{Guirardel, KapovichRafi} that a map $\pi$ is {\it
  alignment preserving} if there is a constant $C\geq 0$ such that the
image of any geodesic segment is contained in the $C$-neighborhood of
any geodesic joining the images of the endpoints.
\begin{rem}
  We want to remark that \cite{KapovichRafi} use only the apparently
  weaker property that $\pi([x,y])$ is bounded whenever
  $\pi(x),\pi(y)$ are close, rather than alignment preserving.
  However, they also show that a map between hyperbolic metric spaces
  with this weaker property is alignment preserving in the stronger
  sense.
\end{rem}

\begin{proof}[Proof of Proposition~\ref{1-ended}]
  Let $K_X$ be a metric ball in $X$. We define $L_X$ to be the
  Hausdorff $N$-neighborhood of $K_X$, with $N$ sufficiently
  large. Let $x_1,x_2\in X\smallsetminus L_X$. We will connect
  $x_1,x_2$ by a path in the complement of $K_X$.

  Fix an axis $\ell$ in $X$ of $g$ (i.e. a quasi-geodesic line where
  $g$ acts by translation). Since the action of $G$ on $X$ is
  cobounded, there is a translate $\ell_1$ of $\ell$ that passes
  within a bounded distance from $x_1$. Let $r$ be the ray starting at
  $x_1$, having a bounded initial segment joining $x_1$ with $\ell_1$,
  and the rest is one of the two half-lines in $\ell_1$. Since $K_X$
  is quasi-convex, there is a choice of a half-line so that if $N$ is
  sufficiently large, $r$ is disjoint from $K_X$. The image of $r$ in
  $Y$ follows an axis of a conjugate of $g$, so it goes to infinity in
  $Y$. We can thus join $x_1$ by a path missing $K_X$ to a point
  $x_1'$ whose image in $Y$ misses a bounded set $P$ such that points
  in the complement of $P$ can be joined by paths missing the
  $N$-neighborhood of $\pi(K_X)$. In the same way we can join $x_2$ to
  a point $x_2'$. It now remains to join $x_1'$ to $x_2'$.

  Join $\pi(x_1'),\pi(x_2')$ by a path missing the $N$-neighborhood of
  $\pi(K_X)$. We will now coarsely lift this path to the desired
  path. Let $\pi(x_1)=y_1,y_2,\cdots,y_s=\pi(x_2)$ be points along the
  path at distance $\leq 1$. For each $y_i$ choose a point $\tilde
  y_i\in X$ whose image in $Y$ is at a bounded distance from $y_i$
  (this is possible since $\pi$ is coarsely onto), and so that $\tilde
  y_1=x_1$ and $\tilde y_s=x_2$. The desired path is the concatenation
  of geodesic segments joining the consecutive $\tilde y_i$. Since
  $\pi$ is alignment preserving, the images in $Y$ of these geodesic
  segments are uniformly bounded, so when $N$ is large they will miss
  $\pi(K_X)$, and the path between $x_1$ and $x_2$ will miss $K_X$.
\end{proof}

\begin{cor}\label{cor:other-complexes}
  For $n\geq \rank$ the free splitting complex $FS_n$, the cyclic
  splitting complex $FZ_n$, and the maximal cyclic splitting complex
  $FZ^{max}_n$ are all 1-ended.
\end{cor}

\begin{proof}
  There are natural coarse maps $$CV_n\to FS_n\to FZ^{max}_n\to
  FZ_n\to FF_n$$ and they are equivariant with respect to the action
  of $Out(F_n)$. Except on $CV_n$, all these spaces are hyperbolic and
  the $Out(F_n)$ action is cobounded. Proofs of hyperbolicity show
  that images of folding paths in $CV_n$ are reparametrized
  quasi-geodesics with uniform constants. This implies that all the
  maps starting from $FS_n$ are alignment preserving. Fully
  irreducible automorphisms are loxodromic in all four complexes.
\end{proof}

\appendix

\section{Explicit Constructions of Curves and Subsurfaces}
\label{sec:pictures-cases}

In this section we collect the constructions of curves and
subsurfaces claimed in Lemma~\ref{lem:additional-twist}.

We begin with the construction of curves in
Lemma~\ref{lem:additional-twist}.  In
Figure~\ref{fig:finding-extra-twists}, the curves for the right
multipliation moves are shown in green; the curves for the left
multiplication moves are shown in purple. Dashed lines entering a
group of bands are understood to follow around the boundary of the
surface, not intersecting the basis loops corresponding to the loops.

\begin{sidewaysfigure}
  \vspace{13 cm}
  \includegraphics[width=0.95\textwidth]{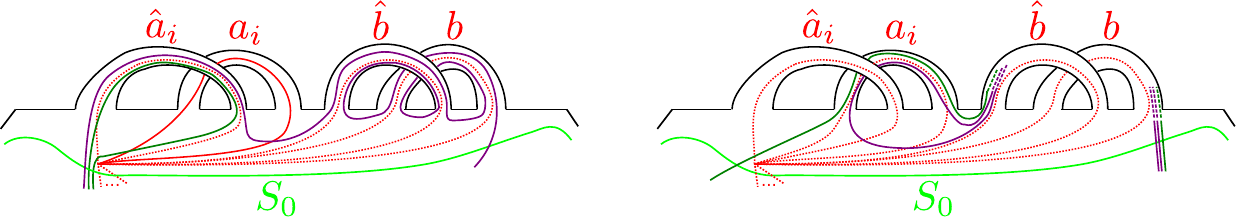}
  \includegraphics[width=0.95\textwidth]{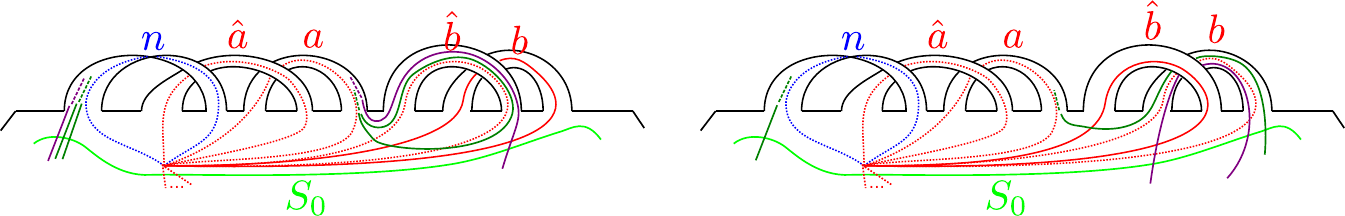}
  \includegraphics[width=0.95\textwidth]{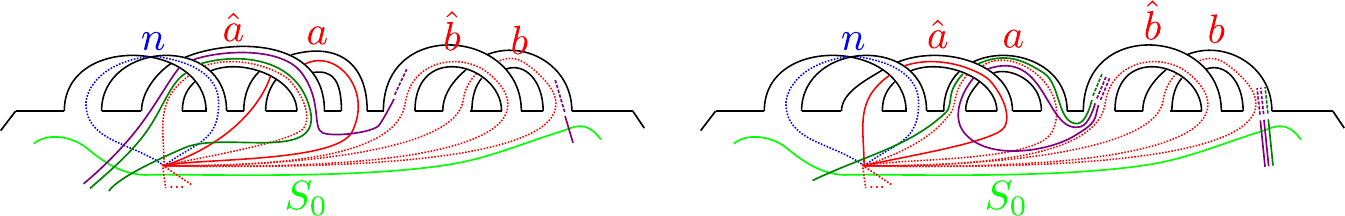}
  \includegraphics[width=0.4\textwidth]{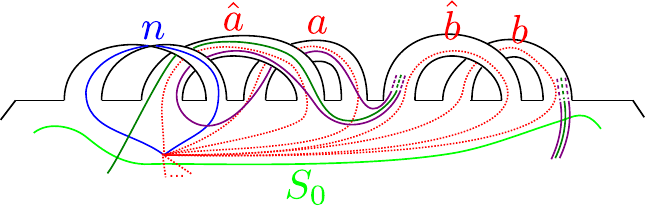}
  \caption{Constructing ``extra twists''}
  \label{fig:finding-extra-twists}
\end{sidewaysfigure}

Finally, in Figure~\ref{fig:finding-extra-twists2}, the additional 
curves for the last claim of the lemma are shown.
\begin{figure}
  \begin{center}
    \includegraphics[width=0.95\textwidth]{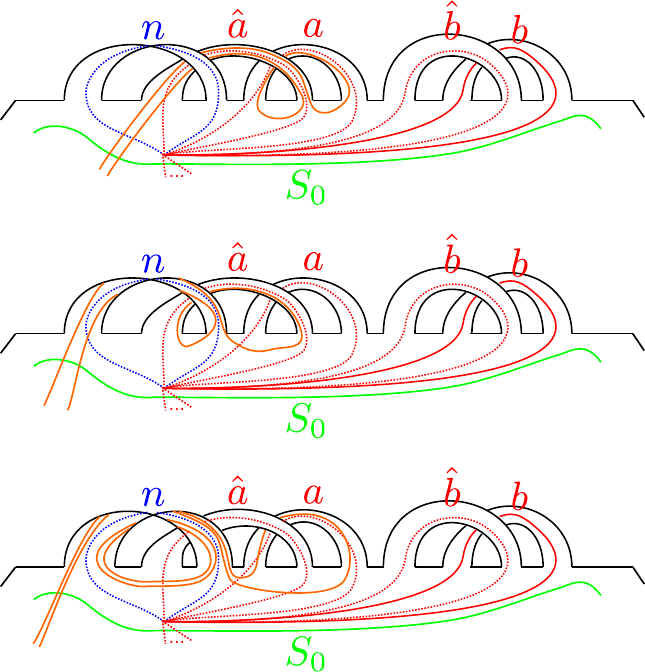}
  \end{center}

  \caption{Constructing even more ``extra twists''}
  \label{fig:finding-extra-twists2}
\end{figure}

From these explicit descriptions, fundamental groups of the bad
subsurfaces can be read off. We collect the results in the following
lemma.
\begin{lem}\label{lem:explicit-good-subsurfaces}
  For a move $\phi=\lambda_{x,y}$ or $\rho_{x,y}$, the bad subsurfaces
  have the following fundamental groups.  We denote by $\partial$ the
  word representing the boundary of the surface, i.e.
  \[ \partial = \prod_{i=1}^g [\hat{a}_i, a^{-1}_i], \] if the surface
  is orientable, and
  \[ \partial = (n \hat{a} a \hat{a}^{-1} n a) \prod_{i=2}^g
  [\hat{a}_i, a^{-1}_i], \] otherwise (here, $\hat{a},a$ are linked
  with the nonorientable letter $n$, and $\hat{a}_2,\ldots$ are the
  following letters). We denote by $\partial_w$ the cyclic permutation
  of $\partial$ starting with the letter $w$.
  
\begin{description}
\item[$x$ two-sided, not linked with one-sided] Here, three
  possibilities for $y$ exist.
  \begin{description}
  \item[$y$ two-sided, not linked with one-sided] For $x=a_i$ (not a
    hatted letter), and the right multiplication move $\rho_{x,y}$ we
    have
    \[ \pi_1(\Sigma-S^g) = \langle y, \hat{y},
    x^{-1}\hat{x}x,\partial_{\hat{a}_{i+1}}\rangle \] For $x=a_i$ (not
    a hatted letter), and the left multiplication move $\lambda_{x,y}$
    we have
    \[ \pi_1(\Sigma-S^g) = \langle y, \hat{y}, \hat{x},\partial_{a_i}\rangle \]

    For $x=\hat{a}_i$ (a hatted letter), and the right multiplication
    move $\rho_{x,y}$ we have
    \[ \pi_1(\Sigma-S^g) = \langle y, \hat{y},
    \hat{x},\partial_{a^{-1}_i}\rangle \] For $x=\hat{a}_i$ (a hatted
    letter), and the left multiplication move $\lambda_{x,y}$ we have
    \[ \pi_1(\Sigma-S^g) = \langle y, \hat{y},
    x\hat{x}x^{-1},\partial_{\hat{a}_i}\rangle \]
				
  \item[$y$ one-sided or linked with one-sided] Here, we call the
    one-sided letter $n$, the linked two-sided letters $a, \hat{a}$
    (i.e. $y$ is one of these three), and we assume that $x$ is one of
    the adjacent $b, \hat{b}$ (this is enough due to the previous
    normalisation). In this case, the fundamental group of the bad
    subsurface $\Sigma-S^g$ has rank three, with two generators $g_1,
    g_2$ depending solely on $y$, and the final one on $x$ and the
    type of move.  Namely, put
    \begin{description}
    \item[$y=a$] 
      \[ g_1 = a, \quad g_2 = n\hat{a}a\hat{a}^{-1} \]
    \item[$y=\hat{a}$] 
      \[ g_1 = \hat{a}, \quad g_2 = na \]
    \item[$y=n$] 
      \[ g_1 = n, \quad g_2 = (a^{-1}n^{-1}\hat{a})a(a^{-1}n^{-1}\hat{a})^{-1} \]
    \end{description}
			
    For $x=b$, and the right multiplication move $\rho_{x,y}$ we have
    \[ \pi_1(\Sigma-S^g) = \langle g_1, g_2,
    b^{-1}\hat{b}b,\partial_{b^{-1}}\rangle \] For $x=b$, and the left
    multiplication move $\lambda_{x,y}$ we have
    \[ \pi_1(\Sigma-S^g) = \langle g_1, g_2,
    \hat{b},\partial_{\hat{b}^{-1}} \rangle \]

    For $x=\hat{b}$, and the right multiplication move $\rho_{x,y}$ we
    have
    \[ \pi_1(\Sigma-S^g) = \langle g_1, g_2,
    b,\partial_{b^{-1}}\rangle \] For $x=\hat{b}$, and the left
    multiplication move $\lambda_{x,y}$ we have
    \[ \pi_1(\Sigma-S^g) = \langle g_1, g_2,
    \hat{b}b\hat{b}^{-1},\partial_{\hat{b}}\rangle \]
  \end{description}
	
\item[$x$ two-sided, linked with one-sided] Again, we call the
  one-sided letter $n$ and the linked two-sided letters $a, \hat{a}$
  (of which $x$ is one), and we assume that $y$ is one of the adjacent
  $b, \hat{b}$ (this is enough due to the previous normalisation).

  For $x=a$, and the right multiplication move $\rho_{x,y}$ we have
  \[ \pi_1(\Sigma-S^g) = \langle b, \hat{b}, a^{-1}n^{-1}\hat{a},
  a^{-1}\hat{a}a ,\partial_{\hat{b}}\rangle \] For $x=a$, and the left
  multiplication move $\lambda_{x,y}$ we have
  \[ \pi_1(\Sigma-S^g) = \langle \hat{a}, b, \hat{b}, a\hat{a}^{-1}n,\partial_{a}
  \rangle \]

  For $x=\hat{a}$, and the right multiplication move $\rho_{x,y}$ we
  have
  \[ \pi_1(\Sigma-S^g) = \langle b, \hat{b}, a,
  \hat{a}^{-1}n\hat{a},\partial_{a} \rangle \] For $x=\hat{a}$, and
  the left multiplication move $\lambda_{x,y}$ we have
  \[ \pi_1(\Sigma-S^g) = \langle b, \hat{b}, n,
  \hat{a}a\hat{a}^{-1},\partial_n \rangle \]

\item[$x$ one-sided] Again, we call the one-sided letter $n=x$, the
  linked two-sided letters $a, \hat{a}$, and we assume that $y$ is one
  of the adjacent $b, \hat{b}$ (this is enough due to the previous
  normalisation).

  For $x=n$, and the right multiplication move $\rho_{x,y}$ we have
  \[ \pi_1(\Sigma-S^g) = \langle b, \hat{b}, \hat{a},
  a\hat{a}^{-1}n,\partial_{\hat{a}} \rangle \] For $x=n$, and the left
  multiplication move $\lambda_{x,y}$ we have
  \[ \pi_1(\Sigma-S^g) = \langle b, \hat{b}, a\hat{a}^{-1},
  n\hat{a},\partial_{n} \rangle \]

\end{description}
In all cases, any loop corresponding to a basis element except for $x,
\hat{x}, y, \hat{y}$ (and possibly $n$, if one of $x,y$ is linked) can
be homotoped into the good subsurface.
\end{lem}

\section{Proof of Lemma~\ref{lem:preliminary-relation}}
\label{sec:check}
Throughout, we call the argument in the proof of
Proposition~\ref{prop:solving-overlap} \emph{Case 1}.

First, we observe that the argument of Case~1 extends to the case
nonorientable surface. The only difference in this case is that the
boundary word $\partial$ has a slightly different form (see
Section~\ref{sec:moves}).  However, as we may assume that all the
auxiliary letters used above are two-sided, and the boundary word of
the nonorientable surface also contains commutators of all the
two-sided letters which are not linked with the one-sided letter, the
argument works completely analogously.
    
\bigskip It remains to discuss the remaining cases in the case of a
non-orientable surface, where either $x$ or $y$ is one-sided or linked
with the one-sided.  Unjustified claims about intersections between
subgroups are proved using the arguments in the proof of
Proposition~\ref{prop:solving-overlap}.  We make use of the following
notation and assumptions throughout:
\begin{enumerate}
\item We denote by $\partial_x$ the cyclic permutation of the boundary
  word $\partial$ starting at (the first occurrence of) the (signed)
  letter $x$.
\item All ``auxiliary letters'' are chosen to be seperated by at least
  one index from all active letters and from each other (so that the
  subword detection arguments from Case 1 apply).
\item If $x$ is a chosen, two-sided letter (i.e. $x = a_i$ or
  $\hat{a}_i$), then we denote by $x_+$ the \emph{next} letter of the
  same type (i.e. $x_+ = a_{i+1}$ or $x_+ = \hat{a}_{i+1}$) and
  $\hat{x}_+$ the next letter of opposite type (i.e. $\hat{x}_+ =
  \hat{a}_{i+1}$ or $\hat{x}_+ = a_{i+1}$).
\end{enumerate}

\subsection{Case 2}
This case case concerns $y$ general $x$ two-sided, linked to
one-sided. Let $n$ be the one-sided letter. The fundamental group of
the bad surface is \[ B = \langle
y,\hat{y},x^{-1}n^{-1}\hat{x},x^{-1}\hat{x}x, \partial_{\hat{x}_+}
\rangle.\]

The relation we will use is:
\[\rho_{x,y}=\rho_{\hat{y},u}^{-1}\rho_{y,z}^{-1}\rho_{n,w}^{-1}\rho_{x,y}\rho_{x,z}\rho_{n,w}\rho_{y,z}\rho_{\hat{y},u}\].

\vspace{2mm}
\noindent
\textit{Check:}
\begin{enumerate}[a)]
\item $\rho_{\hat{y},u}^{-1}$ has bad subgroup $\langle u,\hat{u},y
  ,\partial_{y^{-1}}\rangle$ which intersects $B$ in $\langle
  y, \partial_{y^{-1}} \rangle$.
	
\item $\rho_{y,z}^{-1}$ has bad subgroup $\langle z,\hat{z},
  y^{-1}\hat{y}y, \partial_{\hat{y}_+} \rangle$ which intersects
  $\rho_{\hat{y},u}B=\langle
  y,\hat{y}u,x^{-1}n^{-1}\hat{x},x^{-1}\hat{x}x,
  \rho_{\hat{y},u}\partial_{\hat{x}_+} \rangle$ trivially. Indeed, by
  Lemma~\ref{lem:drop-turn-easy} applied to the path
  $\hat{y}y\hat{y}^{-1}y^{-1}$ we may drop $\partial_{\hat{y}_+}$ from
  $\langle z,\hat{z}, y^{-1}\hat{y}y, \partial_{\hat{y}_+}
  \rangle$. Having done this, applying Lemma~\ref{lem:drop-turn-easy}
  (to a number of paths) we may drop
  $\rho_{\hat{y},u}\partial_{\hat{x}_+}$ from
  $\rho_{\hat{y},u}B$. Having done this we may apply
  Lemma~\ref{lem:drop-turn-easy} to $\hat{y}u$ and $y^{-1}\hat{y}y$ we
  may drop $\hat{y}u$ from $\rho_{\hat{y},u}\langle y, \hat{y},
  x^{-1}n^{-1}\hat{x},x^{-1}\hat{x}x\rangle$ and $y^{-1}\hat{y}y$ from
  $\langle u,\hat{u},y^{-1}\hat{y}y\rangle$. The rest of this case is
  straightforward.
	
\item $\lambda_{n,w}^{-1}$ has bad subgroup $\langle
  n\hat{x},x^{-1}\hat{x}^{-1},w,\hat{w}^{-1}, \partial_n\rangle$ which
  intersects $\rho_{y,z} \rho_{\hat{y},u}B=\langle yz,\hat{y}u,
  x^{-1}n^{-1}\hat{x},x^{-1}\hat{x}x ,\rho_{y,z}
  \rho_{\hat{y},u} \partial_{\hat{x}_+}\rangle$ in $\langle \hat{x}
  \rangle$ up to conjugation.

  Indeed, as in the previous step, by Lemma~\ref{lem:drop-turn-easy}
  applied to the path $\hat{y}y\hat{y}^{-1}y^{-1}$ we may drop
  $\partial_n$ from $\langle
  n\hat{x},x^{-1}\hat{x}^{-1},w,\hat{w}^{-1}, \partial_n\rangle$. Having
  done this we may drop $\rho_{y,z}
  \rho_{\hat{y},u}\partial_{\hat{x}_+}$ and then $yz$ and $\hat{y}u$
  from $\rho_{y,z} \rho_{\hat{y},u}B$.  So it suffices to consider the
  intersection of $\langle n \hat{x},x^{-1}\hat{x}^{-1}\rangle $ and
  $\langle x^{-1}n^{-1}\hat{x},x^{-1}\hat{x}x\rangle$. Since these are
  free factorse we consider the abelinization of these which are
  isomorphic to $\mathbb{Z}^3$ where the vector $(a,b,c)$ represents
  $n^a\hat{x}^bx^c$. The claim follows from the fact that the subspace
  spanned by $\{(1,1,0),(0,-1,-1)\}$ intersects the subspace spanned
  by $\{(-1,1,-1),(0,1,0)\}$ trivially.
	
\item $\rho_{x,y}$ has bad subgroup $B$ which intersects
  $\lambda_{n,w}\rho_{y,z} \rho_{\hat{y},u}B=\langle yz,\hat{y}u,
  n^{-1}w^{-1}\hat{x},x^{-1}\hat{x}x, \lambda_{n,w}\rho_{y,z}
  \rho_{\hat{y},u}\partial_{\hat{x}_+}\rangle$ in $\langle
  x^{-1}\hat{x}x\rangle$.
	
\item $\rho_{x,z}$ has bad subgroup $\langle
  z,\hat{z},xn^{-1}\hat{x},x^{-1}\hat{x}x, \partial_{\hat{x}_+}
  \rangle$ which intersects $\rho_{x,y}^{-1}\lambda_{n,w}\rho_{y,z}
  \rho_{\hat{y},u}B=\langle yz,\hat{y}u,
  yx^{-1}n^{-1}w^{-1}\hat{x},yx^{-1}\hat{x}xy^{-1},
  \rho_{x,y}^{-1}\lambda_{n,w}\rho_{y,z}
  \rho_{\hat{y},u}\partial_{\hat{x}_+}\rangle$ in
  $\langle\hat{x}\rangle$ up to conjugation.
	
\item $\lambda_{n,w}$ has bad subgroup $\langle
  n\hat{x},x\hat{x}^{-1},w,\hat{w}^{-1}, \partial_x\rangle$ which
  intersects $\rho_{x,z}^{-1} \rho_{x,y}^{-1}\lambda_{n,w}\rho_{y,z}
  \rho_{\hat{y},u}B =\langle yz,\hat{y}u,
  yzx^{-1}n^{-1}w^{-1}\hat{x},yzx^{-1}\hat{x}xz^{-1}y^{-1},
  \rho_{x,z}^{-1} \rho_{x,y}^{-1}\lambda_{n,w}\rho_{y,z}
  \rho_{\hat{y},u}\partial_{\hat{x}_+}\rangle=\langle yz,\hat{y}u,
  x^{-1}w^{-1}n^{-1}\hat{x},x^{-1}\hat{x}x, \rho_{x,z}^{-1}
  \rho_{x,y}^{-1}\rho_{n,w}\rho_{y,z}
  \rho_{\hat{y},u}\partial_{\hat{x}_+}\rangle$ trivially.  Similarly
  to in previous cases we apply Lemma~\ref{lem:drop-turn-easy} to
  first drop $\partial_x$ and then
  $\rho_{x,y}^{-1}\rho_{n,w}\rho_{y,z}
  \rho_{\hat{y},u}\partial_{\hat{x}_+}$ from their respective
  subrgoups.  It is now clear that we can restrict our consideration
  to possible intersections of $\langle
  n\hat{x},x\hat{x}^{-1},w\rangle$ and $\langle x
  \hat{x}^{-1}nw,\hat{x}\rangle$. Any nontrivial reduced word in the
  latter (except $\hat{x}$) contains the subword $nw$ (or its inverse)
  without cancellation. As neither $\hat{x}$ nor any word containing
  the subword $nw$ is contained in the former, the claim follows.

\item $\rho_{y,z}$ has bad subgroup $\langle z,\hat{z}, y^{-1}\hat{y}y
  , \partial_{\hat{y}_+}\rangle$ which intersects
  $\lambda_{n,w}^{-1}\rho_{x,z}^{-1}\rho_{x,y}^{-1}\lambda_{n,w}\rho_{y,z}
  \rho_{\hat{y},u}B =\langle yz,\hat{y}u,
  yzx^{-1}w^{-1}w^{-1}n^{-1}\hat{x},yzx^{-1}\hat{x}xz^{-1}y^{-1},
  \rho_{x,z}^{-1}\rho_{x,y}^{-1}\rho_{y,z}
  \rho_{\hat{y},u}\partial_{\hat{x}_+}\rangle$ trivially.
	
\item $\rho_{\hat{y},u}$ has bad subgroup $\langle
  u,\hat{u},y, \partial_{y^{-1}}\rangle$ which intersects
  $\rho_{y,z}^{-1}
  \lambda_{n,w}^{-1}\rho_{x,z}^{-1}\rho_{x,y}^{-1}\lambda_{n,w}\rho_{y,z}
  \rho_{\hat{y},u}B =\langle y,\hat{y}u,
  yx^{-1}w^{-1}w^{-1}n^{-1}\hat{x},yx^{-1}\hat{x}xy^{-1},\rho_{x,y}^{-1}
  \rho_{\hat{y},u}\partial_{\hat{x}_+} \rangle$ in $\langle y\rangle$.
	
\end{enumerate}
This completes the checks for the preliminary relation.

We now collect some variants on this case. First is \textbf{Case 2'}
of the left multiplication move $\lambda_{x,y}$. Here, the bad
subgroup is $B=\langle \hat{x},y,\hat{y}, x\hat{x}^{-1}n, \partial_x
\rangle$. We use the relation
\[ \lambda_{x,y}=
\rho_{\hat{y},u}^{-1}\rho_{y,z}^{-1}\lambda_{n,w}^{-1}\lambda_{x,y}\lambda_{x,z}\lambda_{n,w}\rho_{y,z}\rho_{\hat{y},u}. \]
We only indicate how the checks above need to be amended in this case.

\noindent
a), b), c) are similar to case 2.

\noindent
d) $\lambda_{x,z}$ which has bad subgroup $\langle \hat{x},z,\hat{z},
x\hat{x}^{-1}n, \partial_x \rangle$ which intersects
$\rho_{n,w}\rho_{y,z}\rho_{\hat{y},u}B$ in $\langle \hat{x}\rangle$.

\noindent
e) $\lambda_{x,y}$ which has bad subgroup $B$ which intersects
$\lambda_{x,z}\rho_{n,w}\rho_{y,z}\rho_{\hat{y},u}B=\langle
\hat{x},yz,\hat{y}u, z^{-1}x\hat{x}^{-1}x^{-1}x,
\lambda_{x,z}\rho_{n,w}\rho_{y,z}\rho_{\hat{y},u}\partial_x \rangle$
in $\<\hat{x}\>$.

\noindent f), g) and h) are similar to case 2.

Finally, \textbf{Case 2''} and \textbf{Case 2'''}: with $x$ hatted for
both $\rho$ and $\lambda$ are similar.

\subsection{Case 3}
$x$ general unhatted, $y$ unhatted two-sided and linked to
one-sided. Let $n$ be the one-sided letter and $u,z,w,v$ be general.

$B=\langle y,n\hat{y}y\hat{y}^{-1},
x^{-1}\hat{x}x, \partial_{x^{-1}}\rangle$

$\rho_{x,y}=\rho_{\hat{x},z}^{-1}\rho_{n,u}^{-1}\rho_{y,v}^{-1}\rho_{x,y}\rho_{y,v}\rho_{x,v}\rho_{n,u}\rho_{\hat{x},z}$

\noindent
\textit{Check:}

\begin{enumerate}[a)]
\item $\rho_{\hat{x},z}^{-1}$ has bad subgroup $\langle
  z,\hat{z},x, \partial_{x^{-1}}\rangle$ which intersects $B$ at most
  in $\langle x, \partial_{x^{-1}} \rangle$. In fact, by considering
  immersed graphs representing the subgroups, one can show that the
  intersection is $\langle \partial_{x^{-1}} \rangle$, but we do not
  need this fact.
	
\item $\lambda_{n,u}^{-1}$ has bad subgroup $\langle
  n\hat{y},y\hat{y}^{-1},u, \hat{u}, \partial_n\rangle$ which
  intersects $\rho_{\hat{x},z}B$ in $\langle
  n\hat{y}y\hat{y}^{-1}\rangle.$ Indeed, by applying
  Lemma~\ref{lem:drop-turn-easy} as above we may drop $\partial_n$,
  $\rho_{\hat{x},z}\partial_{x^{-1}}$, and $u,\hat{u},
  x^{-1}\hat{x}zx$ in sequence. So it suffices to consider the
  intersection of $\langle y, n\hat{y}y\hat{y}^{-1}\rangle$ and
  $\langle n\hat{y},y\hat{y}^{-1}\rangle$.  As both of these are free
  factors, the intersection is again a free factor. In particular,
  either the two factors are equal, or the intersection is of rank at
  most $1$. Since neither is contained in the other (e.g. by
  considering Abelianisations), the intersection is at most cyclic. As
  $\langle n\hat{y}y\hat{y}^{-1}\rangle$ is contained in both, the
  claim follows.
	
\item $\rho_{y,v}^{-1}$ has bad subgroup $\langle v, \hat{v},
  y^{-1}n^{-1}\hat{y}, y^{-1}\hat{y}y, \partial_{\hat{y}_+}\rangle$
  which intersects $\lambda_{n,u}\rho_{\hat{x},z}B=\langle
  y,un\hat{y}y\hat{y}^{-1}, x^{-1}\hat{x}zx,
  \lambda_{n,u}\rho_{\hat{x},z}\partial_{x^{-1}}\rangle$ trivally.
	
\item $\rho_{x,y}$ has bad subgroup $B$ which intersects
  $\rho_{y,v}\lambda_{n,u}\rho_{\hat{x},z}B=\langle
  yv,un\hat{y}yv\hat{y}^{-1},x^{-1}\hat{x}zx,
  \rho_{y,v}\lambda_{n,u}\rho_{\hat{x},z}\partial_{x^{-1}}\rangle$
  trivially.
	
\item $\rho_{y,v}$ has bad subgroup $\langle v, \hat{v},
  y^{-1}n^{-1}\hat{y}, y^{-1}\hat{y}y, \partial_{\hat{y}_+}\rangle$
  which intersects
  $\rho_{x,y}^{-1}\rho_{y,v}\lambda_{n,u}\rho_{\hat{x},z}B=\langle
  yv,un\hat{y}yv\hat{y}^{-1}, yx^{-1}\hat{x}zxy^{-1},
  \rho_{x,y}^{-1}\rho_{y,v}\lambda_{n,u}\rho_{\hat{x},z}\partial_{x^{-1}}\rangle
  $ trivially.
	
\item $\rho_{x,v}$ has bad subgroup $\langle v,\hat{v},
  x^{-1}\hat{x}x, \partial_{\hat{x}_+}\rangle$ which intersects
  $\rho_{y,v}^{-1}\rho_{x,y}^{-1}\rho_{y,v}\lambda_{n,u}\rho_{\hat{x},z}B=\langle
  y,un\hat{y}y\hat{y}^{-1}, yv^{-1}x^{-1}\hat{x}zxvy^{-1},
  \rho_{y,v}^{-1}\rho_{x,y}^{-1}\rho_{y,v}\lambda_{n,u}\rho_{\hat{x},z}\partial_{x^{-1}}\rangle$
  in the conjugacy class $\langle\hat{x}\rangle$.
	
\item $\lambda_{n,u}$ has bad subgroup $\langle
  n\hat{y},y\hat{y}^{-1},u, \hat{u}, \partial_n\rangle$ which
  intersects
  $\rho_{x,v}^{-1}\rho_{y,v}^{-1}\rho_{x,y}^{-1}\rho_{y,v}\lambda_{n,u}\rho_{\hat{x},z}B=\langle
  y,un\hat{y}y\hat{y}^{-1}, yx^{-1}\hat{x}zxy^{-1},
  \rho_{x,y}^{-1}\lambda_{n,u}\rho_{\hat{x},z}\partial_{x^{-1}}\rangle$
  trivially.
	
\item $\rho_{\hat{x},z}$ has bad subgroup $\langle
  z,\hat{z},x, \partial_{x^{-1}}\rangle$ which intersects
  $\lambda_{n,u}^{-1}\rho_{x,v}^{-1}\rho_{y,v}^{-1}\rho_{x,y}^{-1}\rho_{y,v}\lambda_{n,u}\rho_{\hat{x},z}B=\langle
  y,n\hat{y}y\hat{y}^{-1}, yx^{-1}\hat{x}zxy^{-1},
  \rho_{x,y}^{-1}\rho_{\hat{x},z}\partial_{x^{-1}}\rangle$ at most in
  $\langle x, \partial_{x^{-1}} \rangle$. In fact, by considering
  immersed graphs representing the subgroups, one can show that the
  intersection is trivial, but we do not need this fact.
	
\end{enumerate}
This completes the checks for the preliminary relation. We now collect
some variants on this case.  The case of $\lambda_{x,y}$ is analogous
using $\lambda_{x,y}=\rho_{\hat
  x,z}^{-1}\rho_{n,u}^{-1}\lambda_{x,y}\rho_{n,u}\rho_{\hat{x},z}.$
Indeed the bad subgroup is the same except $x^{-1}\hat{x}x$ is
replaced by $x$, and $\partial_{x^{-1}}$ by $\partial_{\hat{x}^{-1}}$.

\noindent \textbf{Case 3'} is the case of $x$ general, $y$ hatted
two-sided and linked to one-sided. The bad subgroup now is $B=\langle
\hat{y}, ny,x^{-1}\hat{x}x, \partial_{x^{-1}}\rangle$.

We use the relation
$\rho_{x,y}=\rho_{\hat{x},z}^{-1}\rho_{n,u}^{-1}\rho_{x,y}\rho_{n,u}\rho_{\hat{x},z}$
and the steps are the same except the overlap of the bad factor for
$\rho_{x,y}$ and $\rho_{n,u}\rho_{\hat{x},z}B$ is $\langle ny\rangle$.

\noindent
\textbf{Case 3''} is $x$ general hatted, $y$ unhatted two sided and
linked to one-sided. The bad subgroup is $B=\langle
y,n\hat{y}y\hat{y}^{-1}, \hat{x},\partial_{\hat{x}^{-1}}\rangle$. This
is similar.

\noindent Finally, \textbf{Case 3'''} is $x$ general hatted, $y$
hatted two-sided and linked to one-sided. The bad subgroup is
$B=\langle \hat{a},na,b,\partial_{\hat{x}^{-1}}\rangle.$ Again, this
is similar.

\subsection{Case 4} $x$ general unhatted, $y$ one-sided, $\rho_{x,y}$
and let $a,\hat{a}$ denote the letters that are linked with $y$.

The bad subgroup is $B= \langle
y,(a^{-1}y^{-1}\hat{a})a(a^{-1}y^{-1}\hat{a})^{-1},x^{-1}\hat{x}x
, \partial_{x^{-1}}\rangle$, and we use the relation
\[\rho_{x,y}=\rho_{y,u}^{-1}\rho_{\hat{x},z}^{-1}\rho_{x,y}\rho_{x,u}\rho_{\hat{x},z}\rho_{y,u}\]

\vspace{2mm}

\noindent
\textit{Check:}

\begin{enumerate}[a)]
\item $\rho_{y,u}^{-1}$ has bad subgroup $\langle
  \hat{a},a\hat{a}^{-1}y,u,\hat{u} , \partial_{\hat{a}}\rangle$ which
  intersects $B$ in $\langle \partial_{\hat{a}} \rangle$.

  Indeed, as in previous cases by Lemma~\ref{lem:drop-turn-hard} we
  may drop $u$, $\hat{u}$ and $x^{-1}\hat{x}x$. We now consider
  $\langle
  y,(a^{-1}y^{-1}\hat{a})a(a^{-1}y^{-1}\hat{a})^{-1}, \partial_{x^{-1}}\rangle
  $ and $\langle \hat{a},a\hat{a}^{-1}y,\partial_{\hat{a}}\rangle.$ By
  considering $y^{-1}\hat{a}aay$, a subword of
  $(a^{-1}y^{-1}\hat{a})a(a^{-1}y^{-1}\hat{a})^{-1}$ which can not
  occur in $\langle \hat{a},a\hat{a}^{-1}y,\partial_{\hat{a}}\rangle$
  we reduce to $\langle
  \hat{a},a\hat{a}^{-1}y,\partial_{x^{-1}}\rangle$ and $\langle
  y, \partial_{\hat{a}}\rangle.$ Similarly we may remove
  $a\hat{a}^{-1}y$ and then $y$ and $\hat{a}$.

\item $\rho_{\hat{x},z}^{-1}$ has bad subgroup $\langle z,\hat{z},x
  , \partial_{x^{-1}}\rangle$ which intersects $\rho_{y,u}B= \langle
  yu,(a^{-1}u^{-1}y^{-1}\hat{a})a(a^{-1}u^{-1}y^{-1}\hat{a})^{-1},x^{-1}\hat{x}x,
  \rho_{y,u}\partial_{x^{-1}}\rangle$ trivially.
	
\item $\rho_{x,y}$ has bad subgroup $B$ which intersects
  $\rho_{\hat{x},z}\rho_{y,u}B=\langle
  yu,(a^{-1}u^{-1}y^{-1}\hat{a})a(a^{-1}u^{-1}y^{-1}\hat{a})^{-1},x^{-1}\hat{x}zx,
  \rho_{\hat{x},z}\rho_{y,u}\partial_{x^{-1}}\rangle$ trivially.
	
\item $\rho_{x,u}$ has bad subgroup $\langle u,\hat{u},x^{-1}\hat{x}x
  , \partial_{\hat{x}_+} \rangle$ which intersects
  $\rho_{x,y}^{-1}\rho_{\hat{x},z}\rho_{y,u}B = \langle
  yu,(a^{-1}u^{-1}y^{-1}\hat{a})a(a^{-1}u^{-1}y^{-1}\hat{a})^{-1},yx^{-1}\hat{x}zxy^{-1},
  \rho_{x,y}^{-1}\rho_{\hat{x},z}\rho_{y,u}\partial_{x^{-1}}\rangle$
  trivially.
	
\item $\rho_{\hat{x},z}$ has bad subgroup $\langle z,\hat{z},x
  , \partial_{x^{-1}} \rangle$ which intersects
  $\rho_{x,u}^{-1}\rho_{x,y}^{-1}\rho_{\hat{x},z}\rho_{y,u}B = \langle
  yu,(a^{-1}u^{-1}y^{-1}\hat{a})a(a^{-1}u^{-1}y^{-1}\hat{a})^{-1},yux^{-1}\hat{x}zxu^{-1}y^{-1},
  \rho_{x,u}^{-1}\rho_{x,y}^{-1}\rho_{\hat{x},z}\rho_{y,u}\partial_{x^{-1}}\rangle$
  trivially.
	
\item $\rho_{y,u}$ has bad subgroup $\langle u, \hat{u}, a
  \hat{a}^{-1}y, \hat{a} , \partial_{\hat{a}} \rangle$ which
  intersects
  $\rho_{\hat{x},z}^{-1}\rho_{x,u}^{-1}\rho_{x,y}^{-1}\rho_{\hat{x},z}\rho_{y,u}B
  = \langle
  yu,(a^{-1}u^{-1}y^{-1}\hat{a})a(a^{-1}u^{-1}y^{-1}\hat{a})^{-1},yux^{-1}\hat{x}xu^{-1}y^{-1},
  \rho_{\hat{x},z}^{-1}\rho_{x,u}^{-1}\rho_{x,y}^{-1}\rho_{\hat{x},z}\rho_{y,u}\partial_{x^{-1}}\rangle$
  trivially.

  Namely, as before, we can drop the (modified) boundary words, as
  well as $yux^{-1}\hat{x}xu^{-1}y^{-1}, \hat{a}$. We now need to
  control the intersection of $\langle u, \hat{u}, a \hat{a}^{-1}y
  \rangle$ and $\langle
  yu,(a^{-1}u^{-1}y^{-1}\hat{a})a(a^{-1}u^{-1}y^{-1}\hat{a})^{-1}\rangle$. Since
  both are free factors, and their Abelianisations do not intersect,
  the claim follows.
\end{enumerate}

The case of $x$ general hatted and the relevant $\lambda$ cases are similar. 

\subsection{Case 5} $x$ one sided, $y$ general. As before, we denote
by $a,\hat{a}$ the linked two-sided letters.  Here, we consider
$\rho_{x,y}$ which has bad subgroup
\[B=\langle y, \hat{y},\hat{a}, a\hat{a}^{-1}x , \partial_{\hat{a}}
\rangle.\] We use the relation
\[\rho_{x,y}=\rho_{\hat{y},u}^{-1}\rho_{y,z}^{-1}\lambda_{\hat{a},w}^{-1}\rho_{x,y}\rho_{x,z}\lambda_{\hat{a},w}\rho_{y,z}\rho_{\hat{y},u}.\]

\begin{enumerate}[a)]
\item $\rho_{\hat{y},u}^{-1}$ has bad subgroup $\langle u, \hat{u},y
  , \partial_{y^{-1}} \rangle$. This intersects $B$ in $\langle
  y,\partial_{y^{-1}}\rangle$.
\item $\rho_{y,z}^{-1}$ has bad subgroup $\langle
  z,\hat{z},y^{-1}\hat{y}y , \partial_{\hat{y}_+} \rangle$ and this
  intersects $\rho_{\hat{y},u}B=\langle y, \hat{y}u,\hat{a},
  a\hat{a}^{-1}x, \rho_{\hat{y},u}\partial_{\hat{a}} \rangle$
  trivially.
	
\item $\lambda_{\hat{a},w}^{-1}$ has bad subgroup $\langle
  \hat{a}a\hat{a}^{-1},x,w,\hat{w} , \partial_n \rangle$ and this
  intersects $\rho_{y,z}\rho_{\hat{y},u}B=\langle yz,
  \hat{y}u,\hat{a}, a\hat{a}^{-1}x,
  \rho_{y,z}\rho_{\hat{y},u}\partial_{\hat{a}} \rangle$ in $\langle
  \hat{a}a\hat{a}^{-1}x\rangle$. Namely, after dropping the boundary
  terms as usual, we can also drop $yz, \hat{y}u,w, \hat{w}$. The
  resulting rank $2$ free factors $\langle \hat{a}a\hat{a}^{-1},x
  \rangle$ and $\langle \hat{a}, a\hat{a}^{-1}x \rangle$ have
  Abelianisations that intersect in a rank $1$ submodule.  The
  intersection is therefore at most a rank $1$ free factor, hence it
  is the one claimed.
		
\item $\rho_{x,y}$ has bad subgroup $B$ which intersects
  $\lambda_{\hat{a},w}\rho_{y,z}\rho_{\hat{y},u}B=\langle yz,
  \hat{y}u,w\hat{a}, a\hat{a}^{-1}w^{-1}x,
  \lambda_{\hat{a},w}\rho_{y,z}\rho_{\hat{y},u}\partial_{\hat{a}}
  \rangle$ trivially.
		
\item $\rho_{x,z}$ has bad subgroup $\langle
  \hat{a},a\hat{a}^{-1}x,z,\hat{z}, \partial_{\hat{a}} \rangle$ which
  intersects
  $\rho_{x,y}^{-1}\lambda_{\hat{a},w}\rho_{y,z}\rho_{\hat{y},u}B=
  \langle yz, \hat{y}u,w\hat{a}, a\hat{a}^{-1}w^{-1}xy^{-1},
  \rho_{x,y}^{-1}\lambda_{\hat{a},w}\rho_{y,z}\rho_{\hat{y},u}\partial_{\hat{a}}
  \rangle$ trivially.

\item $\lambda_{\hat{a},w}$ has bad subgroup $\langle
  \hat{a}a\hat{a}^{-1},x,w,\hat{w}, \partial_n\rangle$ which we need
  to intersect with
  $\rho_{x,z}^{-1}\rho_{x,y}^{-1}\lambda_{\hat{a},w}\rho_{y,z}\rho_{\hat{y},u}B=
  \langle yz, \hat{y}u,w\hat{a}, a\hat{a}^{-1}w^{-1}xz^{-1}y^{-1},
  \rho_{x,z}^{-1}\rho_{x,y}^{-1}\lambda_{\hat{a},w}\rho_{y,z}\rho_{\hat{y},u}\partial_{\hat{a}}
  \rangle$.  As usual, we can discard the boundary word terms, and
  clean up generators to compare $\langle
  \hat{a}a\hat{a}^{-1},x,w,\hat{w}\rangle$ and $\langle yz,
  \hat{y}u,w\hat{a}, a\hat{a}^{-1}w^{-1}x \rangle$. We can drop $yz,
  \hat{y}u$ from the latter, replacing it with $\langle w\hat{a},
  a\hat{a}^{-1}w^{-1}x \rangle$. Since the $w\hat{a}$ is not
  homologous into the former factor, the intersection is at most rank
  $1$. Thus, the intersection is $\langle
  w\hat{a}a\hat{a}^{-1}w^{-1}x\rangle$.

\item $\rho_{y,z}$ has bad subgroup $\langle z, \hat{z},
  y^{-1}\hat{y}y, \partial_{\hat{y}_+} \rangle$ which intersects
  $\lambda_{\hat{a},w}^{-1}\rho_{x,z}^{-1}\rho_{x,y}^{-1}\lambda_{\hat{a},w}\rho_{y,z}\rho_{\hat{y},u}B=
  \langle yz, \hat{y}u,\hat{a}, a\hat{a}^{-1}xz^{-1}y^{-1},
  \lambda_{\hat{a},w}^{-1}\rho_{x,z}^{-1}\rho_{x,y}^{-1}\lambda_{\hat{a},w}\rho_{y,z}\rho_{\hat{y},u}\partial_{\hat{a}}
  \rangle$ trivially.
	
\item $\rho_{\hat{y},u}$ has bad subgroup $\langle
  u,\hat{u},y, \partial_{y^{-1}} \rangle$ which intersects
  $\rho_{y,z}^{-1}\lambda_{\hat{a},w}^{-1}\rho_{x,z}^{-1}\rho_{x,y}^{-1}\lambda_{\hat{a},w}\rho_{y,z}\rho_{\hat{y},u}B=
  \langle y, \hat{y}u,\hat{a}, a\hat{a}^{-1}xy^{-1},
  \rho_{y,z}^{-1}\lambda_{\hat{a},w}^{-1}\rho_{x,z}^{-1}\rho_{x,y}^{-1}\lambda_{\hat{a},w}\rho_{y,z}\rho_{\hat{y},u}\partial_{\hat{a}}
  \rangle$ in $\langle y\rangle$.
\end{enumerate}	

\noindent
\textbf{Case 5'} This is the analogous left-multiplication move
$\lambda_{x,y}$ with notation as in Case 5, and thus the bad subgroup
is \[B=\langle x\hat{a},a\hat{a}^{-1},y,\hat{y}, \partial_n \rangle\]
where $a,\hat{a}$ is linked to $x$. Let $z,w,u$ be general. We use the
relation

$\lambda_{x,y}=\rho_{\hat{y},u}^{-1}\rho_{y,z}^{-1}\rho_{\hat{a},w}^{-1}\lambda_{x,z}\lambda_{x,y}\rho_{\hat{a},w}\rho_{y,z}\rho_{\hat{y},u}.$

The checks here are similar to Case 5.
Indeed,
\textit{Check:} 

\begin{enumerate}[a)]
\item $\rho_{\hat{y},u}^{-1}$ has bad subgroup $\langle u, \hat{u},y
  , \partial_{y^{-1}} \rangle$. This intersects $B$ in $\langle
  y, \partial_{y^{-1}} \rangle$.
	
\item $\rho_{y,z}^{-1}$ has bad subgroup $\langle
  z,\hat{z},y^{-1}\hat{y}y , \partial_{\hat{y}_+} \rangle$ and this
  intersects $\rho_{\hat{y},u}B=\langle
  x\hat{a},a\hat{a}^{-1},y,\hat{y}u, \rho_{\hat{y},u}\partial_n
  \rangle$ trivially.
	
\item $\rho_{\hat{a},w}^{-1}$ has bad subgroup $\langle w,
  \hat{w},a,\hat{a}^{-1}x\hat{a}, \partial_a \rangle$ and this
  intersects $\rho_{y,z}\rho_{\hat{y},u}B=\langle
  x\hat{a},a\hat{a}^{-1},yz,\hat{y}u,
  \rho_{y,z}\rho_{\hat{y},u}\partial_n \rangle$ in $\langle
  a\hat{a}^{-1}x\hat{a}\rangle$.
	
\item $\lambda_{x,z}$ has bad subgroup $\langle
  x\hat{a},a\hat{a}^{-1},z,\hat{z}, \partial_n \rangle$ which
  intersect $\rho_{\hat{a},w}\rho_{y,z}\rho_{\hat{y},u}B=\langle
  x\hat{a}w,aw^{-1}\hat{a}^{-1},yz,\hat{y}u,
  \rho_{\hat{a},w}\rho_{y,z}\rho_{\hat{y},u}\partial_n \rangle$
  trivially.
	
\item $\lambda_{x,y}$ has bad subgroup $B$ which intersects
  $\lambda_{x,z}^{-1}\rho_{\hat{a},w}\rho_{y,z}\rho_{\hat{y},u}B=\langle
  z^{-1}x\hat{a}w,aw^{-1}\hat{a}^{-1},yz,\hat{y}u,
  \lambda_{x,z}^{-1}\rho_{\hat{a},w}\rho_{y,z}\rho_{\hat{y},u}\partial_n
  \rangle$ trivially.

\item $\rho_{\hat{a},w}$ has bad subgroup $\langle w,
  \hat{w},a,\hat{a}^{-1}x\hat{a} , \partial_a \rangle$ which
  intersects
  $\lambda_{x,y}^{-1}\lambda_{x,z}^{-1}\rho_{\hat{a},w}\rho_{y,z}\rho_{\hat{y},u}B
  = \langle z^{-1}y^{-1}x\hat{a}w,aw^{-1}\hat{a}^{-1},yz,\hat{y}u,
  \lambda_{x,y}^{-1}\lambda_{x,z}^{-1}\rho_{\hat{a},w}\rho_{y,z}\rho_{\hat{y},u}\partial_n\rangle$
  in $\langle aw^{-1}\hat{a}^{-1}x\hat{a}w\rangle$.

\item $\rho_{y,z}$ has bad subgroup $\langle z, \hat{z},
  y^{-1}\hat{y}y , \partial_{\hat{y}_+} \rangle$ which intersects
  $\rho_{\hat{a},w}^{-1}\lambda_{x,y}^{-1}\lambda_{x,z}^{-1}\rho_{\hat{a},w}\rho_{y,z}\rho_{\hat{y},u}B
  = \langle z^{-1}y^{-1}x\hat{a},a\hat{a}^{-1},yz,\hat{y}u,
  \rho_{\hat{a},w}^{-1}\lambda_{x,y}^{-1}\lambda_{x,z}^{-1}\rho_{\hat{a},w}\rho_{y,z}\rho_{\hat{y},u}\partial_n
  \rangle$ trivially.
	
\item $\rho_{\hat{y},u}$ has bad subgroup $\langle u,\hat{u},y
  , \partial_{y^{-1}} \rangle$ which intersects
  $\rho_{y,z}^{-1}\rho_{\hat{a},w}^{-1}\lambda_{x,y}^{-1}\lambda_{x,z}^{-1}\rho_{\hat{a},w}\rho_{y,z}\rho_{\hat{y},u}B=\langle
  y^{-1}x\hat{a},a\hat{a}^{-1},y,\hat{y}u,
  \rho_{y,z}^{-1}\rho_{\hat{a},w}^{-1}\lambda_{x,y}^{-1}\lambda_{x,z}^{-1}\rho_{\hat{a},w}\rho_{y,z}\rho_{\hat{y},u}\partial_n
  \rangle$ in $\langle y \rangle$.
\end{enumerate}

\section{Minimal foliations for nonorientable surfaces}
The purpose of this appendix is to show the following result, which
was stated as Theorem~\ref{thm:weak-nonorientable-connectivity} above.
\begin{thm}
  Suppose that $\Sigma$ is a nonorientable surface with a single
  boundary component or marked point. Then, there is a path-connected
  subset
  \[ \mathcal{P} \subset \mathcal{M}(\Sigma) \subset \pml(\Sigma) \]
  consisting of minimal measured foliations, which is invariant under
  the mapping class group of $\Sigma$. In addition, if $F$ is any
  finite set of laminations, the set $\mathcal{P}\setminus F$ is still
  path-connected.
\end{thm}
The proof uses methods established in \cite{LS}. We consider
throughout the case of a surface $\Sigma = (S,p)$ with a marked point;
the other claim is equivalent.

We begin by observing that any foliation on $S$ defines a foliation on
$\Sigma$, and the resulting foliations of $\Sigma$ are exactly those
which do not have an angle-$\pi$ singularity at $p$.
\begin{lem}\label{lem:drop-marked-point}
  A foliation $\mathcal{F}$ on $S$ is minimal (as a foliation on $S$)
  if and only if it is minimal as a foliation of $\Sigma$.
\end{lem}
\begin{proof}
  This follows, since any essential simple closed curve on $\Sigma$
  defines an essential simple closed curve on $S$ (i.e. after
  forgetting the marked point).
\end{proof}

\begin{defin}
  We define $\mathcal{P} \subset \pml(\Sigma)$ to be the set of
  minimal foliations which either
  \begin{enumerate}
  \item do not have an angle--$\pi$ singularity at $p$, or
  \item are stable foliations of point-pushing pseudo-Anosovs.
  \end{enumerate}
\end{defin}
It is clear from construction that $\mathcal{P}$ is invariant under
the mapping class group of $\Sigma$. We aim to show that any foliation
in $\mathcal{P}$ of the first type can be connected by a path to any
foliation of the second type, which will prove
Theorem~\ref{thm:weak-nonorientable-connectivity}, as we have full
flexibility which point-pushes to use.

To do so, we need to recall some facts about point-pushing maps;
compare \cite{LS, CH}. Let $\gamma:[0,1]\to S$ be an immersed smooth
loop based at $p$. We let
\[ D_\gamma:[0,1] \to \mathrm{Diff}(S) \] be a smooth isotopy
starting in the identity, so that $D_\gamma(t)(p) = \gamma(t)$.  By
definition, the endpoint $D_\gamma(1)$ is then a representative of the
point-pushing mapping class $\Psi_\gamma$ defined by $\gamma$.

Suppose that $\mathcal{F}$ be a foliation of $S$ which is
minimal. Then, the same is true for $D_\gamma(t)\mathcal{F}$ (as they
are indeed isotopic).  When seen as minimal foliations of $\Sigma$,
the assignment
\[ t \mapsto D_\gamma(t)\mathcal{F} \] is a continuous path of minimal
foliations joining $\mathcal{F}$ to $\Psi_\gamma\mathcal{F}$:
minimality follows by Lemma~\ref{lem:drop-marked-point}, and
continuity since $D_\gamma$ is smooth and intersections with
$\mathcal{F}$ vary continuously with the curve.

Now, we use the following:
\begin{lem}
  Let $\Psi$ be a point-pushing pseudo-Anosov of $\Sigma$. Then $\Psi$
  acts on $\pml(\Sigma)$ with north-south dynamics, and both fixpoints
  have an angle-$\pi$--singularity.
\end{lem}
\begin{proof}
  Let $X \to \Sigma$ be the orientation double cover. Then $\Psi$
  lifts to a pseudo-Anosov of $X$, and the first claim follows. The
  second claim follows since point-pushes have angle--$\pi$
  singularitites at the marked point \cite{LS}.
\end{proof}

By the lemma, the path $\{ D_\gamma(t)\mathcal{F}, t \in [0,1] \}$ is
disjoint from the repelling fixed point of $\Psi$, and thus
\[ \bigcup_{n\in \N}\Psi^n \{  D_\gamma(t)\mathcal{F}, t \in [0,1] \} \]
is the desired path joining $\mathcal{F}$ to the stable foliation of $\Psi$.

\bibliographystyle{alpha}
\bibliography{blas}

\begin{thebibliography}{{Gui}00b}

\bibitem[BF]{bf:outerlimits}
Mladen Bestvina and Mark Feighn.
\newblock Outer limits.
\newblock {\em available at
  http://andromeda.rutgers.edu/feighn/papers/outer.pdf}.

\bibitem[BF14]{BesFei}
Mladen Bestvina and Mark Feighn.
\newblock Hyperbolicity of the complex of free factors.
\newblock {\em Adv. Math.}, 256:104--155, 2014.

\bibitem[BR15]{BesRey}
Mladen Bestvina and Patrick Reynolds.
\newblock The boundary of the complex of free factors.
\newblock {\em Duke Math. J.}, 164(11):2213--2251, 2015.

\bibitem[CH]{CH}
Jon Chaika and Sebastian Hensel.
\newblock Path-connectivity of the set of uniquely ergodic and cobounded
  foliations.
\newblock available on the authors homepage.

\bibitem[CHL07]{CHL0}
Thierry Coulbois, Arnaud Hilion, and Martin Lustig.
\newblock Non-unique ergodicity, observers' topology and the dual algebraic
  lamination for {$\Bbb R$}-trees.
\newblock {\em Illinois J. Math.}, 51(3):897--911, 2007.

\bibitem[CHL08a]{CHL1}
Thierry Coulbois, Arnaud Hilion, and Martin Lustig.
\newblock {$\Bbb R$}-trees and laminations for free groups. {I}. {A}lgebraic
  laminations.
\newblock {\em J. Lond. Math. Soc. (2)}, 78(3):723--736, 2008.

\bibitem[CHL08b]{CHL2}
Thierry Coulbois, Arnaud Hilion, and Martin Lustig.
\newblock {$\Bbb R$}-trees and laminations for free groups. {II}. {T}he dual
  lamination of an {$\Bbb R$}-tree.
\newblock {\em J. Lond. Math. Soc. (2)}, 78(3):737--754, 2008.

\bibitem[CHL08c]{CHL3}
Thierry Coulbois, Arnaud Hilion, and Martin Lustig.
\newblock {$\Bbb R$}-trees and laminations for free groups. {III}. {C}urrents
  and dual {$\Bbb R$}-tree metrics.
\newblock {\em J. Lond. Math. Soc. (2)}, 78(3):755--766, 2008.

\bibitem[CV86]{CullerVogtmann}
Marc {Culler} and Karen {Vogtmann}.
\newblock {Moduli of graphs and automorphisms of free groups}.
\newblock {\em {Invent. Math.}}, 84:91--119, 1986.

\bibitem[FLP12]{FLP}
Albert Fathi, Fran\c{c}ois Laudenbach, and Valentin Po\'{e}naru.
\newblock {\em Thurston's work on surfaces}, volume~48 of {\em Mathematical
  Notes}.
\newblock Princeton University Press, Princeton, NJ, 2012.
\newblock Translated from the 1979 French original by Djun M. Kim and Dan
  Margalit.

\bibitem[Gab09]{GabAF}
David Gabai.
\newblock Almost filling laminations and the connectivity of ending lamination
  space.
\newblock {\em Geom. Topol.}, 13(2):1017--1041, 2009.

\bibitem[Gab14]{GabTop}
David Gabai.
\newblock On the topology of ending lamination space.
\newblock {\em Geom. Topol.}, 18(5):2683--2745, 2014.

\bibitem[Gui00a]{guirardel:thesis}
Vincent Guirardel.
\newblock Dynamics of {${\rm Out}(F_n)$} on the boundary of outer space.
\newblock {\em Ann. Sci. \'{E}cole Norm. Sup. (4)}, 33(4):433--465, 2000.

\bibitem[{Gui}00b]{Guirardel}
Vincent {Guirardel}.
\newblock {Dynamics of \(\text{Out}(F_n)\) on the boundary of outer space}.
\newblock {\em {Ann. Sci. \'Ec. Norm. Sup\'er. (4)}}, 33(4):433--465, 2000.

\bibitem[Ham]{Ham}
Ursula Hamenst\"adt.
\newblock The boundary of the free factor graph and the free splitting graph.
\newblock {\em arXiv:1211.1630}.

\bibitem[{Hor}17]{horbez:boundaryouterspace}
Camille {Horbez}.
\newblock {The boundary of the outer space of a free product}.
\newblock {\em {Isr. J. Math.}}, 221(1):179--234, 2017.

\bibitem[Kap06]{Ilya}
Ilya Kapovich.
\newblock Currents on free groups.
\newblock In {\em Topological and asymptotic aspects of group theory}, volume
  394 of {\em Contemp. Math.}, pages 149--176. Amer. Math. Soc., Providence,
  RI, 2006.

\bibitem[KL09]{Kapovich-Lustig-pairing}
Ilya Kapovich and Martin Lustig.
\newblock Geometric intersection number and analogues of the curve complex for
  free groups.
\newblock {\em Geom. Topol.}, 13(3):1805--1833, 2009.

\bibitem[KL10]{Kapovich-Lustig-GAFA}
Ilya Kapovich and Martin Lustig.
\newblock Intersection form, laminations and currents on free groups.
\newblock {\em Geom. Funct. Anal.}, 19(5):1426--1467, 2010.

\bibitem[Kla]{Kla}
Erica Klarreich.
\newblock The boundary at infinity of the curve complex and the relative
  teichm\"uller space.
\newblock to appear in Groups, Geometry, and Dynamics. arXiv:1803.10339.

\bibitem[KR14]{KapovichRafi}
Ilya {Kapovich} and Kasra {Rafi}.
\newblock {On hyperbolicity of free splitting and free factor complexes}.
\newblock {\em {Groups Geom. Dyn.}}, 8(2):391--414, 2014.

\bibitem[LL03]{Levitt-Lustig}
Gilbert Levitt and Martin Lustig.
\newblock Irreducible automorphisms of {$F_n$} have north-south dynamics on
  compactified outer space.
\newblock {\em J. Inst. Math. Jussieu}, 2(1):59--72, 2003.

\bibitem[LS09]{LS}
Christopher~J. Leininger and Saul Schleimer.
\newblock Connectivity of the space of ending laminations.
\newblock {\em Duke Math. J.}, 150(3):533--575, 2009.

\bibitem[Mar95]{Reiner}
Reiner Martin.
\newblock {\em Non-uniquely ergodic foliations of thin-type, measured currents
  and automorphisms of free groups}.
\newblock ProQuest LLC, Ann Arbor, MI, 1995.
\newblock Thesis (Ph.D.)--University of California, Los Angeles.

\bibitem[MM99]{MM1}
Howard~A. Masur and Yair~N. Minsky.
\newblock Geometry of the complex of curves. {I}. {H}yperbolicity.
\newblock {\em Invent. Math.}, 138(1):103--149, 1999.

\bibitem[Nie24]{Nielsen24}
Jakob Nielsen.
\newblock Die {I}somorphismengruppe der freien {G}ruppen.
\newblock {\em Math. Ann.}, 91(3-4):169--209, 1924.

\bibitem[Rey]{reynolds:reducingSystems}
Patrick Reynolds.
\newblock Reducing systems for very small trees.
\newblock arXiv:1211.3378.

\bibitem[Ser80]{SerreTrees}
Jean-Pierre Serre.
\newblock {\em Trees}.
\newblock Springer-Verlag, Berlin-New York, 1980.
\newblock Translated from the French by John Stillwell.

\bibitem[Sko96]{skora}
Richard~K. Skora.
\newblock Splittings of surfaces.
\newblock {\em J. Amer. Math. Soc.}, 9(2):605--616, 1996.

\bibitem[Sta83]{Stallings}
John~R. Stallings.
\newblock Topology of finite graphs.
\newblock {\em Invent. Math.}, 71(3):551--565, 1983.

\bibitem[Vog08]{Vog08}
Karen Vogtmann.
\newblock What is{$\dots$}outer space?
\newblock {\em Notices Amer. Math. Soc.}, 55(7):784--786, 2008.

\bibitem[Whi36]{Whitehead}
J.~H.~C. Whitehead.
\newblock On equivalent sets of elements in a free group.
\newblock {\em Ann. of Math. (2)}, 37(4):782--800, 1936.

\end{thebibliography}

\end{document}